\documentclass[11pt,oneside,a4paper]{article}
\usepackage[top=1 in, bottom=1 in, left=0.85 in, right=0.85 in]{geometry}

\PassOptionsToPackage{numbers, compress, sort}{natbib}

\usepackage[utf8]{inputenc} 
\usepackage[T1]{fontenc}    
\usepackage{hyperref}       
\usepackage{url}            
\usepackage{booktabs}       
\usepackage{amsfonts}       
\usepackage{nicefrac}       
\usepackage{microtype}      
\usepackage{xcolor}         

\usepackage{microtype}
\usepackage{graphicx}
\usepackage{booktabs} 
\usepackage{subcaption}
\usepackage{pifont}   
\usepackage{enumerate} 
\usepackage{thm-restate}
\usepackage{wrapfig}
\usepackage{enumitem}

\usepackage{algorithm}
\usepackage{algpseudocode}

\usepackage{array}

\usepackage{etoc}

\usepackage{natbib}

\usepackage{mathtools}
\usepackage{amsmath}
\usepackage{amssymb}
\usepackage{amsthm}

\usepackage[capitalize,noabbrev]{cleveref}

\DeclareMathOperator*{\argmax}{arg\,max}

\newcommand{\pp}{\mathbf{P}}

\newcommand{\bE}{\mathbf{E}}

\newcommand{\bb}{\mathbf{b}}

\newcommand{\bx}{\mathbf{x}}

\newcommand{\calS}{\mathcal{S}}

\newcommand{\calB}{\mathcal{B}}

\newcommand{\bZ}{\mathbf{Z}}
\newcommand{\bY}{\mathbf{Y}}

\newcommand{\by}{\mathbf{y}}

\newcommand{\bv}{\mathbf{v}}

\providecommand{\norm}[1]{\| #1\|}

\newcommand{\x}{\mathbf{x}}

\newcommand{\p}{\mathbb{P}}

\newcommand{\bbN}{\mathbb{N}}

\newcommand{\rel}[1][N]{V^{}_{\mathrm{rel}}(\bm(0),T)}

\newcommand\bA{\mathbf{A}}
\newcommand\ba{\mathbf{a}}
\newcommand\bs{\mathbf{s}}

\newcommand\bzero{\mathbf{0}}

\newcommand\calO{\mathcal{O}}
\newcommand\calE{\mathcal{E}}

\newcommand\calA{\mathcal{A}}
\newcommand\calY{\mathcal{Y}}

\newcommand\calU{\mathcal{U}}

\newcommand\be{\mathbf{e}}

\newcommand\bd{\mathbf{d}}

\newcommand{\expect}[1]{\mathbb{E}\left[#1\right]}
\newcommand{\var}[1]{\mathrm{var}\left[#1\right]}
\newcommand{\proba}[1]{\mathbb{P}\left(#1\right)}

\newcommand{\abs}[1]{\left|#1\right|}

\newcommand{\norme}[1]{\left\| #1 \right\|}

\newcommand{\norminf}[1]{\left\| #1 \right\|_\infty}

\newcommand\bXN{\mathbf{X}^{N}}

\newcommand\bX{\mathbf{X}}

\newcommand\bN{\mathbf{N}}

\newcommand\br{\mathbf{r}}

\newcommand\calN{\mathcal{N}}

\newcommand\rmax{r_{\mathrm{max}}}

\newcommand\calD{\mathcal{D}}
\newcommand\calV{\mathcal{V}}

\newcommand\calX{\mathcal{X}}

\newcommand\piNopt{\pi^{N}_{\mathrm{opt}}}

\newcommand\piN{\pi^{N}}

\newcommand\multi{\mathrm{multi}}
\newcommand\cvd{\xrightarrow[N \rightarrow \infty]{d}}

\newcommand\simd{\stackrel{d}{\sim}}

\newcommand\bc{\mathbf{c}}

\newcommand\bSigma{\mathbf{\Sigma}}

\newcommand\Ndelta{\delta_N}

\newcommand\OlogN{\calO\left(\frac{1}{N^{\log N}}\right)}

\newcommand{\RR}{\mathbb{R}}

\newcommand\tON{\tilde{\mathcal{O}} \left(\frac{1}{N}\right)}

\newcommand{\proj}[2]{\mathrm{Proj}_{#1}\left(#2\right)}
\newcommand{\pidelta}{\Pi_{\Ndelta}(\by^*)}
\newcommand{\pisub}{\Pi_{\text{fluid}}(\by^*)}
\newcommand{\Npic}{{\Pi}^c_{\Ndelta}(\by^*)}
\newcommand{\xini}{\bx_{\mathrm{ini}}}
\newcommand{\pic}{{\Pi}^c_{\tilde{D}}(\by^*)}

\newcommand{\pip}{\pi_{\text{pre}}}

\newcommand{\Ex}[1]{\mathbb{E}\left[#1\right]}
\newcommand{\Var}[1]{\mathrm{Var}\left(#1\right)}
\newcommand{\Cov}[1]{\mathrm{Cov}\left(#1\right)}
\newcommand{\wass}[2]{d_W^{(1)}\left(#1,#2\right)}
\newcommand{\simpXN}{\Delta^S_N}
\newcommand{\simpYN}{\Delta^{2S}_N}

\newcommand{\VarMulti}{U}
\newcommand{\bVarMulti}{\mathbf{U}}

\newcommand{\Xsys}{\bX^N}
\newcommand{\Xdiff}{\tilde{\bX}}

\newcommand{\VNopt}{V^N_{\mathrm{opt}}}
\newcommand{\VNLPres}{V^N_{{\mathrm{fluid}}}}
\newcommand{\VNdiffopt}{\tilde{V}^N_{\mathrm{opt}}}

\newcommand{\VLP}{\overline{V}_{\mathrm{LP}}}
\newcommand{\Yopt}{\bY^{N,*}}
\newcommand{\Ydiff}{\tilde{\bY}}
\newcommand{\Ysys}{\bY^N}

\newcommand{\round}{\mathrm{round}}

\newcommand{\YLP}{\overline{\mathbf{Y}}}
\newcommand{\QLP}{\overline{Q}_{\mathrm{LP}}}

\NewDocumentCommand{\VL}{o}{%
    \IfValueTF{#1}{%
        \texttt{V}^{[#1]}%
    }{%
        \texttt{V}^{[\mathbf{L}_1]}%
    }
}

\theoremstyle{plain}

\newtheorem{lemma}{Lemma}[section]

\newtheorem{claim}{Claim}[section]
\newtheorem{definition}{Definition}[section]
\newtheorem{assumption}{Assumption}[section]

\usepackage[textsize=tiny]{todonotes}

\allowdisplaybreaks

\title{Achieving $\tilde{\mathcal{O}}(1/N)$ Optimality Gap in Restless Bandits through Gaussian Approximation}

\author{
    Chen Yan\thanks{Electrical Engineering and Computer Science Department, University of Michigan, Ann Arbor. Email: \texttt{chenyaa@umich.edu}, \texttt{leiying@umich.edu}} \and
    Weina Wang\thanks{Computer Science Department, Carnegie Mellon University. Email: \texttt{weinaw@cs.cmu.edu}} \and
    Lei Ying\footnotemark[1]
}

\date{}

\begin{document}

\maketitle

\begin{abstract}
We study the finite-horizon Restless Multi-Armed Bandit (RMAB) problem with $N$ homogeneous arms.
Prior work has shown that when an RMAB satisfies a non-degeneracy condition, Linear-Programming-based (LP-based) policies derived from the fluid approximation, which captures the mean dynamics of the system, achieve an exponentially small optimality gap.
However, it is common for RMABs to be degenerate, in which case LP-based policies can result in a $\Theta(1/\sqrt{N})$
\footnote{We adopt standard asymptotic notation throughout this paper. Specifically, for functions $f(N)$ and $g(N)$, we write $f(N) = \mathcal{O}(g(N))$ if there exist positive constants $C$ and $N_0$ such that $|f(N)| \leq C|g(N)|$ for all $N \geq N_0$. Similarly, we write $f(N) = \Omega(g(N))$ if $g(N) = \mathcal{O}(f(N))$, and $f(N) = \Theta(g(N))$ if both $f(N) = \mathcal{O}(g(N))$ and $f(N) = \Omega(g(N))$ hold simultaneously. Additionally, we use $\tilde{\mathcal{O}}(\cdot)$ and $\tilde{\Theta}(\cdot)$ notation to indicate that logarithmic factors are omitted.}
optimality gap per arm.
In this paper, we propose a novel Stochastic-Programming-based (SP-based) policy that, under a uniqueness assumption, achieves an $\tilde{\mathcal{O}}(1/N)$ optimality gap for degenerate RMABs.
Our approach is based on the construction of a Gaussian stochastic system that captures not only the mean but also the variance of the RMAB dynamics, resulting in a more accurate approximation than the fluid approximation.
We then solve a stochastic program for this system to obtain our policy.
This is the first result to establish an $\tilde{\mathcal{O}}(1/N)$ optimality gap for degenerate RMABs.
\end{abstract}

\tableofcontents

\section{Introduction}

The Restless Multi-Armed Bandit (RMAB) problem is an important framework in sequential decision-making, where a decision maker selects a subset of tasks (arms) to work on (pull) at each time step to maximize cumulative rewards, under known model parameters \citep{whittle-restless}. Unlike the classical (restful) bandit \citep{Gittins79banditprocesses}, in the restless variant, the state of each arm evolves stochastically regardless of whether it is pulled. RMABs have been widely applied in domains such as machine maintenance \citep{glazebrook2005index, demirci2024restless}, healthcare resource allocation \citep{mate2020collapsing,mate2021risk}, and target tracking \citep{le2006multi, la2006optimal}, to name a few, where optimal decision-making under uncertainty is critical. 
A general RMAB is PSPACE-hard \citep{Papadimitriou99thecomplexity}, and finding optimal policies is computationally challenging, especially as the number of arms grows.
Recently, there have also been efforts to use deep learning and reinforcement learning to learn heuristic policies for RMABs, such as \citep{nakhleh2021neurwin, xiong2022reinforcement, killian2022restless, xiong2023finite, avrachenkov2024lagrangian}.

In this paper, we focus on the finite-horizon version of the RMAB problem with \( N \) homogeneous arms and horizon \( H \), where each arm follows the same (time-dependent, known) state transition and reward function. While computing the exact optimal policy is impractical, the homogeneity of the model allows for the design of computationally efficient policies. One such class of policies is based on fluid approximation, which transforms the original \( N \)-armed RMAB problem into a Linear Program (LP), and an LP-based policy can be efficiently computed based on the solution to the LP.

\textbf{Optimality gap.} In \citep{hu2017asymptotically}, an LP-based index policy was proposed and achieves an \( o(1) \) optimality gap. This gap was later improved to \( \mathcal{O}(\log N / \sqrt{N}) \) in \citep{ZayasCabn2017AnAO}, and subsequently to \( \mathcal{O}(1/\sqrt{N}) \) in \citep{Brown2020IndexPA}. In the numerical experiments of \citep{Brown2020IndexPA}, it was observed that while this \( \mathcal{O}(1/\sqrt{N}) \) gap appears tight for certain problems, in others the gap converges to zero more rapidly. This empirical observation was theoretically confirmed in \citep{zhang2021restless}, where it was shown that under a \emph{non-degenerate} condition (formally defined in Definition~\ref{def:non-degenerate-condition}), the gap is at a smaller order of \( \mathcal{O}(1/N) \).

\textbf{Non-degenerate condition.} This non-degenerate condition has since become a key assumption in subsequent works. \citep{gast2023linear} shows that, under this condition, the optimality gap becomes exponentially small when the rounding error induced by scaling the fluid approximation by \( N \) is eliminated. Further generalizations of the non-degenerate condition have been made in \citep{brown2023fluid}, extending it to multi-action and multi-constraint RMABs (also known as weakly coupled Markov decision processes). In \citep{zhang2024leveraging}, the non-degenerate condition was further extended to settings with heterogeneous arms.

\textbf{Prevalence of degenerant RMABs.} Despite the central role played by the non-degenerate condition, many  RMAB problems are \emph{degenerate}. Notable examples, which originate from real-world applications, have been discussed in \citep{Brown2020IndexPA, zhang2021restless, brown2023fluid}. Moreover, a numerical study presented in Appendix~\ref{append:proportion-degenerate-and-uniqueness} revealed that a significant proportion (about 50\% for some cases) of randomly generated RMABs are degenerate, and almost all of them satisfy the Uniqueness Assumption~\ref{ass:opt-lp-distance}. This highlights the practical importance of addressing degenerate RMABs.

However, to the best of our knowledge, in all previous works, when an RMAB is \emph{degenerate}, the best known optimality gap is \( \mathcal{O}(1/\sqrt{N}) \). It is widely believed that this upper bound is also order-wise tight under LP-based policies, making it \( \Theta(1/\sqrt{N}) \). We will formally prove this result in Theorem \ref{thm:lower-bounds} by using an example.

These results led us to ask the following central question of this paper:

{\em Does there exist a {computationally efficient} algorithm for degenerate RMABs with an optimality gap order-wise smaller than \( \mathcal{O}(1/\sqrt{N}) \)?}

\textbf{Contributions.} This paper answers the question affirmatively and includes the following results:
\begin{itemize}[leftmargin=1.5em]
  \item We construct a Gaussian stochastic system (see \eqref{eq:problem-formulation-Gaussian}) that more accurately captures the behavior of the $N$-system than the fluid system. Unlike the fluid system, which serves as a first-order approximation to the stochastic $N$-system, the Gaussian stochastic system incorporates both the mean and variance of the $N$-system.

  \item We demonstrate that the SP-based policy obtained from the Gaussian stochastic system (see Algorithm~\eqref{algo:SP-based}) achieves $\tilde{\calO}(1/N)$ optimality gap for degenerate RMABs that satisfy the Uniqueness Assumption~\ref{ass:opt-lp-distance} (see Theorem~\ref{thm:global}). We further prove that for degenerate RMABs without the Uniqueness Assumption, SP-based policy results in $\Omega(1/\sqrt{N})$ improvement per arm compared with a large class of LP-based policies (see Theorem~\ref{thm:improvement}).

  \item We further compliment our main result by presenting a \emph{degenerate} example in Section~\ref{sec:SP-based-policy}, demonstrating in Theorem~\ref{thm:lower-bounds} that not only the LP-based policy has \( \Theta(1/\sqrt{N}) \) optimality gap, but also the LP upper bound has \( \Theta(1/\sqrt{N}) \) gap from the optimal value. This indicates that the LP upper bound, which is a widely used baseline, itself is not tight for degenerate RMABs.
\end{itemize}

A comparison of state-of-the-art results in finite-horizon RMABs with this work is provided in Table~\ref{tab:convergence}.

\begin{table}[h]
    \centering
    \caption{Optimality gap and assumptions in finite-horizon RMABs}
    \label{tab:convergence}
    \begin{tabular}{@{\hskip 2pt}p{0.28\linewidth}@{\hskip 3pt}p{0.22\linewidth}@{\hskip 2pt}p{0.45\linewidth}@{\hskip 2pt}}
        \toprule
        \textbf{Paper}         & \textbf{Gap}                & \textbf{Assumption} \\ 
        \midrule
        \citet{Brown2020IndexPA}     & $\mathcal{O}(1/\sqrt{N})$   & General \\ 
        \citet{gast2023linear}       & $\mathcal{O}(\exp(-CN))$    & Non-degeneracy \\ 
        This Work                 & $\tilde{\calO}(1/N)$          & Degeneracy \& Uniqueness Assumption~\ref{ass:opt-lp-distance}  \\ 
        \bottomrule
    \end{tabular}
\end{table}

\begin{figure}[h]
    \centering
    \begin{subfigure}[b]{0.45\linewidth}
        \centering
        \includegraphics[width=\linewidth]{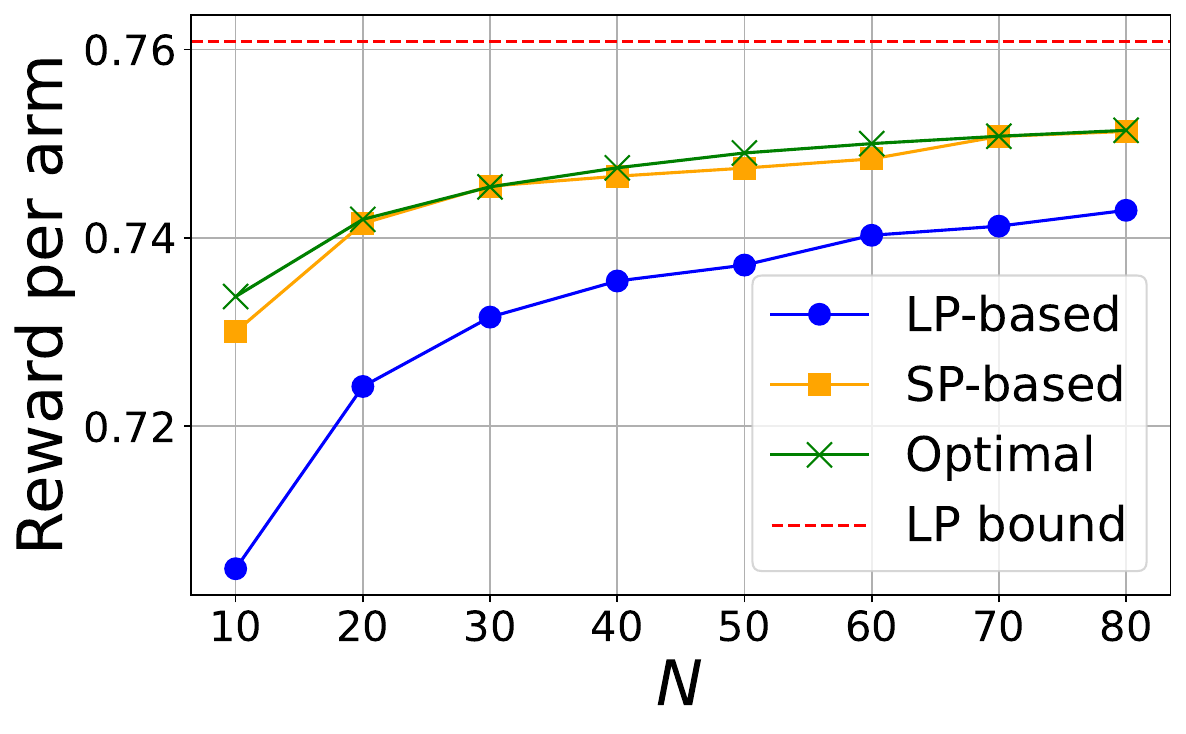}
        \caption*{}
    \end{subfigure}
    \qquad
    \begin{subfigure}[b]{0.45\linewidth}
        \centering
        \includegraphics[width=\linewidth]{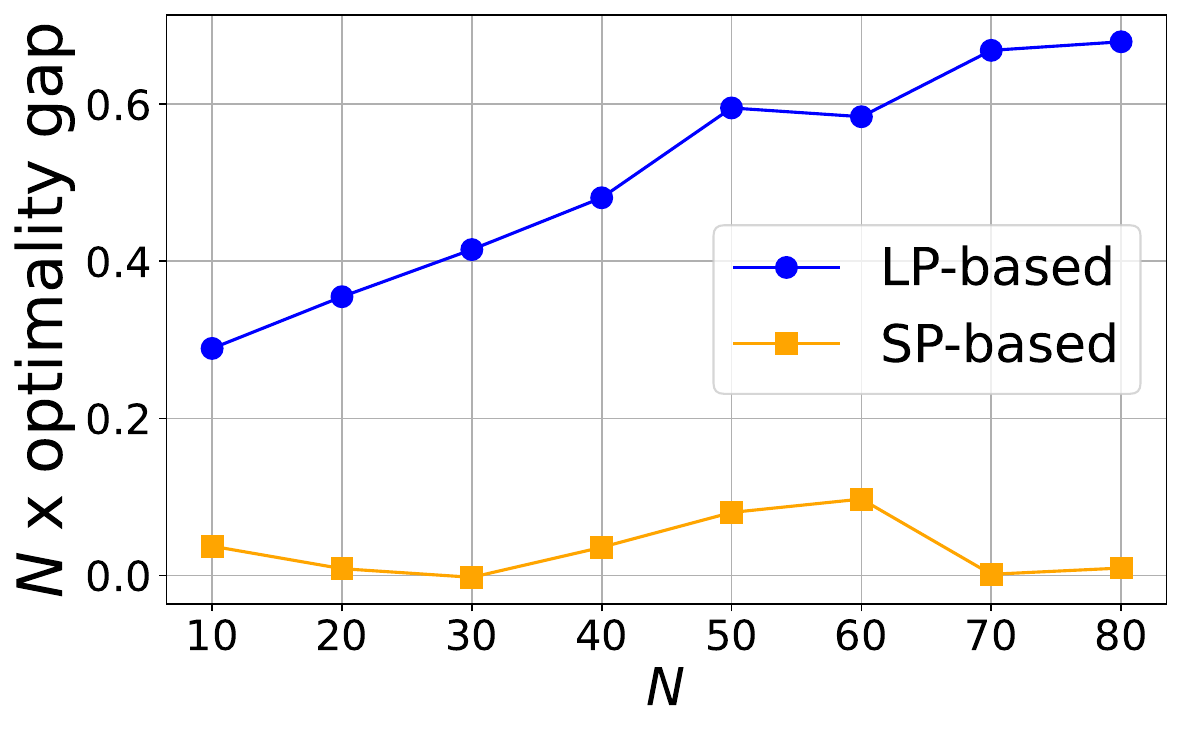}
        \caption*{}
    \end{subfigure}
    \caption{Comparison of LP and SP-based policies on a 2 states 2 steps RMAB example.}
    \label{fig:perf-1}
\end{figure}

As an illustration, consider a degenerate RMAB example with 2 states and a horizon of 2 steps (the details can be found in Section~\ref{sec:example}). Given the small problem size, the optimal policy can be computed exactly through brute-force methods. As shown in Figure~\ref{fig:perf-1}, the performance of the policy obtained by solving the Gaussian stochastic system is already very close to the true optimal value. Considering $N\times$optimality gap (i.e. the total optimality gap of the $N$ arms with respect to the optimal policy), it remains bounded under the SP-based policy 
 as $N$ increases, which confirms the $\tilde{\calO}(1/N)$ optimality gap; while the total optimality gap for the LP-based policy appears to grow unbounded.

\textbf{Additional related work.}
Gaussian approximation, based on the Central Limit Theorem, is a popular approach for approximating stochastic processes. For example, in the heavy-traffic analysis of continuous-time queue theory, such an approximation is called the diffusion approximation because the limiting process, after centering and scaling, becomes a diffusion process. The early works in this line include \cite{Bor_64,Bor_65,IglWhi_70,IglWhi_70_2}. Diffusion approximation in the queueing theory primarily focuses on understanding the convergence or the rate of convergence of the centered and scaled queueing system to the diffusion limit, while this paper uses the Gaussian stochastic system to obtain a near-optimal policy for a decision problem with $N$ homogenous components. Therefore, both the application of the approximation and its analysis are different.

\textbf{Notational convention.} 
Vectors are row vectors by default.
Throughout the paper, we consider three systems: the $N$-armed RMAB system, the fluid system, and the Gaussian stochastic system. As a general convention, variables and functions with ``$\bar{\quad }$'' are for the fluid system, while those with ``$\tilde{\quad }$'' are for the Gaussian stochastic system. The terms ``action'' and ``control'' are used interchangeably in the context of decision-making.

\section{Problem formulation}  \label{sec:model}

\textbf{RMAB model.} We consider the $H$-horizon Restless Multi-Armed Bandit (RMAB) problem with $N$ homogeneous arms. Each arm is modeled as a Markov Decision Process (MDP) with state space \(\calS := \{1, 2, \dots, S\}\) and action space \(\calA := \{0, 1\}\). At each time step $h$ (\(1 \leq h \leq H\)), the decision maker decides which arms to take action $1$, also referred to as the pulling action, subject to the budget constraint that exactly \(\alpha N\) arms should be pulled.
Here $0<\alpha<1$ and we assume \(\alpha N\) is an integer.
After the actions are applied, the \(N\) arms evolve independently.
Specifically, the state transitions from \(\bs \in \calS^N\) to \(\bs' \in \calS^N\) with probability $\pp_h(\bs' \mid \bs, \ba) = \prod_{n=1}^{N} \pp_h (s'_n \mid s_n, a_n)$,
where \((s_n, a_n)\) denotes the state-action pair for the \(n\)-th arm, and \(\pp_h(\cdot \mid \cdot, a)\) is the transition kernel under action \(a\).
For convenience, we refer to this RMAB system with $N$ arms as the \emph{$N$-system}.

At each time step \(h\), the decision maker collects an additive reward $\sum_{n=1}^{N} r_h(s_n, a_n)$, where \(r_h(s, a)\) denotes the reward for the state-action pair \((s, a)\) at time $h$ and is assumed to be nonnegative without loss of generality. The objective is to find a policy \(\pi^N\), which maps the bandit state vector \(\bs_h=(s_{1,h},s_{2,h},\dots,s_{N,h})\) to an action vector \(\ba_h=(a_{1,h},a_{2,h},\dots,a_{N,h})\) for each $h$, that maximizes the total expected reward per arm over the horizon:
\begin{alignat}{2}  
    & \quad \VNopt  :=  \underset{ \piN }{\mathrm{max}} \quad \sum_{h=1}^{H} \frac{1}{N} \sum_{n=1}^{N} \mathbb{E}_{\piN} \left[ r_h(s_{n,h}, a_{n,h} ) \right] \\
    \mathrm{s.t.} & \quad  \sum_{n=1}^N a_{n,h} = \alpha N, \quad 1 \le h \le H. 
\end{alignat}

To reduce the computational complexity, we leverage the fact that the \(N\) arms are homogeneous and aggregate the states of individual arms.
This gives us a simplified representation of the bandit state, which facilitates subsequent analysis and policy design.
Specifically, we represent the bandit state at time $h$ as a vector \(\bX_h =(X_h(s))_{s\in\calS}\in \simpXN\), where each \(X_h(s)\) denotes the fraction of arms in state \(s \in \calS\). 
Here, $\simpXN$ is a discrete subset of \(\Delta^S\), the probability simplex over $\calS$, such that each $\bx \in \simpXN$ satisfies that $N\bx$ has all integer entries.
Similarly, we represent the action at time $h$ as a vector \(\bY_h=(Y_h(s,a))_{s\in\calS,a\in\calA} \in \simpYN\), where each \(Y_h(s, a)\) denotes the fraction of arms in state \(s\) taking action \(a\).
We treat $\bY_h$ as a row vector of length $2S$.
We also write $\bY_h(\cdot,0)=(Y_h(s,0))_{s\in\calS}$ and $\bY_h(\cdot,1)=(Y_h(s,0))_{s\in\calS}$, which are both row vectors of length $S$.

Under the new state and action representation, given an action $\bY_h=\by_h$, the state evolves as
\begin{equation}  \label{eq:rv-Nsys}
  \bX_{h+1}=\frac{1}{N}\sum_{s,a}\mathbf{U}^{(s,a)}_h(\by_h),
\end{equation}
where $\mathbf{U}^{(s,a)}_h(\by_h) \in \mathbb{N}^S$ follows a multinomial distribution $\multi(N y_h(s,a), \pp_h(\cdot \mid s,a))$,
and the $\mathbf{U}^{(s,a)}_h(\by_h)$'s for different $(s,a)$'s are \emph{independent}.

Under the new state and action representation, a policy $\piN$ maps the state $\bX_h$ to an action vector $\bY_h$ for each $h$.
We also rewrite the reward function as a vector \(\br_h=(r_h(s,a))_{s\in\calS,a\in\calA} \in \mathbb{R}^{2S}\).
Suppose the initial state of the $N$-system is \(\bX_1=\xini \in \simpXN\).
Then the objective of the RMAB problem can be reformulated as:
\begin{alignat}{2}  
    & \quad \VNopt(\xini, 1)  =  \ \underset{ \piN }{\mathrm{max}} \quad \sum_{h=1}^{H}  \mathbb{E}_{\piN} \left[ \br_h \bY_h^\top \right] \label{eq:problem-formulation-original} \\
   \hbox{s.t.} & \quad \sum_{s} Y_h(s, 1) = \alpha, \quad 1\le h\le H, \label{eq:budget-constraint} \\
    & \quad \bY_h(\cdot, 0) + \bY_h(\cdot, 1) = \bX_h,\quad
    \bY_h \in \simpYN,\quad 1\le h\le H. \label{eq:positivity-constraint}
\end{alignat}
Here \eqref{eq:budget-constraint} corresponds to the budget constraint, and \eqref{eq:positivity-constraint} ensures that $\bY_h$ is a valid action for state $\bX_h$.

\textbf{Fluid approximation.} 
One classical approach to address the complexity of RMAB is to consider a fluid approximation as in \citep{hu2017asymptotically, ZayasCabn2017AnAO, Brown2020IndexPA, zhang2021restless, gast2023linear}, where the stochastic state transition is replaced by its expected value, leading to a deterministic transition at each step: 
\begin{equation}  \label{eq:f}
  \bx_{h+1}  =  \sum_{s,a}  y_h(s,a) \pp_h(\cdot \mid s,a).
\end{equation}
This fluid system smooths out the stochastic fluctuations in the $N$-system and leads to the following Linear Program (LP), which we refer to as the \emph{fluid LP}: 
\begin{alignat}{2}
    & \quad \VLP(\xini,1) :=  \ \underset{ (\bx_h,\by_h)_{h\in\{1,2,\dots,H\}} }{\mathrm{max}} \quad   \sum_{h=1}^{H}  \br_h \by_h^\top  \label{eq:problem-formulation-LP} \\
\text{s.t.} & \quad  \sum_{s} y_{h} (s,1) = \alpha,\quad 1\le h\le H,\\
&\quad 
\by_h(\cdot, 0) + \by_h(\cdot,1) = \bx_h(\cdot), \quad \by_h \ge \bzero,\quad 1\le h\le H, \\
&\quad \bx_{1} = \xini,
\quad  \bx_{h+1}(\cdot)  = \sum_{s,a} y_{h}(s,a) \mathbf{P}_{h} \left( \cdot \mid s, a \right),\quad 1\le h\le H-1.
\end{alignat}
Let $\bx^*=(\bx_h^*)_{h\in\{1,2,\dots,H\}}$ and $\by^*=(\by_h^*)_{h\in\{1,2,\dots,H\}}$ be an optimal solution to this fluid LP.
Note that this LP can be efficiently solved.
Furthermore, it has been shown that $\VNopt(\xini,1) \le \VLP(\xini,1)$; see, for instance, \citep[Lemma 1]{zhang2021restless} or \citep[Lemma 3]{yan2024optimal}. 
That is, the optimal value of the LP provides an upper bound on the optimal value of the RMAB problem.
Based on the solution to this LP, LP-based policies can be obtained but have  \( \Theta(1/\sqrt{N}) \)  optimality gap per arm in degenerate RMABs, as we pointed out in the Introduction.

\section{Gaussian approximation and SP-based policy}   \label{sec:SP-based-policy}

\subsection{Gaussian stochastic system} 
The performance of LP-based policies is inherently limited by how accurately the fluid system approximates the original $N$-system.
To design policies that outperform the LP-based policies, we construct a Gaussian stochastic system that better approximates the original $N$-system by capturing not only the mean but also the variance of the system.
This Gaussian stochastic system is centered around $\by^*$, an optimal solution to the fluid LP.
We then search for an optimal policy for the Gaussian stochastic system in a neighborhood of $\by^*$, which adjusts $\by^*$ to account for the stochasticity.
This policy is then applied to the $N$-system after properly handling the integer effect.

Specifically, we construct a Gaussian stochastic system with state space $\Delta^S$ and action space $\Delta^{2S}$. 
The system has the same initial state $\xini$ as the $N$-system.
Under action vector $\by_h$, the state of the Gaussian stochastic system at the next time step, \(\Xdiff_{h+1}\), is a random vector of the following form: 
\begin{equation}  \label{eq:X-c}
 \Xdiff_{h+1} = \proj{\Delta^S}{\sum_{s,a} y_h(s,a)\pp_h(\cdot \mid s,a)+\bZ_h/\sqrt{N}},
\end{equation} where $\bZ_h$ is a Gaussian random vector with distribution $\calN(\bzero, \Gamma_h(\by^*_h))$, and $\bZ_h$'s are independent across steps.
The covariance matrix $\Gamma_h(\by^*_h)$ is a constant matrix independent of $N$, with its explicit expression given below.
The projection, $\proj{\Delta^S}{\cdot}$, is onto the simplex $\Delta^S$ under the $\ell_2$-distance.
This projection will not be used with a high probability.
Indeed, we will show in Lemma~\ref{lem:high-proba-bound} in the appendix that $\sum_{s,a} y_h(s,a)\pp_h(\cdot \mid s,a) + \bZ_h/\sqrt{N} \in \Delta^S$ occurs with probability $1 - \calO(1/N^{\log N})$.

We now explain how the Gaussian stochastic system captures the mean and variance of the $N$-system.
Suppose the action at time $h$ is $\by_h$.
Ignoring the projection, we have 
\begin{equation*}
    \Xdiff_{h+1} =\sum_{s,a} y_h(s,a)\pp_h(\cdot \mid s,a) + \bZ_h/\sqrt{N}.
\end{equation*}
The term $\sum_{s,a} y_h(s,a)\pp_h(\cdot \mid s,a)$ is the expectation of the state $\bX_{h+1}$ in the $N$-system, which is the same as the state $\bx_{h+1}$ in the fluid system under action $\by_h$.
The additional term, $\bZ_h/\sqrt{N}$, is random, and we construct it in such a way that its covariance matrix matches that of $\bX_{h+1}$ in the $N$-system if $\by_h=\by_h^*$. In other words, for $(h,s,a) \in [1,H] \times \calS \times \calA$, denote the matrix $\bSigma_h(s,a)$ of size $\abs{\calS} \times \abs{\calS}$, for which the $(i,j)$-th entry is:
\begin{equation}  \label{eq:covariance-matrix}
 \begin{cases}
   - \pp_h(i \mid s,a) \pp_h(j \mid s,a), & \mbox{if } i \neq j \\
   \pp_h(i \mid s,a)  (1 - \pp_h(i \mid s,a)), & \mbox{if } i=j.
 \end{cases}
\end{equation}
Then
\begin{equation}  \label{eq:Gamma-definition}
    \Gamma_{h}(\by^*_h) := \sum_{s,a} y^*_{h}(s,a) \bSigma_{h}(s,a).
\end{equation}
Note that by construction, $\Gamma_{h}(\by^*_h)/N$ is the covariance matrix of the state $\bX_{h+1}$ in the $N$-system if the action at time $h$ is $\by_h^*$.
Furthermore, this term captures the variance of the $N$-system when $\by_h$ is close to the optimal solution $\by_h^*$ of the fluid LP. We remark that it is important to define the covariance matrix based on $\by^*_h$ instead of $\by_h.$ Using $\Gamma_h(\by_h)$ in \eqref{eq:Gamma-definition} would make $\tilde{\bX}_{h+1}$ a more accurate approximation of $\bX_{h+1}$ but would make  $\bZ_h$ action-dependent. The problem becomes much more difficult if $\bZ_h$'s distributions need to be optimized instead of being pre-determined

\subsection{Stochastic-Programming-based (SP-based) policy}

After constructing the Gaussian stochastic system, we search for an optimal policy in the Gaussian stochastic system.
However, we restrict the search to a \emph{neighborhood of $\by^*$}, since the Gaussian stochastic system closely approximates the $N$-system when the state and action of the $N$-system stay close to $\by^*$.
In particular, given a fixed parameter $\Ndelta :=  2 \log N/\sqrt N = \tilde{\Theta}(1/\sqrt{N})$, we define the following policy class: 
\begin{equation}\label{eq:policy-pidelta}
\begin{split}
 \pidelta :=  \big\{ &\pi\colon \forall 1\le h\le H,\;\|\pi(\bx_h,h)-\by^*_h\|_\infty\leq \kappa z_h \Ndelta,\;\hbox{if }  \|\bx_h-\bx^*_h\|_\infty \leq z_h \Ndelta;\\
   & \mspace{25mu}\pi = \pip \hbox{ for a predefined policy }\pip,  \; \hbox{if }  \|\bx_h-\bx^*_h\|_\infty > z_h \Ndelta \big\}.  
\end{split}
\end{equation}
Here $\kappa$ is a positive constant with value $\kappa := \max \left\{ 2 + 6S, \, 3 + 2 \rmax HS/\sigma \right\}$, where $\rmax := \max_{s,a,h} r_h(s,a)$; $\sigma$ is a constant depending on the fluid LP; $z_h$ for $1 \le h \le H$ is a recursive sequence. The explicit expressions for these constants are collected in Appendix~\ref{append:list-of-notations-and-constants}.
We note that all of them are independent of $N$, and can be computed or estimated solely from the fluid LP \eqref{eq:problem-formulation-LP}.

We then formulate the following Gaussian stochastic program (SP):  
\begin{alignat}{2}  
    &  \quad \ \underset{ \pi\in \pidelta}{\mathrm{max}} \quad \sum_{h=1}^{H}  \mathbb{E} \left[ \br_h \tilde\bY_h^\top \right] \label{eq:problem-formulation-Gaussian} \\
    \hbox{s.t.} & \quad \sum_{s} \tilde Y_h(s, 1) = \alpha,\quad 1\le h\le H,\\
    &\quad \tilde \bY_h(\cdot, 0) + \tilde \bY_h(\cdot, 1) = \tilde \bX_h,\quad \tilde \bY_h \geq {\bf 0}, \quad 1\le h\le H,\\
    & \quad \Xdiff_{h+1} = \proj{\Delta^S}{\sum_{s,a} \tilde{Y}_h(s,a)\pp_h(\cdot \mid s,a)+\bZ_h/\sqrt{N}}, \quad 1\le h\le H-1.
\end{alignat}

\begin{algorithm}[t]
  \caption{Stochastic-Programming-based (SP-based) policy}
  \label{algo:SP-based}
  \begin{algorithmic}[1]
    \State \textbf{Input:} An optimal solution $\by^*$ to LP~\eqref{eq:problem-formulation-LP}; constants $z_h$, $\Ndelta$ and $\kappa$; a predefined policy $\pip$
    \If{$\by^*$ is non-degenerate}
        \State Use an LP-based policy
        \State \textbf{Break}
      \EndIf
    \State Solve the Gaussian stochastic program~\eqref{eq:problem-formulation-Gaussian} to obtain an optimal policy $\tilde{\pi}^{N,*} \in \pidelta$
    \For{$h = 1$ \textbf{to} $H$}
      \State $\bx_h\gets$ state of the $N$-system at time $h$
      \If{$\norminf{\bx_h - \bx^*_h} \le z_h \Ndelta$}
        \State $\bY_h \gets \round(\tilde{\pi}^{N,*}(\bx_h,h))$
      \Else 
        \State $\bY_h \gets \round(\pip(\bx_h,h))$
      \EndIf
      \State Apply action $\bY_{h}$
    \EndFor
  \end{algorithmic}
\end{algorithm}

We now present our SP-based policy, formally described in Algorithm~\ref{algo:SP-based}.
The core idea of this policy is to solve the Gaussian stochastic program and obtain an optimal policy $\tilde{\pi}^{N,*} \in \pidelta$.
This policy is then applied to the $N$-system through a simple rounding procedure, which maps the action given by $\tilde{\pi}^{N,*}$ to a feasible action in the $N$-system with $\calO(1/N)$ error; see Appendix~\ref{app:details-Gaussian} for details.
We remark that we only use the policy $\tilde{\pi}^{N,*}$ when the RMAB is degenerate.
When the RMAB is non-degenerate, our policy defaults to a LP-based policy, which prior work (see Introduction) has shown to achieve an exponentially small optimality gap (in terms of \( N \)) relative to \( \VLP \).
The definition of non-degeneracy is given below, and it is easy for an algorithm to check whether an RMAB is non-degenerate or not.

\begin{definition}[Non-degeneracy \citep{zhang2021restless, gast2023linear, brown2023fluid}]  \label{def:non-degenerate-condition}
  An RMAB is \emph{non-degenerate} if, its corresponding fluid LP \eqref{eq:problem-formulation-LP} admits an optimal solution $\by^*$, such that for each $h$ with $1 \le h \le H$, there exists at least one state $s \in \calS$ such that $y^*_h(s,0)>0$ and $y^*_h(s,1)>0$. 
\end{definition}

\subsection{Illustration of the SP-based policy on a degenerate example}  
\label{sec:example}
We illustrate the SP-based policy via a two-state RMAB with horizon $H=2$ and pulling budget $\alpha=0.5$.
The rewards are given by $r_1(1,1) = r_2(1,1)=1$, with all other $(h,s,a)$-tuples being $r_h(s,a)=0$. 
The transition probabilities at $h=1$ are $\pp_1(1 \mid 1,1) = 0.2$, $\pp_1(1 \mid 1,0) = 0.9$, $\pp_1(1 \mid 2,1) = 0.7$, $\pp_1(1 \mid 2,0) = 0.25$. 
Assume at $h=1$ there are $N/2$ arms in state $1$, and the other $N/2$ arms are in state $2$.

\paragraph{$N$-system problem.}
Note that at time $h=2$, only $r_2(1,1)=1,$ so the optimal action at $h=2$ is to pull as many arms in state $1$ as possible, i.e., $Y_2(1,1)=\min\{0.5,X_2(1)\}.$ 
Therefore, we only need to decide the optimal action at $h=1$. 
We further notice that since $X_1(1)=X_1(2)=\alpha=0.5,$ the entries of $\bY_1$ are all determined by $Y_1(1,1)$ as follows
\[
Y_1(1,0)=Y_1(2,1)=0.5-Y_1(1,1),\quad Y_1(2,0)=Y_1(1,1).
\]
Therefore, we only need to optimize $Y_1(1,1)$.
The $N$-system optimal policy is given by the solution of the following problem
\begin{equation}\label{eq:example-N-system}
    \max_{0\le Y_1(1,1)\le 0.5} \ Y_1(1,1)  + \expect{\min \left\{0.5,X_2(1)\right\}}.
\end{equation}

\paragraph{Fluid LP and its optimal solution.}
In the fluid system, we replace $X_2(1)$ with its mean $x_2(1)=\sum_{s,a}y_1(s,a)\pp_1(1\mid s,a)=0.8-1.15\times y_1(1,1)$.
Then the fluid LP is
\begin{align} 
\VLP 
=&\max_{0 \le y_1(1,1) \le 0.5} y_1(1,1) + \min \big\{ 0.5, \; 0.8-1.15\times y_1(1,1) \big\},\label{eq:LP-simple}
\end{align} 
which exchanges the expectation and the min operator in the $N$-system problem \eqref{eq:example-N-system}.
The optimal solution is $y^*_1(1,1)=0.2609.$
One can verify that this problem is degenerate (see Definition~\ref{def:non-degenerate-condition}).

\paragraph{SP-based policy.}
Given the fluid solution $\by^*$ above, the corresponding Gaussian stochastic program is 
     \begin{equation}   \label{eq:Gaussian-system-problem}
    \max_{0\le \tilde{Y}_1(1,1)\le 0.5} \ \tilde{Y}_1(1,1)  + \expect{\min \left\{0.5, \; 0.8-1.15\times\tilde{Y}_1(1,1)+\frac{Z_1}{\sqrt{N}} \right\}},
     \end{equation} where $Z_1\simd \calN(0, \, 0.1624).$ Since we are searching for an optimal solution around $\by^*_1,$ let $\tilde{Y}_1(1,1)=y_1^*(1,1)+\frac{c}{\sqrt{N}}.$ 
     Then the stochastic program can be written as
     \begin{equation*} 
        \by^*_1(1,1)+0.5+\frac{1}{\sqrt{N}}\max_{c } \left(c + \expect{\min \left\{0, Z_1-1.15c \right\}}\right),
     \end{equation*} which is equivalent to  the following problem
        \begin{equation*} 
        \max_{c } \ c + \expect{\min \left\{0, Z_1-1.15c \right\}}.
     \end{equation*} There exists an explicit and unique solution to the problem above, and the solution can be numerically computed and is $c^*_{\text{d}}=0.3940$ (see details in Appendix~\ref{append:proof-of-sqrt-N-gap}). Therefore, the SP-based policy is given by $\round(\tilde{Y}^*_1(1,1))=\round(0.2609+0.3940/\sqrt{N})$.
This SP-based policy outperforms LP-based policies, as illustrated in Figure~\ref{fig:perf-1} in the Introduction.

\paragraph{Some insights.}
We compare the fluid approximation in \eqref{eq:LP-simple} with the Gaussian approximation in \eqref{eq:Gaussian-system-problem} when they both take an action $y_1(1,1)=\tilde{Y}_1(1,1)=y_1^*(1,1)+c/\sqrt{N}$ for a positive constant $c$.
In the fluid system, the reward for $h=2$ is $\min \big\{ 0.5, \; 0.8-1.15\times y_1(1,1) \big\}$, which is capped at $0.5$.
One can verify that this deviates from the value $\expect{\min \left\{0.5,X_2(1)\right\}}$ in the $N$-system by $\Theta(1/\sqrt{N})$, caused by the exchange of the expectation and the min operator.
In contrast, in the Gaussian stochastic system,
\[
\expect{\min \left\{0.5, 0.8-1.15\times\tilde{Y}_1(1,1)+\frac{Z_1}{\sqrt{N}} \right\}}=0.5+\expect{\min\left\{0,\frac{Z_1-1.15c}{\sqrt{N}}\right\}},
\]
which can be verified to be $\tilde{\calO}(1/N)$ away from the value $\expect{\min \left\{0.5,X_2(1)\right\}}$ in the $N$-system and thus giving a better approximation.
This better approximation allows us to find a better policy near $y_1^*(1,1)$.
The inaccuracy of the fluid approximation is in fact a fundamental reason that LP-based policies have a $\Theta(1/\sqrt{N})$ optimality gap for some degenerate RMABs, 
whereas the correction in our SP-based policy reduces the $\Theta(1/\sqrt{N})$ inaccuracy to $\tilde{\calO}(1/N)$.

\section{Main theoretical results}   \label{subsec:performance-analysis-general}

In this section, we present our main theoretical results.  
We will frequently use the following quantities.
Let $V^N_\pi(\bx_h, h)$ and ${Q}^N_\pi(\bx_h, \by_h, h)$ denote the value function and Q-function of policy $\pi$ evaluated in the $N$-system; and $\tilde{V}^N_\pi(\bx_h, h)$ and $\tilde{Q}^N_\pi(\bx_h, \by_h, h)$ denote the value function and Q-function of a policy $\pi$ evaluated in the Gaussian stochastic system.

Further, we consider the following optimal policies within the policy class $\pidelta$:
\begin{align*}
    \quad {\pi}^{N,*} \in \argmax_{\pi\in\pidelta}  {V}^N_\pi(\bx_h,h) \quad \mbox{and}\quad
    \tilde{\pi}^{N,*} \in \argmax_{\pi\in\pidelta}   \tilde{V}^N_\pi(\bx_h,h),
\end{align*} 
which are referred to as the {\em locally-optimal} policy and the {\em locally-SP-optimal} policy, respectively. 
We sometimes say that we apply the locally-SP-optimal policy $\tilde{\pi}^{N,*}$ to the $N$-system and denote its value function as $V^N_{\tilde{\pi}^{N,*}}(\bx_h, h)$, with the understanding that we apply $\tilde{\pi}^{N,*}$ with the rounding procedure.

There is a minor technical subtlety in the optimization within the policy class $\Pi_{\Ndelta}(\by^*)$.
By the definition of $\Pi_{\Ndelta}(\by^*)$, a policy in this class follows a predefined policy outside the $\tilde{\Theta}(1/\sqrt{N})$-neighborhood of $\bx^*_h$, which is not affected by the optimization.
With a slight abuse of terminology, we choose not to distinguish these locally optimal policies with the ones that also optimize outside this neighborhood. This simplification is justified for the following reason.  The smallest optimality gap of interest in this paper is of order $\tilde{\calO}(1/N)$.
However, according to the high probability bound established in Lemma~\ref{lem:high-proba-bound}, the state of the $N$-system's state falls out of this neighborhood with probability $\calO(1/N^{\log N})$, and consequently has only this order of influence in the value function.

\subsection{Global optimality} 
\begin{assumption}[Uniqueness]  \label{ass:opt-lp-distance}
The fluid LP in \eqref{eq:problem-formulation-LP} has a unique optimal solution $\by^*$.
\end{assumption}
Note that once we obtain an optimal solution to the LP, verifying uniqueness is straightforward \citep{bertsimas1997introduction}.

\begin{restatable}[Global optimality]{theorem}{GlobalOptimality}  \label{thm:global}
Consider an RMAB that satisfies the Uniqueness Assumption~\ref{ass:opt-lp-distance}.
Then the locally-SP-optimal policy, $\tilde{\pi}^{N,*}$, when applied to the $N$-system (with rounding), achieves an optimality gap of $\tilde{\calO}(1/N)$; i.e.,
\begin{equation*}
  \VNopt(\xini, 1) - V^N_{\tilde{\pi}^{N,*}}(\xini, 1) = \tilde{\calO}(1/N), 
\end{equation*}
where $\VNopt$ is the optimal value function, and $V^N_{\tilde{\pi}^{N,*}}$ is the value function of $\tilde{\pi}^{N,*}$, both in the $N$-system. 
\end{restatable}

A detailed proof of Theorem~\ref{thm:global} is presented in Appendix~\ref{appendix:proof-global-optimality}.
Below, we provide an outline and highlight some novel and technically interesting components of the proof.

\begin{figure}
\centering
\includegraphics[scale=1]{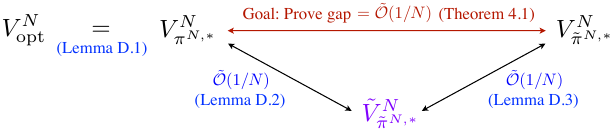}
\caption{Proof structure of Theorem~\ref{thm:global}.}
\label{fig:proof-outline}
\end{figure}
Our proof of Theorem~\ref{thm:global} consists of proving the following statements, as illustrated in Figure~\ref{fig:proof-outline}.
\begin{itemize}[leftmargin=1.5em]
    \item \emph{Prove that $\VNopt=V^N_{\pi^{N,*}}$ (Lemma~\ref{prop:sufficient-condition-of-assumption} in the appendix), where $\pi^{N,*}$ is the locally optimal policy within the policy class $\pidelta$.}
    This is proved by showing that a (globally) optimal policy for the $N$-system belongs to the policy class $\pidelta$ under Assumption \ref{ass:opt-lp-distance}.
    \item \emph{Prove that $|V^N_{\pi^{N,*}}(\bx, h)-\tilde{V}^N_{{\tilde{\pi}^{N,*}}}(\bx, h)| =\tilde{\calO}(1/N)$ (Lemma~\ref{lem:vd} in the appendix) and $|\tilde{V}^N_{{\tilde{\pi}^{N,*}}}(\bx, h)-V^N_{\tilde{\pi}^{N,*}}(\bx, h)|=\tilde{\calO}(1/N)$ (Lemma~\ref{lem:ed} in the appendix).}
    These two lemmas are enabled by the fact that we restrict the policies $\pi^{N,*}$ and $\tilde{\pi}^{N,*}$ to $\pidelta$, i.e., a $\tilde{\Theta}(1/\sqrt{N})$-neighborhood of $\by^*$, which translates into a Wasserstein distance of $\tilde{\calO}(1/N)$ between the respective next-state distributions in the $N$-system and in the Gaussian stochastic system.
\end{itemize}

We next highlight two interesting components in the proofs.
\begin{itemize}[leftmargin=1.5em]
    \item \emph{Characterization of optimal policies in the $N$-system.}
    Lemma~\ref{prop:sufficient-condition-of-assumption} in the appendix establishes that under Assumption~\ref{ass:opt-lp-distance}, there exists an optimal policy of the $N$-system whose actions are close to the optimal fluid solution $\by^*$.
    This result is noteworthy because optimal policies of the $N$-system are not well-understood in the literature.
    Prior work often circumvents this by only studying the LP upper bound $\VLP$, which can be loose as shown in Theorem~\ref{thm:lower-bounds}.
    \item \emph{Approximate Lipschitz continuity in the Gaussian stochastic system.}
    A key step in proving Lemmas~\ref{lem:vd} and \ref{lem:ed} in the appendix is to establish an approximate local Lipschitz property of the value function $\tilde{V}^N_{{\tilde{\pi}^{N,*}}}$ in the Gaussian stochastic system, where restricting the policy to ${\tilde{\pi}^{N,*}}$ introduces technical challenges.
    We overcome these challenges through the careful construction of an action mapping.
\end{itemize}

\subsection{The $\Theta(1/\sqrt{N})$ optimality gap of LP-based policies}
\label{sec:fluid-gap}
To complement Theorem~\ref{thm:global}, we next present a result showing that the $\Theta(1/\sqrt{N})$ optimality gap is fundamental to a large class of LP-based policies. 
Specifically, consider the following policy class, which includes a large class of LP-based policies such as those in \citep{Brown2020IndexPA, zhang2021restless, gast2023linear}:  
\begin{align}   \label{def:policy-sub}
\pisub := \big\{ \pi\colon  \|\pi(\bx_h,h) - \by^*_h \|_\infty\leq \kappa \|\bx_h - \bx^*_h\|_\infty, \; \forall 1\le h \le H \big\},
\end{align}
where $(\bx^*,\by^*)$ is an optimal solution of the fluid LP in ~\eqref{eq:problem-formulation-LP}.
It has been shown in \citep[Theorem 1]{gast2023linear} that any policy in $\pisub$ has an $\calO(1/\sqrt{N})$ optimality gap.
We now show that there exist RMAB instances where this optimality gap order is tight. The result is established based on the example given in Section \ref{sec:example}, and the detailed proof is presented in Appendix~\ref{append:proof-of-sqrt-N-gap}.

\begin{restatable}[Fluid gap]{theorem}{LowerBounds}   \label{thm:lower-bounds}
There exist RMAB instances for which all LP-based policies in $\pisub$ have an $\Theta(1 / \sqrt{N})$ optimality gap; i.e.,
$\VNopt(\xini,1) - \VNLPres(\xini,1) = \Theta(1 / \sqrt{N}),$
where $\VNopt$ is the optimal value function, and $\VNLPres$ is the value function of the optimal policy within the policy class $\pisub$.
Moreover, there is an $\Theta(1 / \sqrt{N})$ gap between the optimal value of the $N$-system and the optimal value of the fluid LP in \eqref{eq:problem-formulation-LP}; i.e., $\VLP(\xini,1) - \VNopt(\xini,1) = \Theta(1 / \sqrt{N})$.
\end{restatable}

We remark that although this theorem is proved via a specific example, we believe the result to hold more broadly for many degenerate RMABs.
The $\Theta(1 / \sqrt{N})$ optimality gap of LP-based policies stems from the $\Theta(1 / \sqrt{N})$ approximation error in the fluid approximation, which arises when exchanging the expectation and the minimum operator, as shown in the example in Section~\ref{sec:example}.
This phenomenon is common in degenerate RMABs.

\subsection{Performance improvement}

Under the Uniqueness Assumption~\ref{ass:opt-lp-distance}, we have shown that our SP-based policy achieves an $\tilde{\calO}(1/N)$ optimality gap.
When this assumption does not hold, the same optimality gap may not apply.
However, Theorem~\ref{thm:improvement} below shows that the SP-based policy can still yield improvement over LP-based policies.

Let $\by^*$ be any optimal solution to the fluid LP in~\eqref{eq:problem-formulation-LP}, and let $\tilde{\pi}^{N,*}$ be the corresponding SP-based policy.
Recall that $\tilde{V}^N_{\tilde{\pi}^{N,*}}(\xini, 1)$ and $\tilde{Q}^N_{\tilde{\pi}^{N,*}}(\xini, \by_1, 1)$ denote the value function and the Q-function of the policy $\tilde{\pi}^{N,*}$ in the Gaussian stochastic system, where in the Q-function action $\by_1$ is applied at time $1$.
Theorem~\ref{thm:improvement} states that if the action given by the SP-based policy $\tilde{\pi}^{N,*}$ at time~$1$ is ``strictly'' better than the optimal LP-based action $\by_1^*$ in the Gaussian stochastic system, then the SP-based policy improves over any LP-based policy in $\pisub$ by $\Omega(1/\sqrt{N})$ in the $N$-system.

\begin{restatable}[Performance improvement]{theorem}{PerformanceImprovement}   \label{thm:improvement}
If there exists a positive constant $\epsilon$ independent of $N$ such that $\tilde{V}^N_{\tilde{\pi}^{N,*}}(\xini, 1) - \tilde{Q}^N_{\tilde{\pi}^{N,*}}(\xini, \by^*_1, 1) \ge \epsilon/\sqrt{N}$, then for any $\pi\in \pisub$, we have $V^N_{\tilde{\pi}^{N,*}}(\xini, 1) - V_{\pi}^N(\xini, 1) =\Omega(1/\sqrt{N})$.
\end{restatable}

Note that a key strength of this result is that we only need to evaluate the first step action $\by_1^*$ in the Gaussian stochastic system, and the result then holds for all policies in $\pisub$.  We refer to Appendix~\ref{append:proof-of-performance-improvement} for a detailed proof of Theorem~\ref{thm:improvement}.

\section{Numerical experiments}   \label{sec:numerical-experiments}

In this section, we first present the details of the numerical method we used to solve the Gaussian stochastic program~\eqref{eq:problem-formulation-Gaussian}, which is an essential part of Algorithm~\ref{algo:SP-based}. 
For the SP~\eqref{eq:problem-formulation-Gaussian}, we introduce the vectors $(\bc_h, \bd_h)_{h \in \left\{ 1,2,\dots,H\right\}}$ as follows: 
\begin{align}
{\tilde \bY}_h := & \ \by^*_h+\frac{\bc_h}{\sqrt{N}},  \label{eq:Y-and-c-relation} \\
\bd_{h+1} := & \ \proj{\sqrt{N}(\Delta^S - \bx^*_{h+1})}{\sum_{s,a} c_h(s,a)\pp_h(\cdot \mid s,a)+{\bZ_h}}, \nonumber \\
\tilde{\bX}_{h+1} := & \ \bx^*_{h+1}+\frac{\bd_{h+1}}{\sqrt{N}}.  \label{eq:X-and-d-relation}
\end{align}
With these vectors, we can rewrite the SP~\eqref{eq:problem-formulation-Gaussian} as follows: 
\begin{alignat}{2}  
    & \quad \underset{\pi_c\in \Npic}{\mathrm{max}} \quad \sum_{h=1}^{H}  \mathbb{E} \left[ \br_h \bc_h^\top \right] \label{eq:problem-formulation-Gaussian-new} \\
    \hbox{s.t.} & \quad \sum_{s} c_h(s, 1) = 0 \ \text{and} \ \bc_h(\cdot, 0) + \bc_h(\cdot, 1) = \bd_h, \quad 1 \le h \le H, \nonumber \\
    & \quad c_h(s, a) \geq - \sqrt{N} y_h^*(s,a), \quad 1 \le h \le H,   \label{eq:constraint-on-c}  \\
    & \quad \bd_1 = \bzero, \quad \quad \bd_{h+1} = \proj{\sqrt{N}(\Delta^S - \bx^*_{h+1})}{\sum_{s,a} c_h(s,a)\pp_h(\cdot \mid s,a)+{\bZ_h}},  \quad 1 \le h \le H-1.  \label{eq:constraint-on-d}
    \end{alignat} 
    Here $\Npic$ is a set of policies where each policy maps $\bd_h$ to $\bc_h$ such that $\|\bc_h\|_\infty \leq \kappa z_h \Ndelta\sqrt{N}$ if $\|\bd_h\|_\infty \leq z_h \Ndelta\sqrt{N}$, and it follows a predefined policy otherwise; i.e.,
    \begin{align*}
    \Npic := \big\{ \pi_c\colon & \|\pi_c(\bd_h,h)\|_\infty\leq \kappa z_h \Ndelta \sqrt{N},  \ \hbox{ if }  \|\bd_h\|_\infty \leq z_h \Ndelta \sqrt{N}; \\
    & \pi_c \hbox{ follows a predefined policy},  \ \hbox{ if }  \|\bd_h\|_\infty > z_h \Ndelta \sqrt{N} \big\}.
    \end{align*}
Note that there is a one-to-one correspondence between $\pidelta$ and $\Npic$. Hence solving the SP~\eqref{eq:problem-formulation-Gaussian} is equivalent to solving the SP~\eqref{eq:problem-formulation-Gaussian-new}.
We further simplify the SP~\eqref{eq:problem-formulation-Gaussian-new} to the following $N$-independent and projection-free SP:
\begin{alignat}{2}  
    & \quad \underset{ \pi_c\in \pic}{\mathrm{max}} \quad \sum_{h=1}^{H}  \mathbb{E} \left[ \br_h \bc_h^\top \right] \label{eq:problem-formulation-Gaussian-N-independent} \\
    \hbox{s.t.} & \quad \sum_{s} c_h(s, 1) = 0 \ \text{and} \ \bc_h(\cdot, 0) + \bc_h(\cdot, 1) = \bd_h, \quad 1 \le h \le H, \nonumber \\
    & \quad c_h(s, a) \geq 0 \ \text{ if $y^*_h(s,a)=0$}, \quad 1 \le h \le H, \label{eq:constraint-on-c-N-independent}  \\
    & \quad \bd_1 = \bzero, \quad \bd_{h+1} = \sum_{s,a} c_h(s,a)\pp_h(\cdot \mid s,a) + \bZ_h, \quad 1 \le h \le H-1, \label{eq:constraint-on-d-N-independent}
    \end{alignat} 
where the policy class $\pic$ is
\begin{align*}
    \pic := \big\{ \pi_c\colon & \|\pi_c(\bd_h,h)\|_\infty\leq \kappa \tilde{D},  \ \hbox{ if }  \|\bd_h\|_\infty \leq \tilde{D}; \\
    & \pi_c \hbox{ follows a predefined policy},  \ \hbox{ if }  \|\bd_h\|_\infty > \tilde{D} \big\}.
\end{align*} In our numerical experiments, we have chosen $\tilde{D}$ to be $20$.

Let ${\tilde{\pi}^{*}_c}$ be an optimal policy in $\pic$ for the SP~\eqref{eq:problem-formulation-Gaussian-N-independent}. For the $N$-system, given state $\bx_h$ at step $h,$ we apply action $$\by_h=\round\left(\by^*_h+\frac{\tilde{\pi}^{*}_c\left(\sqrt{N}(\bx_h-\bx^*_h)\right)}{\sqrt{N}}\right).$$ In other words, the solution to the SP \eqref{eq:problem-formulation-Gaussian-N-independent} provides an SP-based policy that can be applied to all $N.$

The SP~\eqref{eq:problem-formulation-Gaussian-N-independent} is a stochastic linear program, which is significantly easier to solve than the original $N$-armed RMAB Problem~\eqref{eq:problem-formulation-original} because the ``noise'' $\bZ_h$'s are predefined Gaussian random vectors whose distributions are independent of the state and action of the system.  
For such a problem, there exist computationally efficient methods.  In our numerical examples, we numerically solved the $1/\sqrt{N}$-scale SP~\eqref{eq:problem-formulation-Gaussian-N-independent} using the standard Sample Average Approximation (SAA) approach \citep[Chapter~5]{shapiro2021lectures} and  the Explorative Dual Dynamic Programming (EDDP) algorithm \citep[Algorithm 3]{lan2022complexity}, which is an enhancement of the classical Stochastic Dual Dynamic Programming (SDDP) \citep{pereira1991multi,shapiro2011analysis}.

\subsection{Numerical evaluation on machine maintenance RMABs}

\begin{figure}[h]
    \centering
    \begin{subfigure}[b]{0.45\linewidth}
        \centering
        \includegraphics[width=\linewidth]{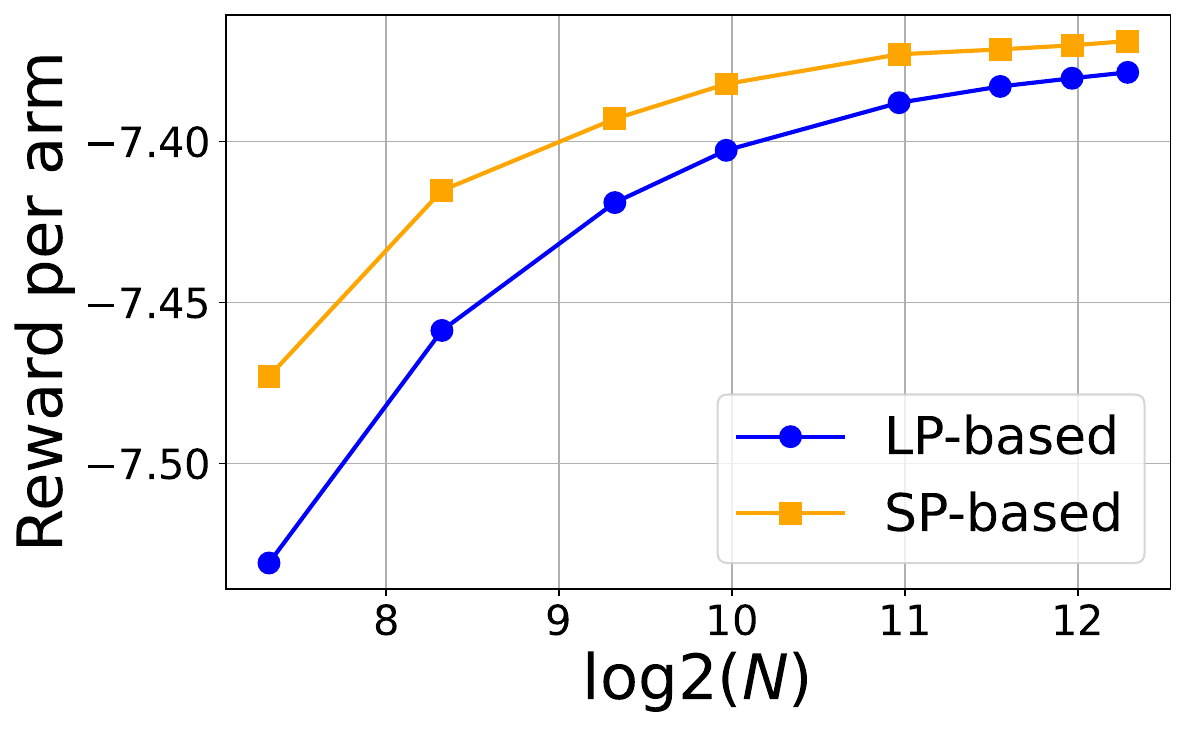}
        \caption{Reward per arm: unique fluid LP solution}
        \label{subfig:unique-1}
    \end{subfigure}
    \qquad
    \begin{subfigure}[b]{0.45\linewidth}
        \centering
        \includegraphics[width=\linewidth]{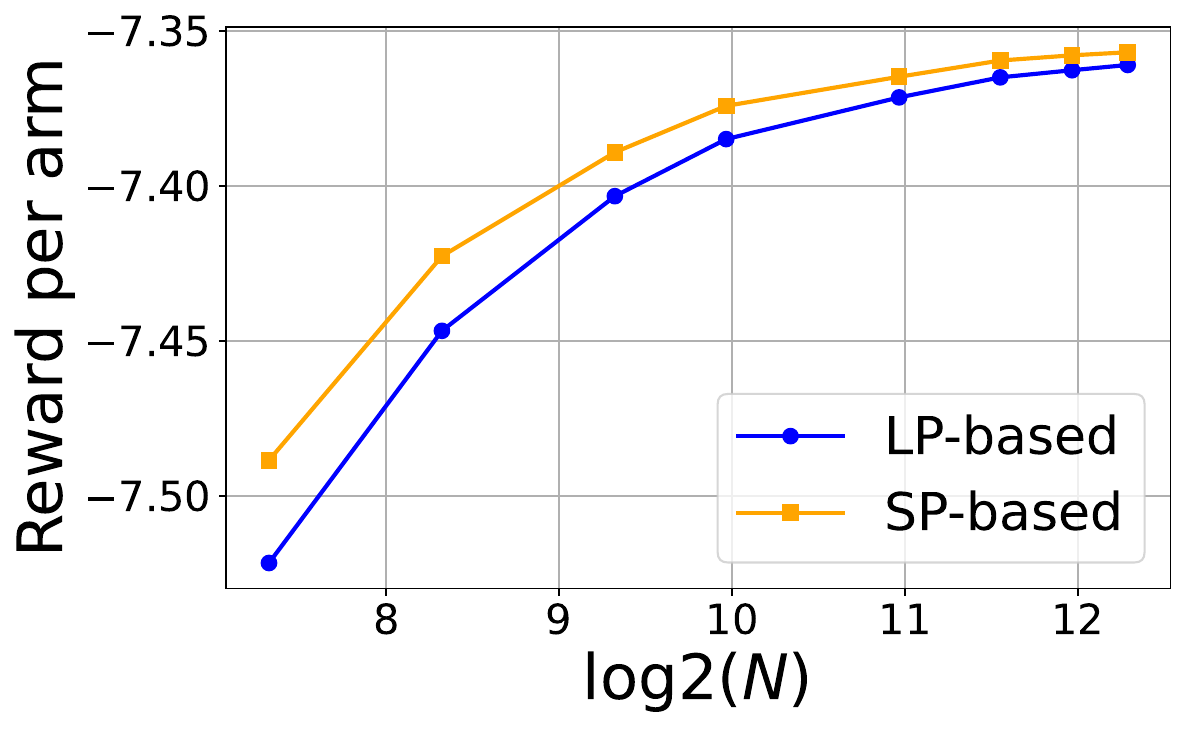}
        \caption{Reward per arm: non-unique fluid LP solution}
        \label{subfig:non-unique-1}
    \end{subfigure}

    \vspace{0.4cm}

    \begin{subfigure}[b]{0.43\linewidth}
        \centering
        \includegraphics[width=\linewidth]{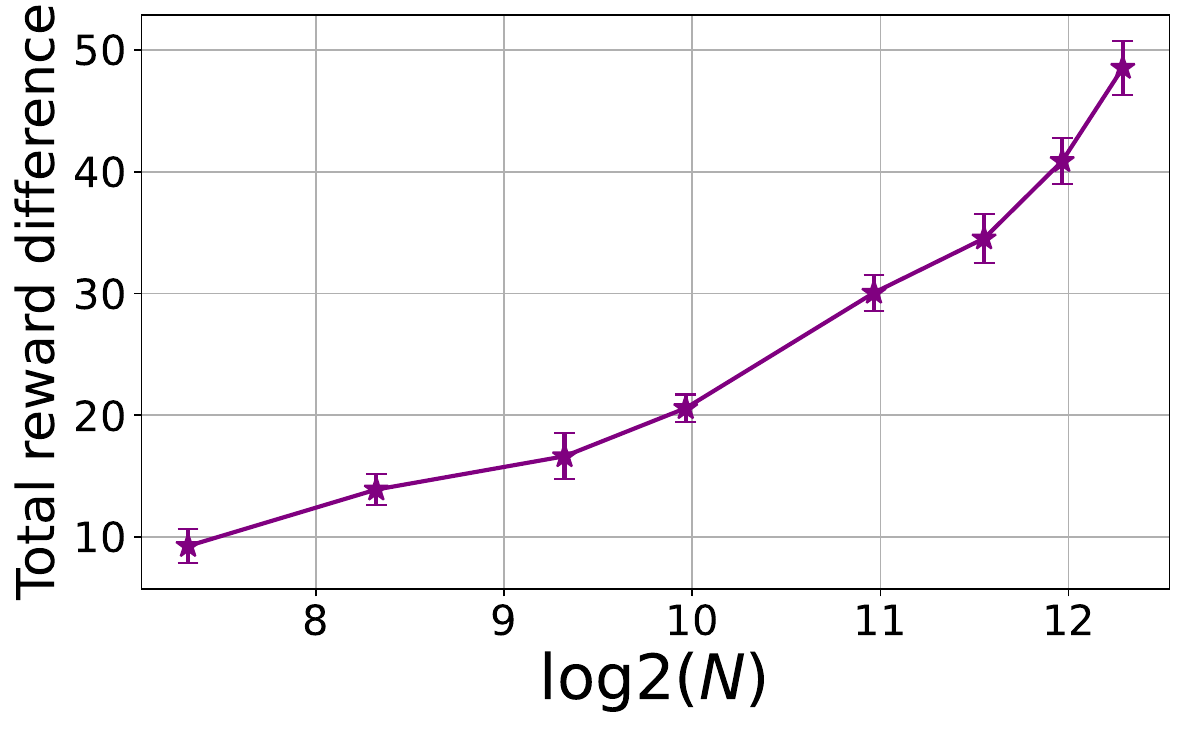}
        \caption{Total reward difference: unique fluid LP solution}
        \label{subfig:unique-2}
    \end{subfigure}
    \qquad
    \begin{subfigure}[b]{0.43\linewidth}
        \centering
        \includegraphics[width=\linewidth]{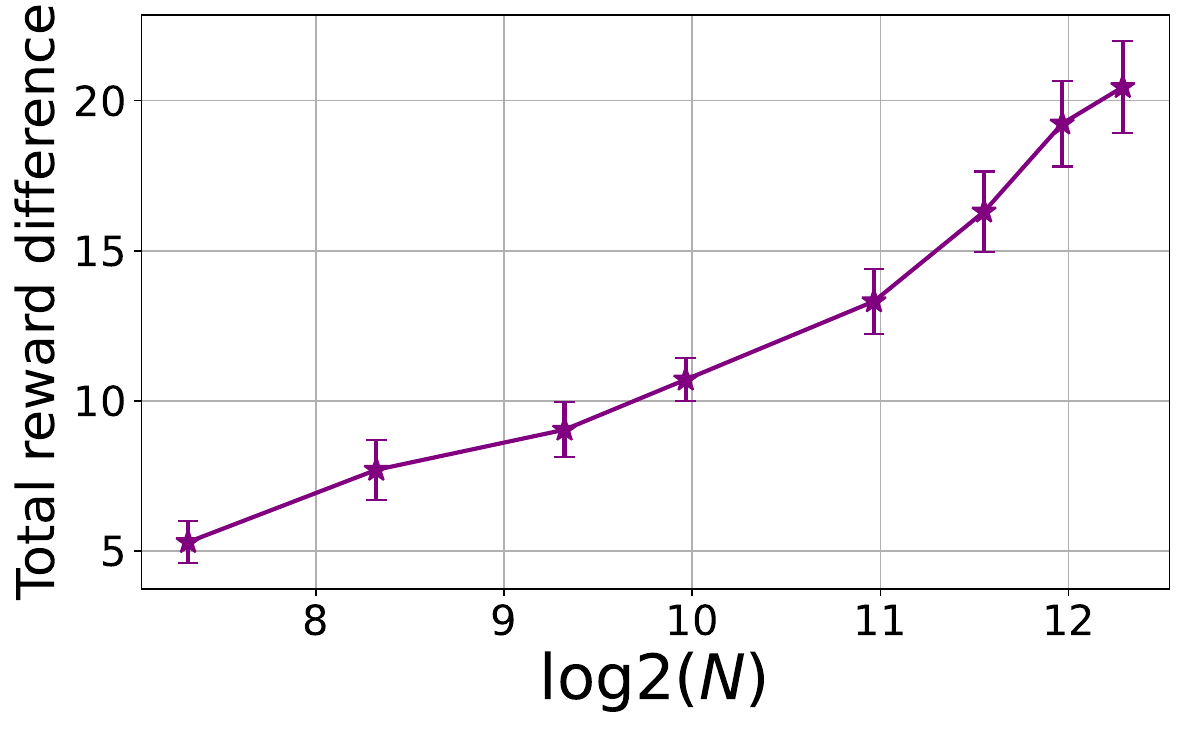}
        \caption{Total reward difference: non-unique fluid LP solution}
        \label{subfig:non-unique-2}
    \end{subfigure}
    \caption{Comparison of LP and SP-based policies on a machine maintenance example. Top row: reward per arm. Bottom row: total reward difference (SP minus LP) with 2-sigma error bars.}
    \label{fig:perf-3}
\end{figure}

In our experiments, we evaluated the performance of our SP-based policy (Algorithm~\ref{algo:SP-based}) on a machine maintenance problem \citep{glazebrook2005index,demirci2024restless}, an RMAB formulation motivated by real-world trade-offs in preventive maintenance and resource allocation. Each machine has five states, where a higher state represents a more deteriorated condition. The first state is a pristine state. Under action 1 (performing maintenance), a deteriorated machine has a high probability of returning to the pristine state. Under action 0 (not performing maintenance), the machine gradually deteriorates and has an increasing probability of breaking down as the state worsens. When a breakdown occurs, the machine must be replaced (returning to the pristine state) and incurs a high cost. The negative rewards (i.e., costs) reflect the nature of this model.

We performed two sets of experiments where the problem instances are degenerate: one set where the fluid LP solution is unique (Figures~\ref{subfig:unique-1}/\ref{subfig:unique-2}) and one set where it is not unique (Figures~\ref{subfig:non-unique-1}/\ref{subfig:non-unique-2}).
For both sets, computing the optimal policies is intractable. The resulting SP~\eqref{eq:problem-formulation-Gaussian-N-independent} is solved using the EDDP algorithm \citep[Algorithm 3]{lan2022complexity}, as discussed previously. The details of the experiments are presented in Appendix~\ref{append:numerics-and-practical}.
In each set of experiments, we evaluated the performance of our SP-based policy and the LP-based policy on a sequence of problems with increasing numbers of machines $N$, and compared the reward difference.
Figure~\ref{fig:perf-3} demonstrates that the improvement of our SP-based policy over the LP-based policy grows with $N$ in both settings.

We also evaluated the computational complexity of EDDP when varying $H$ and $S$. The results are shown in Figure~\ref{fig:EDDP-computation-time} and include 2-sigma error bars. Each recorded computation time in the figures was averaged from running EDDP on 1,000 RMAB instances, with each instance solved for 10 independent SAA realizations. The results in Figure~\ref{fig:EDDP-computation-time} clearly illustrate that the computational complexity under EDDP increases linearly with respect to the problem parameters $H$ and $S$. 

\begin{figure}[h]
    \centering
    \begin{subfigure}[b]{0.45\linewidth}
        \centering
        \includegraphics[width=\linewidth]{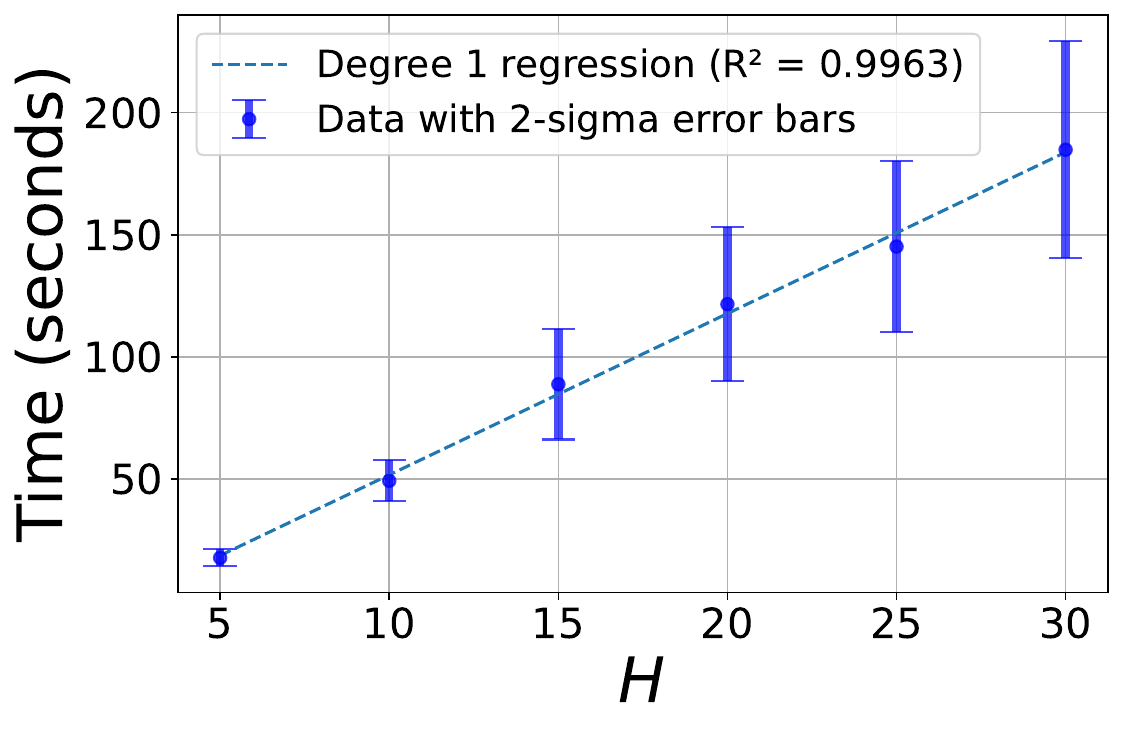}
        \caption{Fix $S=5$}
    \end{subfigure}
    \qquad
    \begin{subfigure}[b]{0.45\linewidth}
        \centering
        \includegraphics[width=\linewidth]{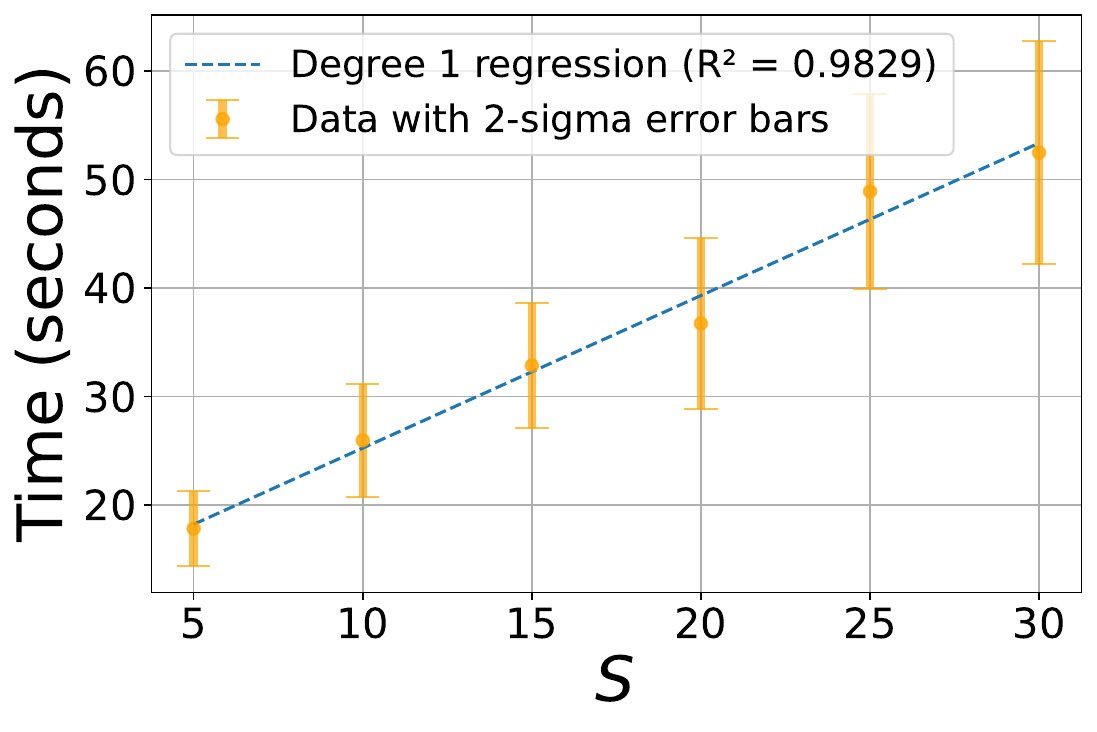}
        \caption{Fix $H=5$}
    \end{subfigure}
    \caption{The computation time of EDDP.}
    \label{fig:EDDP-computation-time}
\end{figure}
\section{Conclusion}

In this paper, we considered the finite-horizon Restless Multi-Armed Bandit (RMAB) problem and proposed a Stochastic-Programming-based (SP-based) policy, which is based on a carefully constructed Gaussian stochastic system that approximates the original system.
Our SP-based policy achieves an $\tilde{\calO}(1/N)$ optimality gap for degenerate RMABs under a uniqueness assumption, while existing LP-based policies based on the fluid approximation only guarantee a $\Theta(1/\sqrt{N})$ optimality gap.
Further, our SP-based policy can achieve an $\Omega(\sqrt{N})$ improvement over a large class of LP-based policies even without the uniqueness assumption.

\bibliography{reference}

\begin{thebibliography}{34}
\providecommand{\natexlab}[1]{#1}
\providecommand{\url}[1]{\texttt{#1}}
\expandafter\ifx\csname urlstyle\endcsname\relax
  \providecommand{\doi}[1]{doi: #1}\else
  \providecommand{\doi}{doi: \begingroup \urlstyle{rm}\Url}\fi

\bibitem[Avrachenkov et~al.(2024)Avrachenkov, Borkar, and
  Shah]{avrachenkov2024lagrangian}
Konstantin Avrachenkov, Vivek~S Borkar, and Pratik Shah.
\newblock Lagrangian index policy for restless bandits with average reward.
\newblock \emph{arXiv preprint arXiv:2412.12641}, 2024.

\bibitem[Bertsimas and Tsitsiklis(1997)]{bertsimas1997introduction}
Dimitris Bertsimas and John~N Tsitsiklis.
\newblock \emph{Introduction to linear optimization}, volume~6.
\newblock Athena Scientific Belmont, MA, 1997.

\bibitem[Borovkov(1964)]{Bor_64}
A.~Borovkov.
\newblock Some limit theorems in the theory of mass service, {I}.
\newblock \emph{Theory of Probability and its Applications}, 9\penalty0
  (4):\penalty0 550--565, 1964.

\bibitem[Borovkov(1965)]{Bor_65}
AA~Borovkov.
\newblock Some limit theorems in the theory of mass service, {II}.
\newblock \emph{Theory of Probability and its Applications}, 10\penalty0
  (3):\penalty0 375--400, 1965.

\bibitem[Brown and Smith(2020)]{Brown2020IndexPA}
David~B. Brown and James~E. Smith.
\newblock Index policies and performance bounds for dynamic selection problems.
\newblock \emph{Manag. Sci.}, 66:\penalty0 3029--3050, 2020.

\bibitem[Brown and Zhang(2023)]{brown2023fluid}
David~B Brown and Jingwei Zhang.
\newblock Fluid policies, reoptimization, and performance guarantees in dynamic
  resource allocation.
\newblock \emph{Operations Research}, 2023.

\bibitem[Demirci et~al.(2024)Demirci, Arts, and Houtum]{demirci2024restless}
Ece~Zeliha Demirci, Joachim Arts, and Geert-Jan~van Houtum.
\newblock A restless bandit approach for capacitated condition based
  maintenance scheduling.
\newblock \emph{Flexible Services and Manufacturing Journal}, pages 1--29,
  2024.

\bibitem[Gast et~al.(2023)Gast, Gaujal, and Yan]{gast2023linear}
Nicolas Gast, Bruno Gaujal, and Chen Yan.
\newblock Linear program-based policies for restless bandits: Necessary and
  sufficient conditions for (exponentially fast) asymptotic optimality.
\newblock \emph{Mathematics of Operations Research}, 2023.

\bibitem[Gittins(1979)]{Gittins79banditprocesses}
J.~C. Gittins.
\newblock Bandit processes and dynamic allocation indices.
\newblock \emph{Journal of the Royal Statistical Society, Series B}, pages
  148--177, 1979.

\bibitem[Glazebrook et~al.(2005)Glazebrook, Mitchell, and
  Ansell]{glazebrook2005index}
Kevin~D Glazebrook, HM~Mitchell, and PS~Ansell.
\newblock Index policies for the maintenance of a collection of machines by a
  set of repairmen.
\newblock \emph{European Journal of Operational Research}, 165\penalty0
  (1):\penalty0 267--284, 2005.

\bibitem[Hu and Frazier(2017)]{hu2017asymptotically}
Weici Hu and Peter Frazier.
\newblock An asymptotically optimal index policy for finite-horizon restless
  bandits.
\newblock \emph{arXiv preprint arXiv:1707.00205}, 2017.

\bibitem[Iglehart and Whitt(1970{\natexlab{a}})]{IglWhi_70}
Donald Iglehart and Ward Whitt.
\newblock Multiple channel queues in heavy traffic. {I}.
\newblock \emph{Advances in Applied Probability}, 2\penalty0 (1):\penalty0
  150--177, 1970{\natexlab{a}}.

\bibitem[Iglehart and Whitt(1970{\natexlab{b}})]{IglWhi_70_2}
Donald~L Iglehart and Ward Whitt.
\newblock Multiple channel queues in heavy traffic. {II: Sequences,} networks,
  and batches.
\newblock \emph{Advances in Applied Probability}, 2\penalty0 (2):\penalty0
  355--369, 1970{\natexlab{b}}.

\bibitem[Killian et~al.(2022)Killian, Xu, Biswas, and
  Tambe]{killian2022restless}
Jackson~A Killian, Lily Xu, Arpita Biswas, and Milind Tambe.
\newblock Restless and uncertain: Robust policies for restless bandits via deep
  multi-agent reinforcement learning.
\newblock In \emph{Uncertainty in Artificial Intelligence}, pages 990--1000.
  PMLR, 2022.

\bibitem[La~Scala and Moran(2006)]{la2006optimal}
Barbara~F La~Scala and Bill Moran.
\newblock Optimal target tracking with restless bandits.
\newblock \emph{Digital Signal Processing}, 16\penalty0 (5):\penalty0 479--487,
  2006.

\bibitem[Lan(2022)]{lan2022complexity}
Guanghui Lan.
\newblock Complexity of stochastic dual dynamic programming.
\newblock \emph{Mathematical Programming}, 191\penalty0 (2):\penalty0 717--754,
  2022.

\bibitem[Le~Ny et~al.(2006)Le~Ny, Dahleh, and Feron]{le2006multi}
Jerome Le~Ny, Munther Dahleh, and Eric Feron.
\newblock Multi-agent task assignment in the bandit framework.
\newblock In \emph{Proceedings of the 45th IEEE Conference on Decision and
  Control}, pages 5281--5286. IEEE, 2006.

\bibitem[Mate et~al.(2020)Mate, Killian, Xu, Perrault, and
  Tambe]{mate2020collapsing}
Aditya Mate, Jackson Killian, Haifeng Xu, Andrew Perrault, and Milind Tambe.
\newblock Collapsing bandits and their application to public health
  intervention.
\newblock \emph{Advances in Neural Information Processing Systems},
  33:\penalty0 15639--15650, 2020.

\bibitem[Mate et~al.(2021)Mate, Perrault, and Tambe]{mate2021risk}
Aditya Mate, Andrew Perrault, and Milind Tambe.
\newblock Risk-aware interventions in public health: Planning with restless
  multi-armed bandits.
\newblock In \emph{AAMAS}, pages 880--888, 2021.

\bibitem[Nakhleh et~al.(2021)Nakhleh, Ganji, Hsieh, Hou, Shakkottai,
  et~al.]{nakhleh2021neurwin}
Khaled Nakhleh, Santosh Ganji, Ping-Chun Hsieh, I~Hou, Srinivas Shakkottai,
  et~al.
\newblock {NeurWIN: Neural Whittle index network for restless bandits via deep
  RL}.
\newblock \emph{Advances in Neural Information Processing Systems},
  34:\penalty0 828--839, 2021.

\bibitem[Nocedal and Wright(1999)]{nocedal1999numerical}
Jorge Nocedal and Stephen~J Wright.
\newblock \emph{Numerical optimization}.
\newblock Springer, 1999.

\bibitem[Papadimitriou and Tsitsiklis(1999)]{Papadimitriou99thecomplexity}
Christos~H. Papadimitriou and John~N. Tsitsiklis.
\newblock The complexity of optimal queuing network control.
\newblock \emph{Math. Oper. Res}, pages 293--305, 1999.

\bibitem[Pereira and Pinto(1991)]{pereira1991multi}
Mario~VF Pereira and Leontina~MVG Pinto.
\newblock Multi-stage stochastic optimization applied to energy planning.
\newblock \emph{Mathematical programming}, 52:\penalty0 359--375, 1991.

\bibitem[Shapiro(2011)]{shapiro2011analysis}
Alexander Shapiro.
\newblock Analysis of stochastic dual dynamic programming method.
\newblock \emph{European Journal of Operational Research}, 209\penalty0
  (1):\penalty0 63--72, 2011.

\bibitem[Shapiro et~al.(2021)Shapiro, Dentcheva, and
  Ruszczynski]{shapiro2021lectures}
Alexander Shapiro, Darinka Dentcheva, and Andrzej Ruszczynski.
\newblock \emph{Lectures on stochastic programming: modeling and theory}.
\newblock SIAM, 2021.

\bibitem[Valiant and Valiant(2010)]{valiant2010clt}
Gregory Valiant and Paul Valiant.
\newblock A clt and tight lower bounds for estimating entropy.
\newblock In \emph{Electronic Colloquium on Computational Complexity (ECCC)},
  volume~17, page~9, 2010.

\bibitem[Whittle(1988)]{whittle-restless}
P.~Whittle.
\newblock Restless bandits: activity allocation in a changing world.
\newblock \emph{Journal of Applied Probability}, 25A:\penalty0 287--298, 1988.

\bibitem[Xiong and Li(2023)]{xiong2023finite}
Guojun Xiong and Jian Li.
\newblock Finite-time analysis of whittle index based q-learning for restless
  multi-armed bandits with neural network function approximation.
\newblock \emph{Advances in Neural Information Processing Systems},
  36:\penalty0 29048--29073, 2023.

\bibitem[Xiong et~al.(2022)Xiong, Li, and Singh]{xiong2022reinforcement}
Guojun Xiong, Jian Li, and Rahul Singh.
\newblock Reinforcement learning augmented asymptotically optimal index policy
  for finite-horizon restless bandits.
\newblock In \emph{Proceedings of the AAAI Conference on Artificial
  Intelligence}, volume~36, pages 8726--8734, 2022.

\bibitem[Yan(2024)]{yan2024optimal}
Chen Yan.
\newblock An optimal-control approach to infinite-horizon restless bandits:
  Achieving asymptotic optimality with minimal assumptions.
\newblock \emph{arXiv preprint arXiv:2403.11913}, 2024.

\bibitem[Zayas-Cab{\'a}n et~al.(2017)Zayas-Cab{\'a}n, Jasin, and
  Wang]{ZayasCabn2017AnAO}
Gabriel Zayas-Cab{\'a}n, Stefanus Jasin, and Guihua Wang.
\newblock An asymptotically optimal heuristic for general non-stationary
  finite-horizon restless multi-armed multi-action bandits.
\newblock \emph{Ross: Technology \& Operations (Topic)}, 2017.

\bibitem[Zhang(2024)]{zhang2024leveraging}
Jingwei Zhang.
\newblock Leveraging nondegeneracy in dynamic resource allocation.
\newblock \emph{Available at SSRN}, 2024.

\bibitem[Zhang and Frazier(2021)]{zhang2021restless}
Xiangyu Zhang and Peter~I Frazier.
\newblock Restless bandits with many arms: Beating the central limit theorem.
\newblock \emph{arXiv preprint arXiv:2107.11911}, 2021.

\bibitem[Ziegler(2012)]{ziegler2012lectures}
G{\"u}nter~M Ziegler.
\newblock \emph{Lectures on polytopes}, volume 152.
\newblock Springer Science \& Business Media, 2012.

\end{thebibliography}
\bibliographystyle{plainnat}

\newpage

\appendix

\begin{table}[H]
    \centering
    {\small
    \begin{tabular}{|c|c|}
    \hline
         $\mathbb{N}$ & the set of natural numbers \\
         \hline
         $\mathbb{R}$ & the set of real numbers \\
         \hline
         $N$ & the number of arms \\
         \hline
         $H$ & the number of steps in the finite-horizon RMAB\\
         \hline
         $\cal S$ & the set of states \\
         \hline
         $ \cal A$& the set of actions\\
         \hline
         $r_h(s,a)$ & the reward of taking action $a$ in state $s$ at step $h$\\
         \hline
           $\rmax$ & the maximum reward $\max_{s,a,h} r_h(s,a)$\\
           \hline
             $\pp_h(s'_n|s_n,a_n)$ & the probability that arm $n$ is in state $s'_n$ after taking action $a_n$ in state $s_n$ at step $h$\\
         \hline
         $\bX_h$ and $\tilde{\bX}_h$ &  the states at step $h$ in the $N$-system and the Gaussian stochastic system, respectively \\
         \hline
        $\bY_h$ and $\tilde{\bY}_h$ &  the actions at step $h$ in the $N$-system and the Gaussian stochastic system, respectively \\
         \hline
         $\Delta^S$ & the probability simplex of dimension \(S\) \\
         \hline
 $\simpXN$ & the discrete subset of $\bx \in \Delta^S$ for which $N \bx \in \bbN^S$  \\
 \hline
${\calX}_{z_h\Ndelta}$ & a neighborhood of state $\bx_h^*:$ $\left\{\bx \in \Delta^S: \|\bx - \bx^*_h\|_\infty\leq z_h\Ndelta\right\}$ \\
\hline
${\cal Y}_{\kappa z_h\Ndelta}$ &  a neighborhood of action $\by^*_h:$ $\left\{\by_h \in \Delta^{2S}: \|\by_h - \by^*_h\|_\infty\leq \kappa z_h\Ndelta, \; \sum_{s=1}^{S} y_h(s,1) = \alpha \right\}$ \\
\hline
${\cal Y}_{\bx}$ &  feasible actions for state $\bx:$ $\left\{\by_h \in \Delta^{2S}: \; \by_h(\cdot,1)+\by_h(\cdot,0)=\bx, \; \sum_{s=1}^{S} y_h(s,1) = \alpha \right\}$ \\
\hline
 $\Gamma_{h}(\by^*_h)$ & the covariance matrix of Gaussian random vector $\bZ_h$, defined in \eqref{eq:Gamma-definition} \\ 
 \hline
 ${\cal Y}^s_{\kappa z_h\Ndelta}$ &  the set $\left\{\by_h(s,\cdot) \in \mathbb{R}: \|\by_h(s,\cdot)-\by^*_h(s,\cdot)\|_\infty\leq \kappa z_h \Ndelta\right\}$ \\
\hline
${\cal Y}^s_{\bx}$ & the set $\left\{\by_h(s,\cdot) \in \mathbb{R}: \by_h(s,\cdot)\geq \bzero, \; y_h(s,1) + y_h(s,0) = x(s)\right\}.$\\
\hline
$\Ndelta$ & $\Ndelta=2 \log N/\sqrt N$  \\
\hline
   $\VNopt(\xini,1)$ & the optimal solution to the $N$-system problem (or the value function of the $N$-system)\\
         \hline
                 $V^N_\pi(\x_h,h)$ & the value function of policy $\pi$ in the $N$-system\\
         \hline
                          $Q^N_\pi(\x_h,\by_h,h)$ & the Q-function of policy $\pi$ in the $N$-system\\
         \hline
         $\pi^{N,*}$ & the locally-optimal policy for the $N$-system\\
         \hline
            $\VLP(\xini,1)$ & the optimal solution to the fluid LP (or the value function of the fluid LP)\\
         \hline
         $\by^*$ & an optimal solution of the fluid LP \\
         \hline
                 $\tilde{V}^N_\pi(\x_h,h)$ & the value function of policy $\pi$ in the Gaussian stochastic system\\
         \hline
                          $\tilde{Q}^N_\pi(\x_h,\by_h,h)$ & the Q-function of policy $\pi$ in the Gaussian stochastic system\\
         \hline
            $\tilde{\pi}^{N,*}$ & the locally-SP-optimal policy\\
         \hline
    \end{tabular}
    }
    \vspace{0.2cm}
    \caption{Notation Table}
    \label{tab:notation}
\end{table}

\section{Lists of frequently used notations and constants}  \label{append:list-of-notations-and-constants}

Frequently used notations are included in Table \ref{tab:notation}. Below is a list of frequently used constants that are {\em independent} of $N$:
\begin{align*}
  \sigma :&  = \inf_{\substack{\by\in \calU \\ \by\neq\by^{*}}} \frac{\br \bigl(\by^{*} - \by\bigr)^\top}{\norminf{\by^{*} - \by}}. & \quad \text{first defined in \eqref{eq:sigma-b-min-expression-I}}; \\
  \kappa &> \max \left\{ 1 + 6S, \, 3 + 2 \rmax HS/\sigma \right\}; \\
  L_h &:= \text{Lipschitz constant of $\phi_h(\cdot)$ satisfying $L_h \le 2S$}, & \quad \text{first defined in \eqref{ea:function-phi}}; \\
  z_1 &:= 1, \; z_{h+1} = \sqrt{S} \ (\kappa L_h z_h + 1), \; 1 \le h \le H-1, & \quad \text{first defined in \eqref{eq:recursive-relation}}.
\end{align*}

\section{Rounding procedure}
\label{app:details-Gaussian}

Throughout this paper, we represent systems and controls as ``fractions'' of the total population $N$ of arms, using vectors such as $\bx$, $\by$. In the unnormalized-scale system, the coordinates of $N \bx$ and $N \by$ therefore need to be integers. If these quantities are not integers, they are implicitly understood to be appropriately rounded via some \emph{rounding procedure}. For example, after obtaining $\by$ from the SP-based policy, we need to applying rounding before using it for the $N$-system.

Formally, a rounding procedure is to solve the following problem: Given an integer $N \in \mathbb{N}$ and $\bX \in \simpXN$ so that $N \bX$ have integer coordinates, then, for any $\bY \in \calY_{\bX}$,   
we need to construct $\bY^N = \round(\bY)\in \calY_{\bX} $ for which $N \bY^N$ have integer coordinates.

This can be accomplished as follows. Set
\begin{equation}   \label{eq:rounding-1}
  Y^N(s,1) := \frac{\lfloor N Y(s,1) \rfloor}{N}, \  Y^N(s,0) := \frac{\lceil N Y(s,0) \rceil}{N}, \mbox{ for $1 \le s \le S-1$};
\end{equation}
\begin{equation}  \label{eq:rounding-2}
  Y^N(S,1) := \alpha - \sum_{s=1}^{S-1} Y^N(s,1), \ Y^N(S,0) := 1 - \alpha - \sum_{s=1}^{S-1} Y^N(s,0).
\end{equation}
It is then easy to verify that, $\bY^N := \round(\bY)$ defined in \eqref{eq:rounding-1}-\eqref{eq:rounding-2} indeed satisfies all these required conditions. Furthermore, by construction, it holds true that 
\begin{equation} \label{eq:rounding-error-bound}
  \norminf{\bY - \round(\bY)} \le \frac{1}{N} = \calO(\frac{1}{N}),
\end{equation}
independently of $\bX$ and $\bY \in \calY_{\bX}$.

\section{Proof of Theorem~\ref{thm:global}}
\label{appendix:proof-global-optimality}

\GlobalOptimality*

As stated in the proof outline, the proof of Theorem~\ref{thm:global} is based on the following three results.

\begin{restatable}[Proximity of optimal policy]{lemma}{SufficientCondition}  \label{prop:sufficient-condition-of-assumption}
Under Assumption~\ref{ass:opt-lp-distance}, let $\by^*$ be the unique optimal solution to $\VLP(\xini,1)$ at $h=1$. Then the $N$-system has an optimal policy $\piNopt$ in the policy class $\Pi_{\Ndelta}(\by^*)$, i.e., for any $\bx_h$ with $\|\bx_h-\bx^*_h\|_\infty\leq z_h\Ndelta$, we have $\|\piNopt(\bx_h,h) - \by_h^*\|_\infty\leq \kappa z_h\Ndelta$.  
\end{restatable}

In other words, under Assumption~\ref{ass:opt-lp-distance}, for all $1 \le h \le H$, an optimal control for the $N$-system is guaranteed to lie within $\kappa z_h \Ndelta$-neighborhood of the corresponding $\by^*_h$.

The next two lemmas are very similar. The difference is that Lemma \ref{lem:vd} compares the performance of the {\em locally-SP-optimal} policy in the {\em Gaussian stochastic} system with the {\em locally-optimal} policy in the {\em $N$-system}.  Lemma \ref{lem:ed} compares the performance of the locally-SP-optimal policy, $\tilde{\pi}^{N,*}$, in two different systems (Gaussian and $N$-system).

\begin{restatable}[Value difference]{lemma}{ValueDifference} \label{lem:vd}
Consider any time step $h$ with $1\le h\le H$.
For any $\bx$ such that $\bx \in{\cal X}_{z_h\Ndelta} \bigcap \simpXN$, and any $\by_h$ such that $\by_h \in{\cal Y}_{\bx} \bigcap {\cal Y}_{\kappa z_h\Ndelta} \bigcap \simpYN$,
the Q-function and value function differences between the Gaussian stochastic system and the $N$-system under their respectively optimal policies can be bounded as
\begin{align*}
\left|\tilde{Q}^N_{{\tilde{\pi}^{N,*}}}(\bx, \by_h, h)-Q^N_{\pi^{N,*}}(\bx, \by_h, h)\right| = & \ \tON,\\ 
\left|\tilde{V}^N_{{\tilde{\pi}^{N,*}}}(\bx, h)-V^N_{\pi^{N,*}}(\bx, h)\right| = & \ \tON.
\end{align*} 
\end{restatable}

\begin{restatable}[Evaluation difference]{lemma}{EvaluationDifference}  \label{lem:ed}
Consider any time step $h$ with $1\le h\le H$.
For any $\bx$ such that $\bx \in{\cal X}_{z_h\Ndelta} \bigcap \simpXN$, and any $\by_h$ such that $\by_h \in{\cal Y}_{\bx} \bigcap {\cal Y}_{\kappa z_h\Ndelta} \bigcap \simpYN$,
the Q-function and value function differences between the Gaussian stochastic system and the $N$-system under the same policy $\tilde{\pi}^{N,*}$ can be bounded as
\begin{align*}
\left|\tilde{Q}^N_{{\tilde{\pi}^{N,*}}}(\bx, \by_h, h)-Q^N_{\tilde{\pi}^{N,*}}(\bx, \by_h, h)\right| = & \ \tON. \\ 
\left|\tilde{V}^N_{{\tilde{\pi}^{N,*}}}(\bx, h)-V^N_{\tilde{\pi}^{N,*}}(\bx, h)\right| = & \ \tON.
\end{align*} 
\end{restatable}

\begin{proof}[Proof of Theorem~\ref{thm:global}]
By Lemma~\ref{prop:sufficient-condition-of-assumption}, an optimal policy for the $N$-system is in $\pidelta$, so that $\VNopt(\xini, 1) = V^N_{\pi^{N,*}}(\xini, 1)$ for a policy $\pi^{N,*} \in \pidelta$. We then have 
\begin{align*}
& \ \left|V^N_{\pi^{N,*}}(\xini, 1) - {V}^N_{{\tilde{\pi}^{N,*}}}(\xini, 1)\right|\\
\leq & \ \left|V^N_{\pi^{N,*}}(\xini, 1)-\tilde{V}^N_{{\tilde{\pi}^{N,*}}}(\xini, 1)\right| + \left|\tilde{V}^N_{\tilde{\pi}^{N,*}}(\xini, 1)-{V}^N_{{\tilde{\pi}^{N,*}}}(\xini, 1)\right| = \tON,
\end{align*}    
where the last equality holds due to Lemmas~\ref{lem:vd} and \ref{lem:ed}.
\end{proof}

\subsection{Proof of Lemma~\ref{prop:sufficient-condition-of-assumption} (Proximity of optimal policy)}

\paragraph{LP that starts at time $h$.}
Before we begin the proof, we first introduce the following LP that is used throughout the proof.
For any $h$ with $1 \le h \le H$ and $\bx_h \in \calX_{z_h \Ndelta}$, consider the LP that starts with initial condition $\bx_h$ and time spanning from $h$ to $H$:
\begin{alignat}{2}
    & \quad \QLP(\bx_h,\by_h, h) :=  \br_{h} \by_{h}^\top+ \underset{ \by_{h'} }{\mathrm{max}} \quad   \sum_{h'=h+1}^{H}  \br_{h'} \by_{h'}^\top \label{eq:problem-formulation-LP-stage-h-Q} \\
     \text{s.t.} & \quad \sum_{s} y_{h'} (s,1) = \alpha,\quad h\le h'\le H,\nonumber\\
     &\quad \by_{h'}(\cdot, 0) + \by_{h'}(\cdot,1) = \bx_{h'}(\cdot) \quad  \by_{h'} \ge {\bf 0},\quad h\le h'\le H, \nonumber  \\
    & \quad \bx_{h'+1}(\cdot)  = \sum_{s,a} y_{h'}(s,a) \mathbf{P}_{h'} \left( \cdot \mid s, a \right),
    \quad h\le h'\le H-1,\nonumber
\end{alignat}

\begin{alignat}{2}
    & \quad \VLP(\bx_h,h) :=  \max_{\by_h} \QLP(\bx_h, \by_h, h) \label{eq:problem-formulation-LP-stage-h-V} \\
     \text{s.t.} & \quad \sum_{s} y_{h} (s,1) = \alpha,\\
     &\quad \by_{h}(\cdot, 0) + \by_{h}(\cdot,1) = \bx_{h}(\cdot),\quad \by_{h} \ge {\bf 0}. \nonumber  
\end{alignat}

Under Assumption~\ref{ass:opt-lp-distance}, let $\by^*$ be the unique optimal solution to $\VLP(\xini,1)$ at $h=1$. Then by the principle of optimality, for any $h$ with $1 \le h \le H$ and initial state $\bx_h^*$ given by $\bx_h^* = \by_h^*(\cdot, 0) + \by_h^*(\cdot,1)$, the LP in \eqref{eq:problem-formulation-LP-stage-h-V} has a unique optimal solution.

We prove Lemma~\ref{prop:sufficient-condition-of-assumption} by contradiction. 
Suppose that: 
\begin{center}
  For some $h$ with $1 \le h \le H$ and some $\bx_{h} \in \simpXN$ such that $\|\bx_{h} - \bx^*_{h}\|_\infty\leq z_{h} \Ndelta$, we have $\|\piNopt(\bx_{h},h) - \by_{h}^*\|_\infty > \kappa z_{h}\Ndelta$. 
\end{center}
To ease the notation in what follows let us denote by $\by_h := \piNopt(\bx_h,h)$. 

\paragraph{Intuition of the proof plan.}
We first explain the intuition before presenting the rigorous proof.  
\begin{itemize}
    \item Suppose that $\by_h$ and $\by^*_h$ are both feasible actions for $\bx_h^*.$ Note that given $\by_h$ is $\Omega\left(\frac{\log N}{\sqrt{N}}\right)$ away from $\by^*_h$ and $\by^*_h$ is the optimal action at step $h$ for the LP \eqref{eq:problem-formulation-LP-stage-h-V} with $\bx_h^*$ as the initial condition, the Q-value \eqref{eq:problem-formulation-LP-stage-h-Q} of the LP with $\by_h$ has to be $\Omega\left(\frac{\log N}{\sqrt{N}}\right)$ smaller than that under $\by_h^*,$ or 
    \begin{equation}
     \QLP(\bx^*_h, \by_h^*,h)-\QLP(\bx^*_h, \by_h, h)= \Omega\left(\frac{\log N}{\sqrt{N}}\right).   \label{eq:QLP-1}
    \end{equation}
    \
    \item Next suppose that $\by_h$ and $\by^*_h$ are both feasible actions for $\bx_h.$ Since $\bx_h$ is at most $z_h\Ndelta$ away from $\bx_h^*,$ their value functions would differ by ${\cal O}(z_h\Ndelta) ={\cal O}\left(\frac{\log N}{\sqrt{N}}\right).$ By choosing a proper $\kappa,$ we can make sure that the Q-function differences, when we change the initial condition from $\bx_h^*$ to $\bx_h,$ do not offset the Q-function difference of taking different actions with initial condition $\bx^*_h$ (given in \eqref{eq:QLP-1}). Therefore, we can have
    \begin{equation}
    \QLP(\bx_h, \by_h^*,h)-\QLP(\bx_h, \by_h, h)= \Omega\left(\frac{\log N}{\sqrt{N}}\right).    \label{eq:QLP-2}
    \end{equation}
    
    \item If we could show that the Q-function of the LP differs ${\cal O}\left(\frac{1}{\sqrt{N}}\right)$ from the $N$-system, then \eqref{eq:QLP-2} implies 
    \begin{equation}
      Q^N(\bx_h,\by^*_h,h)-Q^N(\bx_h,\by_h,h)= \Omega\left(\frac{\log N}{\sqrt{N}}\right).
    \end{equation} so $\by_h$ cannot be an optimal action, which leads to the contradiction. 

    \item We note that in general, $\by_h$ is not a feasible action for $\bx^*_h$ and $\by_h^*$ is a not a feasible action for $\bx_h.$ Therefore, to rigorously prove the statements above, we need to construct $\by_h'$ that has a minimum difference from $\by_h$ but is feasible for $\bx^*_h,$ and $\by_h^{*,'}$ that has a minimum difference from $\by_h^*$ but is feasible for $\bx_h.$ To achieve this, we first lift the state space to $S+1$ dimension and consider $\bx^+$ whose first $S$ components are $\min\{\bx_h, \bx_h^*\}$ (component-wise) and construct $\by^+_h$ and $\by^{*,+}_h$ for the lifted system, and then project the lifted actions back to the original system. During the lifting and projection, each action component will be changed by at most $z_h\Ndelta,$ and the change of the Q-function can be bounded by a similar amount as well in our construction. Then inequalities \eqref{eq:QLP-1} and \eqref{eq:QLP-2} would hold when we replace the infeasible actions with the feasible ones. Therefore, we obtain an action that is better than $\by_h$ with initial condition $\bx_h$ in the $N$-system, which would complete the proof. 
\end{itemize}

\paragraph{Modified LP with lifted state and action.}
We introduce the following modified LP, where its state $\bx^+$ has an additional $(S+1)$-th coordinate and belongs to the simplex $\Delta^{S+1}$. For $1 \le s \le S$, $1 \le h \le H$ and $a = 0, 1$:
\begin{itemize}
  \item $r^+_h(s,a) = r_h(s,a), \; r^+_h(S+1, a) = \rmax$,
  \item $\pp^+_h(s' \mid s, a) = \pp_h(s' \mid s,a)$,
  \item $\mathbf{P}^+_{h} \left( S+1 \mid S+1, a \right)=1$, $\mathbf{P}^+_{h} \left( S+1 \mid s, a \right)=0$,
\end{itemize} 
i.e., the newly added state has the maximum reward independent of the action and the state is an isolated state. We define $\QLP^+$ and $\VLP^+$ similar to \eqref{eq:problem-formulation-LP-stage-h-Q} and \eqref{eq:problem-formulation-LP-stage-h-V} for this modified LP.

\paragraph{State and action mappings.} 
We further have the following mappings of the state and action vector.
\begin{itemize}
\item  Given $\bx_h, \bx^*_h \in \Delta^S$, we construct $\bx^+_h \in \Delta^{S+1}$ such that $x^+_h(s) := \min\{x_h(s), x^*_h(s)\}$ for $1\leq s\leq S$ and $x^+_h(S+1) := 1 - \sum_{s=1}^S x^+_h(s)$.
\item Given $\by_h$ for $\bx_h$, we construct $\by^+_h \in \Delta^{2(S+1)}$ such that 
       \begin{equation*}
         y^+_h(s,a) := \frac{y_h(s,a)}{\sum_a y_h(s,a)} x^+_h(s)
       \end{equation*}
       for $1\leq s\leq S$ and $y^+_h(S+1,a) := \sum_{s=1}^S \left(y_h(s,a) - y^+_h(s,a)\right)$. We use the convention that $0/0 = 0$. We verify that effectively $\by^+_h$ belongs to $\Delta^{2(S+1)}$, since from our construction $0 \le y^+_h(s,a) \le y_h(s,a)$ for $1 \le s \le S$, so $y^+_h(S+1,a) \ge 0$. Furthermore, we note that
       \begin{equation*}
         \sum_{s=1}^{S+1} y^+_h(s,1) = \sum_{s=1}^{S} y_h(s,1) = \alpha.
       \end{equation*}
       
\item Similarly, given $\by^*_h$ for $\bx^*_h$, we construct $\by_h^{*,+} \in \Delta^{2(S+1)}$ such that 
     \begin{equation*}
       y^{*,+}_h(s,a) := \frac{y^*_h(s,a)}{\sum_a y^*_h(s,a)} x^{+}_h(s)
     \end{equation*}
    for $1\leq s\leq S$ and $y^{*,+}_h(S+1,a) := \sum_{s=1}^S \left(y^*_h(s,a) - y^{*,+}_h(s,a)\right)$.
    
\item Given $\by^+_h \in \Delta^{2(S+1)}$ and $\bx^*_h \in \Delta^S$, we construct $\by'_h \in \Delta^{2S}$ recursively over $s$ as follows: Define $\theta := \sum_{s=1}^S \left(y_h(s,1) - y^+_h(s,1)\right)$. Starting from $s=1$ to $s=S$, repeat the following steps:
\begin{itemize}
    \item Step 1: Set 
    \begin{align*}
    y'_h(s,1) &:= y^+_h(s,1) + \min\left\{x^*_h(s) - x^+_h(s), \; \theta\right\}; \\
    y'_h(s,0) &:= y^+_h(s,0) + (x^*_h(s) - x^+_h(s)) - \min\left\{x^*_h(s) - x^+_h(s), \theta\right\}\\
    &= x_h^*(s)-y'_h(s,1).
    \end{align*}
    \item Step 2: Set $s \leftarrow s+1$ and $\theta\leftarrow \theta - \min \left\{x^*_h(s)-x^+_h(s), \theta\right\},$ and repeat step 1.      
\end{itemize}
    We verify that by our construction,
    \begin{equation*}
      \sum_{s=1}^{S} y'_h(s,1) = \sum_{s=1}^{S} y^+_h(s,1) + \sum_{s=1}^S \left(y_h(s,1) - y^+_h(s,1)\right) = \alpha.
    \end{equation*}

\item Similarly, given $\by^{*,+}_h \in \Delta^{2(S+1)}$ and $\bx_h \in \Delta^S$, we construct $\by^{*,'}_h \in \Delta^{2S}$ recursively over $s$ as follows: Define $\theta := \sum_{s=1}^S y_h^*(s,1) - y^+_h(s,1)$. Starting from $s=1$ to $s=S$, repeat the following steps:
\begin{itemize}
    \item Step 1: Set 
    \begin{align*}
    y^{*,'}_h(s,1) &:=  y_h^{*,+}(s,1) + \min\left\{x_h(s) - x^+_h(s), \theta\right\}; \\
    y^{*,'}_h(s,0) &:=  y_h^{*,+}(s,0) + (x_h(s) - x^+_h(s)) - \min\left\{x_h(s) - x^+_h(s), \theta\right\} \\
    &= x_h(s) - y'_h(s,1).
    \end{align*}
    \item Step 2: Set $s \leftarrow s+1$ and $\theta \leftarrow \theta - \min\left\{x_h(s) - x^+_h(s), \theta\right\}$, and repeat step 1.  
\end{itemize}
\end{itemize}

Note that we can view $\by'_h$ as an action mapping from $\by_h$ for state $\bx_h$ to an action that is feasible for $\bx_h^*$, and likewise $\by^{*,'}_h$ is an action mapping from $\by^*_h$ for state $\bx^*_h$ to an action that is feasible for $\bx_h$.

We note that the changes we made from $\by_h$ to $\by^+_h$ is at most $\|\bx_h - \bx^*_h\|_\infty$ component-wise for the first $S$ components, and the same holds from $\by^+_h$ to $\by'_h$. Indeed,
\begin{equation*}
  \abs{y^+_h(s,a) - y_h(s,a)} = \frac{y_h(s,a)}{x_h(s)} \abs{x^+_h(s) - x_h(s)} \le \abs{x^+_h(s) - x_h(s)} \le \abs{x^*_h(s) - x_h(s)},
\end{equation*}
and
\begin{equation*}
  \abs{y^+_h(s,a) - y'_h(s,a)} \le \abs{x^*_h(s) - x^+_h(s)} \le \abs{x^*_h(s) - x_h(s)}.
\end{equation*}
Consequently, 
\begin{equation*}
  \norminf{\by_h - \by'_h} \le 2 \norminf{\bx_h - \bx^*_h}. 
\end{equation*}

Since $\|\by_h-\by^*_h\|_\infty>\kappa z_h\Ndelta$ and $\|\bx_h - \bx^*_h\|_\infty\leq z_h\Ndelta$, we deduce that
\begin{equation*}
  \|\by'_h - \by^*_h\|_\infty>(\kappa-2) z_h\Ndelta.
\end{equation*}

\paragraph{Properties of the mappings.}
We establish some properties of the mappings in the following claim, whose proof is given in Section~\ref{sec:proof-claim:properties-mappings}.

\begin{claim}\label{claim:properties-mappings}
The modified LP and the state and action mappings satisfy the following inequalities.
\begin{enumerate}[label=(\roman*)]
    \item 
    \begin{align}
    & \QLP(\bx_h, \by_h, h) \leq \QLP^+(\bx^+_h, \by^+_h, h), \label{eq:upper-yh}\\
    & \QLP(\bx^*_h, \by^*_h,h) \leq \QLP^+(\bx^+_h, \by^{*,+}_h,h).\label{eq:upper-yhstar}
    \end{align}
    \item 
    \begin{align}
     & \QLP(\bx_h, \by_h, h) \geq \QLP^+(\bx^+_h, \by^+_h, h)-r_{\max}HSz_h\Ndelta, \label{eq:lower-yh}\\
     & \QLP(\bx^*_h, \by^*_h,h)\geq \QLP^+(\bx^+_h, \by^{*,+}_h,h)-r_{\max}HSz_h\Ndelta.\label{eq:lower-yhstar}
    \end{align}
    \item 
    \begin{align}
     & \QLP(\bx_h, \by^{*,'}_h, h) \geq \QLP^+(\bx^+_h, \by^{*,+}_h, h)-r_{\max}HSz_h\Ndelta,\label{eq:lower-yhstarprime} \\
     & \QLP(\bx^*_h, \by'_h,h)\geq \QLP^+(\bx^+_h, \by^{+}_h,h)-r_{\max}HSz_h\Ndelta.\label{eq:lower-yhprime}
    \end{align}
\end{enumerate}
\end{claim}
 
\paragraph{Improved LP solution.}
The goal of this part is to show $\QLP(\bx_h, \by^{*,'}_h, h) - \QLP(\bx_h, \by_h, h) =\Theta \left( \frac{\log N}{\sqrt{N}} \right)$, i.e., the action $\by^{*,'}_h$ improves the LP solution from $\by_h$.

Using the established properties in Claim~\ref{claim:properties-mappings}, we have
\begin{align*}
   &\mspace{23mu}\QLP(\bx_h, \by^{*,'}_h, h) - \QLP(\bx_h, \by_h, h) \\
   &\ge \QLP^+(\bx^+_h, \by^{*,+}_h, h)-r_{\max}HSz_h\Ndelta-\QLP^+(\bx^+_h, \by^+_h, h)\\
   &\ge \QLP(\bx^*_h, \by^*_h,h)-r_{\max}HSz_h\Ndelta-\left(\QLP(\bx^*_h, \by'_h,h)+r_{\max}HSz_h\Ndelta\right)\\
   &\geq \QLP(\bx^*_h, \by_h^*, h)  - \QLP(\bx^*_h, \by'_h, h) - 2r_{\max} H S z_h\Ndelta.
\end{align*}

To further lower bound $\QLP(\bx^*_h, \by_h^*, h)  - \QLP(\bx^*_h, \by'_h, h)$, we utilize the following claim, which establishes an ``anti-Lipschitz'' property of the function $\QLP(\bx^*_h, \cdot, h)$.
The proof of this claim is given in Section~\ref{sec:anti-lip}.

\begin{claim}\label{claim:anti-Lipschitz}
For any $h$ with $1\le h\le H$, there exists a positive constant $\sigma > 0$ that depends on $\bx^*_h$ but not on $N$, such that for any $\overline{\by}_h \in \calY_{\bx^*_h}$, it holds that
\begin{equation}  \label{eq:anti-lipschitz}
  \QLP(\bx_h^*,\by^*_h,h)-\QLP(\bx_h^*,\overline{\by}_h,h)\geq \sigma \|\by^*_h-\overline{\by}_h\|_{\infty}.
\end{equation}
\end{claim}

Applying Claim~\ref{claim:anti-Lipschitz} to $\overline{\by}_h=\by'_h \in \calY_{\bx^*_h}$, we have
\begin{equation*}
  \QLP(\bx_h^*,\by^*_h,h)-\QLP(\bx_h^*,\by'_h,h) \geq \sigma \|\by^*_h-\by'_h\|_{\infty}\geq \sigma(\kappa-2) z_h \Ndelta .
\end{equation*} 
Therefore,
\begin{align}
   &\mspace{23mu}\QLP(\bx_h, \by^{*,'}_h, h) - \QLP(\bx_h, \by_h, h)\nonumber \\
   &\geq \QLP(\bx^*_h, \by_h^*, h)  - \QLP(\bx^*_h, \by'_h, h) - 2r_{\max} H S z_h\Ndelta\nonumber\\
   &\ge \sigma(\kappa-2) z_h \Ndelta- 2r_{\max} H S z_h\Ndelta\nonumber\\
    &=  \frac{2(\sigma \kappa-2\sigma -2r_{\max}HS) z_h \log N}{\sqrt{N}} = \Theta \left( \frac{\log N}{\sqrt{N}} \right),\label{eq:lower-LP}
\end{align} 
as long as 
\begin{equation*}  \label{eq:condition-on-kappa}
   \kappa \ge \frac{2r_{\max}HS}{\sigma} + 3.
\end{equation*}

\paragraph{Improved policy for $N$-system and contradiction.}
We now prove that for the $N$-system with state $\bx_h$ at time $h$, the action $\round(\by^{*,'}_h)$ strictly improves over the action $\by_h$.
However, recall that the action $\by_h := \piNopt(\bx_h,h)$ is given by the optimal policy for the $N$-system, which thus leads to a contradiction.

To prove the improvement of $\round(\by^{*,'}_h)$ over $\by_h$, we will lower bound $Q^{N}(\bx_h, \round(\by^{*,'}_h), h)-Q^{N}(\bx_h, \by_h, h)$ by relating the Q-function of the $N$-system to the Q-function of the LP.
In particular, we establish the following claim, whose proof is given in Section~\ref{sec:proof-claim:Q-relation}.
\begin{claim}\label{claim:Q-relation}
    The Q-function of the $N$-system and the Q-function of the LP satisfy that
    \begin{align}
    \QLP(\bx_h, \by^{*,'}_h, h) - Q^{N}(\bx_h, \round(\by^{*,'}_h), h) &\le \frac{C_h}{\sqrt{N}},\label{eq:fluid-and-N-sys-gap}
    \end{align}
    for some positive constant $C_h>0$ independent of $N$. 
    In addition, for any $\overline{\by}_h\in \calY_{\bx_h} \cap \simpYN$,
    \begin{equation}\label{eq:QLP-larger-QN}
        \QLP(\bx_h,\overline{\by}_h,h) \ge Q^N(\bx_h,\overline{\by}_h,h).
    \end{equation}
\end{claim}

With Claim~\ref{claim:Q-relation}, we have
\begin{align}
    &\mspace{23mu}Q^{N}(\bx_h, \round(\by^{*,'}_h), h) - Q^{N}(\bx_h, \by_h, h)\\
    &\ge \QLP(\bx_h, \by^{*,'}_h, h) - \frac{C_h}{\sqrt{N}} - \QLP(\bx_h,\by_h,h)\\
    &\ge \frac{2(\sigma \kappa-2\sigma -2r_{\max}HS) z_h \log N}{\sqrt{N}}- \frac{C_h}{\sqrt{N}}\label{eq:lower-1}\\
    &\ge \frac{C_h}{\sqrt{N}},\label{eq:lower-2}
\end{align}
where \eqref{eq:lower-1} follows from the lower bound \eqref{eq:lower-LP} proved in the last step, and \eqref{eq:lower-2} holds as long as $N$ is large enough such that
\begin{equation*}  \label{eq:condition-on-N}
  \log N \ge \frac{C_h}{(\sigma \kappa-2\sigma -2r_{\max}HS) z_h}.
\end{equation*}
This completes the proof.

\subsubsection{Proof of Claim~\ref{claim:properties-mappings}}\label{sec:proof-claim:properties-mappings}
We first derive equivalent forms of $\QLP(\bx_h, \by_h, h)$ and $\QLP^+(\bx^+_h, \by^+_h, h)$.

\paragraph{Equivalent form of $\QLP(\bx_h, \by_h, h)$.}
For convenience, we let $(\bx^+_h)_{[S]}$ denote the $S$-dimensional vector $(x^+_h(s))_{s\in\{1,2,\dots,S\}}$ and $(\by^+_h)_{[S]}$ denote the $2S$-dimensional vector $(y^+_h(s,a))_{s\in\{1,2,\dots,S\},a\in\{0,1\}}$.

At a high level, we decompose $\bx_h$ and  $\by_h$ as
\begin{align*}
    \bx_h &= (\bx_h^+)_{[S]} + (\bx_h - (\bx_h^+)_{[S]}),\\
    \by_h &= (\by_h^+)_{[S]} + (\by_h - (\by_h^+)_{[S]}),
\end{align*}
where $\bx_h - (\bx_h^+)_{[S]}$ and $\by_h - (\by_h^+)_{[S]}$ are both nonnegative vectors,
and decompose the LP correspondingly.
Consider the following generalization of the LPs considered before.
For any nonnegative vectors $\widehat{\bx}_h$ ($\widehat{\bx}_h(\cdot)$ may not sum up to $1$), 
any $(\alpha_{h},\alpha_{h+1},\dots,\alpha_{H})$ with $0\le \alpha_{h'}\le \alpha$ for all $h'$,
and any $\widehat{\by}_h$ with $\sum_a\widehat{y}_h(s,a)=\widehat{x}_h(s)$ and $\sum_s\widehat{y}_h(s,1)=\alpha_h$, we define
\begin{alignat}{2}
    & \quad \QLP(\widehat{\bx}_h,\widehat{\by}_h, h,(\alpha_h,\alpha_{h+1},\dots,\alpha_{H})) :=  \br_{h} \widehat{\by}_h^\top+ \underset{ \widehat{\by}_{h'} }{\mathrm{max}} \quad   \sum_{h'=h+1}^{H}  \br_{h'} \widehat{\by}_{h'}^\top \label{eq:problem-formulation-LP-stage-h-Q-generalized} \\
     \text{s.t.} & \quad \sum_{s} \widehat{y}_{h'} (s,1) = \alpha_{h'},\quad h\le h'\le H,\nonumber\\
     &\quad \widehat{\by}_{h'}(\cdot, 0) + \widehat{\by}_{h'}(\cdot,1) = \widehat{\bx}_{h'}(\cdot),\quad  \widehat{\by}_{h'} \ge {\bf 0},\quad h\le h'\le H, \nonumber  \\
    & \quad \widehat{\bx}_{h'+1}(\cdot)  = \sum_{s,a} \widehat{y}_{h'}(s,a) \mathbf{P}_{h'} \left( \cdot \mid s, a \right),  h\le h'\le H-1\nonumber
\end{alignat}
\begin{alignat}{2}
    & \quad \VLP(\widehat{\bx}_h,h,(\alpha_h,\alpha_{h+1},\dots,\alpha_{H})) :=  \max_{\widehat{\by}_h} \QLP(\widehat{\bx}_h, \widehat{\by}_h, h,(\alpha_h,\alpha_{h+1},\dots,\alpha_{H})) \label{eq:problem-formulation-LP-stage-h-V-generalized} \\
     \text{s.t.} & \quad \sum_{s} \widehat{y}_{h} (s,1) = \alpha_h,\nonumber\\
     &\quad \widehat{\by}_{h}(\cdot, 0) + \widehat{\by}_{h}(\cdot,1) = \widehat{\bx}_{h}(\cdot), \quad  \widehat{\by}_{h} \ge {\bf 0}. \nonumber  
\end{alignat}
Note that both the Q-value and the V-value are continuous in $(\alpha_h,\alpha_{h+1},\dots,\alpha_{H})$.
Also any feasible solution satisfies that $\sum_{s=1}^S\widehat{x}_{h'}(s)=\sum_{s=1}^S\widehat{x}_{h}(s)$ for all $h'$ with $h\le h'\le H$.

Then we can write
\begin{equation}\label{eq:equiv-LP}
\begin{split}
    &\mspace{23mu}\QLP(\bx_h, \by_h, h)\\
    &=\br_h(\by^+_h)_{[S]}^\top+\br_h(\by_h-(\by^+_h)_{[S]})^\top\\
    &\mspace{23mu}+\sup_{(\alpha_{h+1},\alpha_{h+2},\dots,\alpha_H)}\left(\VLP(\phi((\by^+_h)_{[S]}),h+1,(\alpha_{h+1},\dots,\alpha_{H}))\right.\\
    &\mspace{180mu}\left.+\VLP(\phi((\by_h-(\by^+_h)_{[S]})),h+1,(\alpha-\alpha_{h+1},\dots,\alpha-\alpha_{H}))\right).
\end{split}
\end{equation}
Let $(\alpha^*_{h+1},\alpha^*_{h+2},\dots,\alpha^*_H)$ be the one that achieves the supremum, which exists due to continuity and the fact that $(\alpha_{h+1},\alpha_{h+2},\dots,\alpha_H)$ lies in the closed set $[0,\alpha]^{H-h}$.

\paragraph{Equivalent form of $\QLP^+(\bx^+_h, \by^+_h, h)$.}
It is not hard to see that
\begin{equation}\label{eq:equiv-LP-plus}
\begin{split}
    \QLP^+(\bx^+_h, \by^+_h, h)&=x_h^+(S+1)\cdot r_{\max}\cdot(H-h+1)\\
    &\mspace{23mu}+\sup_{(\alpha_{h+1},\alpha_{h+2},\dots,\alpha_H)}\QLP((\bx^+_h)_{[S]}, (\by_h^+)_{[S]}, h, (\alpha_{h},\alpha_{h+1},\dots,\alpha_H))\\
    &=x_h^+(S+1)\cdot r_{\max}\cdot(H-h+1)+\br_h(\by_h^+)_{[S]}^\top\\
    &\mspace{23mu}+\sup_{(\alpha_{h+1},\alpha_{h+2},\dots,\alpha_H)}\VLP(\phi((\by_h^+)_{[S]}), h+1, (\alpha_{h+1},\dots,\alpha_H)),
\end{split}
\end{equation}
where $\alpha_h=\sum_{s=1}^S y_h^+(s,1)$.
Similarly, the supremum can be achieved at some vector $(\alpha^{+,*}_{h+1},\alpha^{+,*}_{h+2},\dots,\alpha^{+,*}_H)$

\paragraph{Proof of \eqref{eq:upper-yh} and \eqref{eq:upper-yhstar}.}
Recall that to prove \eqref{eq:upper-yh} is to prove $\QLP(\bx_h, \by_h, h) \leq \QLP^+(\bx^+_h, \by^+_h, h)$.
By the equivalent form of $\QLP(\bx_h, \by_h, h)$ in \eqref{eq:equiv-LP}, we know that
\begin{align*}
    &\mspace{23mu}\QLP(\bx_h, \by_h, h)\nonumber\\
    &=\br_h(\by^+_h)_{[S]}^\top+\VLP(\phi((\by^+_h)_{[S]}),h+1,(\alpha_{h+1}^*,\dots,\alpha_{H}^*))\\
    &\mspace{23mu}+\br_h(\by_h-(\by^+_h)_{[S]})^\top+\VLP(\phi(\by_h-(\by^+_h)_{[S]}),h+1,(\alpha-\alpha_{h+1}^*,\dots,\alpha-\alpha_{H}^*)).
\end{align*}
Note that 
\begin{align*}
    \br_h(\by_h-(\by^+_h)_{[S]})^\top&=\sum_{s=1}^S \sum_a r_h(s,a)(y_h(s,a)-y^+_h(s,a))\\
    &\le r_{\max}\sum_{s=1}^S \sum_a (y_h(s,a)-y^+_h(s,a))\\
    &=r_{\max}x_h^+(S+1)
\end{align*}
and
\begin{align*}
    &\mspace{23mu}\VLP(\phi(\by_h-(\by^+_h)_{[S]}),h+1,(\alpha-\alpha_{h+1}^*,\dots,\alpha-\alpha_{H}^*))\\
    &\le r_{\max}\sum_{s=1}^S\sum_a (y_h(s,a)-y^+_h(s,a))\cdot (H-h)\\
    &=r_{\max}x_h^+(S+1)\cdot(H-h).
\end{align*}
Therefore,
\begin{align*}
    &\mspace{23mu}\QLP(\bx_h, \by_h, h)\nonumber\\
    &\le\br_h(\by^+_h)_{[S]}^\top+\VLP(\phi(\by_h-(\by^+_h)_{[S]}),h+1,(\alpha_{h+1}^*,\dots,\alpha_{H}^*))+r_{\max}x_h^+(S+1)\cdot(H-h+1)\\
    &\le \QLP^+(\bx^+_h, \by^+_h, h).
\end{align*}

The inequality \eqref{eq:upper-yhstar} can be proven similarly and we omit the details here.

\paragraph{Proof of \eqref{eq:lower-yh} and \eqref{eq:lower-yhstar}.}
Recall that $(\alpha^{+,*}_{h+1},\alpha^{+,*}_{h+2},\dots,\alpha^{+,*}_H)$ achieves the supremum in the equivalent form of $\QLP^+(\bx^+_h, \by^+_h, h)$ in \eqref{eq:equiv-LP-plus}.
Then
\begin{align}
    &\mspace{23mu}\QLP(\bx_h, \by_h, h)\nonumber\\
    &\ge\br_h(\by^+_h)_{[S]}^\top+\VLP(\phi((\by^+_h)_{[S]}),h+1,(\alpha_{h+1}^{+,*},\dots,\alpha_{H}^{+,*}))\nonumber\\
    &\mspace{23mu}+\br_h(\by_h-(\by^+_h)_{[S]})^\top+\VLP(\phi(\by_h-(\by^+_h)_{[S]}),h+1,(\alpha-\alpha_{h+1}^{+,*},\dots,\alpha-\alpha_{H}^{+,*}))\nonumber\\
    &\ge \br_h(\by^+_h)_{[S]}^\top+\VLP(\phi((\by^+_h)_{[S]}),h+1,(\alpha_{h+1}^{+,*},\dots,\alpha_{H}^{+,*}))\label{eq:ignoring-nonnegative-reward}\\
    &=\QLP^+(\bx^+_h, \by^+_h, h)-x_h^+(S+1)\cdot r_{\max}\cdot(H-h+1),\nonumber
\end{align}
where \eqref{eq:ignoring-nonnegative-reward} uses the nonnegativity of rewards and of the $\by_h-(\by^+_h)_{[S]}$ vector.
Note that $x_h^+(S+1)\le \sum_{s=1}^S|x_h(s)-x_h^*(s)|\le Sz_h\delta_h$.
Therefore,
\[
\QLP(\bx_h, \by_h, h)\ge \QLP^+(\bx^+_h, \by^+_h, h)-r_{\max}HSz_h\delta_h,
\]
which completes the proof of \eqref{eq:lower-yh}.
The inequality \eqref{eq:lower-yhstar} can be proven similarly and we omit the details here.

\paragraph{Proofs of \eqref{eq:lower-yhstarprime} and \eqref{eq:lower-yhprime}.} 
Note that the vector $\by^{*,'}_h-(\by_h^{*,+})_{[S]}$ is a nonnegative vector, so we can decompose $\by^{*,'}_h$ as $\by^{*,'}_h=(\by_h^{*,+})_{[S]}+(\by^{*,'}_h-(\by_h^{*,+})_{[S]})$ and derive the equivalent forms of $\QLP(\bx_h, \by^{*,'}_h, h)$ and $\QLP^+(\bx^+_h, \by^{*,+}_h, h)$ accordingly.
Similarly, the vector $\by'_h-(\by_h^{+})_{[S]}$ is a nonnegative vector, so we can decompose $\by'_h$ as $\by'_h=(\by_h^{+})_{[S]}+(\by'_h-(\by_h^{+})_{[S]})$ and derive the equivalent forms of $\QLP(\bx_h^*, \by'_h, h)$ and $\QLP^+(\bx^+_h, \by^{+}_h, h)$ accordingly.
Therefore, the proofs of \eqref{eq:lower-yhstarprime} and \eqref{eq:lower-yhprime} are similar to those of \eqref{eq:lower-yh} and \eqref{eq:lower-yhstar}, respectively, and we omit the details here.

\subsubsection{Proof of Claim~\ref{claim:anti-Lipschitz}.}
\label{sec:anti-lip}
To ease the discussion, let us write the LP in \eqref{eq:problem-formulation-LP-stage-h-Q} and \eqref{eq:problem-formulation-LP-stage-h-V} in the following standard form:
\begin{align*}
\text{LP}:\quad 
& \max \{ \br \by^\top : \; \bA \by = \bb, \;  \by \ge \bzero \}.
\end{align*}
where
\begin{itemize}
    \item \( \bA \in \mathbb{R}^{m \times n} \) is a given matrix with $m \le n$, and $\bb \in \mathbb{R}^m$ is a given r.h.s.\ vector;
    \item \( \br \in \mathbb{R}^n \) is the reward vector;
    \item \( \by \in \mathbb{R}^n \) is the decision variable vector.
\end{itemize}

Define the \emph{feasible set} of $\text{LP}$ by
\begin{equation}  \label{eq:polyhedron-definition}
  \calU := \{  \by \in \mathbb{R}^n  \mid  \bA \by = \bb, \; \by \ge \bzero \}.
\end{equation}
By \citep[Theorem 1.1]{ziegler2012lectures}, the set $\calU$ written in the above form is a polyhedron. In our case it is further a polytope (i.e., a bounded polyhedron) since any element in $\calU$ has $\norm{\cdot}_1$-norm bounded by $H$.

We suppose that the feasible set $\calU$ is \emph{not} a singleton. Note that under some extremal and trivial cases $\calU$ can reduce to a singleton, e.g. if the initial condition has $x(s)=1$ and $x(s')=0$ for all $s' \neq s$, and the transition kernels are all identity matrices. But then Equation~\eqref{eq:anti-lipschitz} holds trivially.

From the Fundamental Theorem of Linear Programming (see, e.g. By \citep[Theorem 13.2]{nocedal1999numerical}), which states that \textit{if an optimal solution of an LP exists, at least one optimal solution is a vertex; consequently, whenever the optimum is unique it must be that vertex}, we deduce that $\by^*$, being the unique optimal solution by assumption, must be a vertex of $\calU$.

Let $V(\by) := \br \by^\top$ be the objective value for any $\by \in \calU$. Because $\by^*$ is the unique maximizer, for all $\by \in \calU$ with $\by \neq \by^*$ we have $V(\by) < V(\by^*)$. We then define the \emph{smallest (nontrivial) descent} at $\by^*$ as
\begin{equation}
  \sigma := \inf_{\substack{\by\in \calU \\ \by\neq\by^{*}}} \frac{\br \bigl(\by^{*} - \by\bigr)^\top}{\norminf{\by^{*} - \by}}.
  \label{eq:sigma-b-min-expression-I} 
\end{equation}
Note that since $\calU$ is not a singleton, $\sigma$ is well defined.  

Now the Minkowski–Weyl Theorem \citep[Theorem 1.1]{ziegler2012lectures} states that every polytope is the convex hull of a \emph{finite} set of its vertices.
Hence we write $\calU = \operatorname{conv}\{\bv_{1},\dots,\bv_{r}\}$ where $\bv_{1},\dots,\bv_{r} \in \calV$, where $\calV$ is the \emph{finite} set of vertices of $\calU$. Note that $\by^* \in \calV$ and we may write $\by^* = \bv^*$. 
Therefore, we also have
\begin{align*}  
  \sigma &= \min_{\substack{\bv \in \calV \\ \bv \neq \bv^*}} \frac{\br \bigl(\bv^*-\bv\bigr)^\top}{\norminf{\bv^*-\bv}}.
\end{align*}
To see this, for any $\by \in \calU$ and $\by \neq \by^*$, write $\by = \sum_{i=1}^{r} \lambda_i \bv_i$ with $\lambda_i \ge 0$ and $\sum_{i=1}^{r} \lambda_i = 1$, we have
\begin{align*}
 \frac{\br \bigl(\by^*-\by\bigr)^\top}{\norminf{\by^*-\by}} &=  \sum_{i=1}^{r} \frac{\lambda_i \br (\bv^* - \bv_i)^\top}{\norminf{\sum_{j=1}^{r} \lambda_j (\bv^* - \bv_j)}} \\
  &\ge \sum_{i=1}^{r} \frac{\lambda_i \br (\bv^* - \bv_i)^\top}{\sum_{i=j}^{r} \norminf{\lambda_j (\bv^* - \bv_j)}} \\
  &= \sum_{i=1}^{r} \left( \frac{\lambda_i \norminf{\bv^* - \bv_i}}{\sum_{j=1}^{r} \lambda_j \norminf{\bv^* - \bv_j}} \right) \frac{\br (\bv^*-\bv_i)^\top}{\norminf{\bv^*-\bv_i}} \\
  &\ge \min_{\substack{\bv \in \calV \\ \bv \neq \bv^*}} \frac{\br \bigl(\bv^*-\bv\bigr)^\top}{\norminf{\bv^*-\bv}} \\
  &= \sigma.
\end{align*}
The above holds for any $\by \in \calU$ and $\by \neq \by^*$, we deduce that
\begin{equation*}
 \sigma = \inf_{\substack{\by\in \calU \\ \by\neq\by^{*}}} \frac{\br \bigl(\by^*-\by\bigr)^\top}{\norminf{\by^*-\by}} \ge \min_{\substack{\bv \in \calV \\ \bv \neq \bv^*}} \frac{\br \bigl(\bv^*-\bv\bigr)^\top}{\norminf{\bv^*-\bv}} \ge \sigma,
\end{equation*}
so all the inequalities above must be equalities.

Now $\sigma>0$ because of the uniqueness assumption and the fact that there is only a finite number of vertices in $\calV$. 

To complete the proof of Claim~\ref{claim:anti-Lipschitz}, we apply the general result established above to $\VLP(\xini,1)$ in \eqref{eq:problem-formulation-LP}, with unique optimal solution $\by^*$. For any given $1 \le h \le H$ and $\by_h \in \calY_{\bx^*_h}$, we construct a feasible $\hat{\by}$ for \eqref{eq:problem-formulation-LP} as follows:
\begin{itemize}
  \item For $1 \le h' \le h-1$, set $\hat{\by}_{h'} := \by^*_{h'}$;
  \item Set $\hat{\by}_h := \by_h$, and define $\hat{\bx}_{h+1}$ such that $\hat{\bx}_{h+1}(\cdot)  = \sum_{s,a} \hat{y}_{h}(s,a) \mathbf{P}_{h} \left( \cdot \mid s, a \right)$;
  \item The concatenated vector $(\hat{\by}_{h+1}, \dots, \hat{\by}_H)$ optimally solves $\VLP(\hat{\bx}_{h+1}, h+1)$ in \eqref{eq:problem-formulation-LP-stage-h-V}.
\end{itemize}

Note that we can see $\by^*_h$ and $\by_h$ as the $(2Sh+1)$-th to $2S(h+1)$-th coordinates of the vectors $\by^*$ and $\hat{\by}$, respectively, and from the construction we deduce that
\begin{align*}
  \frac{\QLP(\bx^*_h,\by^*_h,h) - \QLP(\bx^*_h,\by_h,h)}{\norminf{\by^*_h - \by_h}} & \ =  \frac{V(\by^*) - V(\hat{\by})}{\norminf{\by^*_h - \by_h}} \\
  & \ \ge \frac{V(\by^*) - V(\hat{\by})}{\norminf{\by^* - \hat{\by}}} \\
  & \ \ge \sigma.
\end{align*}

\subsubsection{Proof of Claim~\ref{claim:Q-relation}.}\label{sec:proof-claim:Q-relation}
\paragraph{Proof of \eqref{eq:fluid-and-N-sys-gap}.}

For a step $1 \le h \le H$, define the deterministic drift function $\phi_h: \Delta^{2S} \rightarrow \Delta^{S}$ as for given $\by_h \in \Delta^{2S}$, in each coordinate $1 \le s' \le S$: 
\begin{equation}  \label{ea:function-phi}
  \phi_h(\by_h)(s') := \sum_{s,a} y_h(s,a) P_h(s' \mid s,a).
\end{equation}
We remark that $\phi_h$ is a linear function, hence it is Lipschitz-continuous. Denote by $L_h>0$ its Lipschitz constant under the $\norminf{\cdot}$-norm for vectors, which is independent of $N$. We note that a loose upper bound of $L_h$ is $2S$.

We use backward induction on $h$ below to show that there exists constants $C_h > 0$ independent of $N$, $\bx_h \in \simpXN$ and $\bY_h \in \calY_{\bx_h}$ such that 
\begin{equation} \label{eq:middle-step-I}
  \VLP(\bx_h,h) - \VNopt(\bx_h,h) \le \frac{C_h}{\sqrt{N}},
\end{equation} 
and 
\begin{equation*}
  \QLP(\bx_h, \bY_h, h) - Q^{N}(\bx_h, \round(\bY_h), h) \le \frac{C_h}{\sqrt{N}},
\end{equation*}
where the latter is a repetition of \eqref{eq:fluid-and-N-sys-gap}.

The claim holds for $h = H$ since the two value functions coincide. Now suppose that the claim holds for $h+1$.

Standard theory on sensitivity analysis of LP (see, e.g. \citep[Section 5.2]{bertsimas1997introduction}) implies that $\bx \mapsto \VLP(\bx,h)$ and $\YLP_h \mapsto \QLP(\bx,\YLP_h,h)$ are Lipschitz functions, we denote the Lipschitz constant of the former by $L_{V,h}$, since the Lipschitz constant of the later can be upper bounded by $\rmax + L_{V,h+1} L_{h}$, where $L_{h} > 0$ is the Lipschitz constant of $\phi_h(\cdot)$ via
\begin{align} \label{eq:Lipschitz-constant-relation}
     & \abs{\QLP(\bx,\YLP_h,h) - \QLP(\bx,\YLP'_h,h)} \nonumber \\
 \le & \ \rmax \norminf{\YLP_h - \YLP'_h} + \abs{\VLP(\phi_h(\YLP_h), h+1) - \VLP(\phi_h(\YLP'_h),h+1)}  \nonumber  \\
 \le & \ \left( \rmax + L_{V,h+1} L_{h} \right) \norminf{\YLP_h - \YLP'_h}.
\end{align}

We write 
\begin{align}
  \VLP(\bx,h) - \VNopt(\bx,h) = & \ \QLP(\bx,\YLP^*_h,h) - Q^N(\bx,\Yopt_h,h) \nonumber \\
  \le & \ \QLP(\bx,\YLP^*_h,h) - \QLP(\bx, \round(\YLP^*_h),h)  \nonumber \\
  & \qquad + \QLP(\bx, \round(\YLP^*_h), h) - Q^N(\bx, \round(\YLP^*_h), h) \label{eq:part-III} \\
  \le_{(a)} & \ (\rmax + L_{V,h+1} L_{h}) \frac{1}{N} \nonumber  \\
  & \qquad + \VLP(\phi_h(\YLP'_h), h+1) - \expect{\VNopt(\Xsys, h+1) \mid \YLP'_h}  \nonumber \\
  \le & \ (\rmax + L_{V,h+1} L_{h}) \frac{1}{\sqrt{N}}  \nonumber \\
   & \ + \abs{\VLP(\phi_h(\YLP'_h), h+1) - \expect{\VLP(\Xsys, h+1) \mid \YLP'_h}}  \label{eq:part-I} \\
   & \ + \abs{\expect{\VLP(\Xsys, h+1) \mid \YLP'_h} - \expect{\VNopt(\Xsys, h+1) \mid \YLP'_h}}.  \label{eq:part-II}
\end{align}
Here we abbreviate $\YLP'_h = \round(\by^{*,'}_h)$. Inequality (a) follows from the rounding error bound \eqref{eq:rounding-error-bound} together with \eqref{eq:Lipschitz-constant-relation}.

Since $\expect{\Xsys \mid \YLP'_h} = \phi_h(\YLP'_h)$ and $\bx \mapsto \VLP(\bx, h+1)$ is a Lipschitz function with constant $L_{V, h+1} > 0$, we have
\begin{equation*}
  \eqref{eq:part-I} \le L_{V,h+1} \expect{\norminf{\Xsys - \phi_h(\YLP'_h)} \mid \YLP'_h} \le_{(a)} \frac{L_{V,h+1} \sqrt{S}}{\sqrt{N}},
\end{equation*}   
where inequality (a) follows from \citep[Lemma 1]{gast2023linear}.

Furthermore, from the induction hypothesis we have $\eqref{eq:part-II} \le C_{h+1}/\sqrt{N}$. Hence it suffices to define
\begin{equation*}
  C_h := (\rmax + L_{V,h+1} L_{h})  + C_{h+1} + L_{V,h+1} \sqrt{S}
\end{equation*}
to complete the induction step for proving Equation~\eqref{eq:middle-step-I}, and Equation~\eqref{eq:fluid-and-N-sys-gap} holds as well since it appears in \eqref{eq:part-III}.

\paragraph{Proof of \eqref{eq:QLP-larger-QN}.}  

One has
\begin{align*}
   & \ \QLP(\bx,\bY_h,h) - Q^N(\bx,\bY_h,h) \\
 = & \  \VLP(\phi_h(\bY_h), h+1) - \expect{\VNopt(\bXN_{h+1}, h+1) \mid \bY_h} \\
 \ge_{(a)} & \ \expect{\VLP(\bXN_{h+1}, h+1) \mid \bY_h}  - \expect{\VNopt(\bXN_{h+1}, h+1) \mid \bY_h}  \\
 \ge_{(b)} & \ 0,
\end{align*}
where inequality (a) follows from Jensen's inequality with the facts that the value function $\bx \mapsto \VLP(\bx,h+1)$ is concave, and $\mathbb{E}_{\bXN_{h+1} \mid \bY_h \simd \text{Equation}\eqref{eq:rv-Nsys} } \left[ \bXN_{h+1} \mid \bY_h \right] = \phi_h(\bY_h)$; inequality (b) follows since $\VNopt(\bx,h) \le \VLP(\bx,h)$.

\subsection{Proof of Lemma~\ref{lem:vd} (Value difference)}

We prove this result using backward induction. We will drop $N$ in the superscript and the policies in the subscript to simplify the notation, and drop $h$ in the subscript in the state variable $\bx_h$ and $\bX_h$, without causing confusion. 

We also note that $V(\bx, h)$ is not defined if $\bx \notin \simpXN$, i.e, if $N\bx$ is not an integer vector, so we define $V(\bx, h)=\tilde{V}(\bx,h)$ if $\bx \notin \simpXN$. Similarly, we define $Q(\bx,\by_h, h)=\tilde{Q}(\bx,\by_h,h)$ if $\bx \notin \simpXN$, and $Q(\bx,\by_h, h)={Q}(\bx,\round(\by_h),h)$ if  $\bx \in \simpXN$ but $\by_h \notin \simpYN$.

We first consider $h=H.$  The result bound holds for both the Q-function and value function if $\bx \notin \simpXN$ by definition. Now if $\bx \in \simpXN$, then
\begin{align*}
Q(\bx, \by_H, H) = & \ Q(\bx, \round(\by_H), H)=\br_H \round(\by_H)^\top, \\    
\tilde{Q}(\bx, \by_H, H) = & \ \br_H \by^\top_H,
\end{align*} 
which implies 
\begin{equation*}
  |Q(\bx, \by_H, H)-\tilde{Q}(\bx, \by_H, H)| = \calO \left(\frac1N\right)
\end{equation*}    
due to rounding error detailed in Appendix~\ref{app:details-Gaussian}. Furthermore, for $\|\bx-\bx^*_H\|_\infty \leq z_H \Ndelta$ and $\bx \in \simpXN$, we first note that $\pi^*(\bx, H), \tilde{\pi}^*(\bx, H) \in{\cal Y}_{\bx}\bigcap {\cal Y}_{\kappa \Ndelta H}$ according to the definition of policy class $\pidelta$. 
Now if $V(\bx, H)\geq \tilde{V}(\bx, H),$ then 
\begin{align*}
  & |V(\bx,H)-\tilde{V}(\bx,H)|\\
= & \ V(\bx,H)-\tilde{V}(\bx,H)\\
= & \ Q(\bx, \pi^*(\bx,H), H) - \tilde{Q}(\bx, \tilde{\pi}^*(\bx,H),H)\\
\leq & \ {Q}(\bx, {\pi}^*(\bx,H), H)-\tilde{Q}(\bx, {\pi}^*(\bx,H),H)\\
= & \ \calO\left(\frac1N\right),
\end{align*} 
where the inequality holds because $$\tilde{Q}(\bx, \tilde{\pi}^*(\bx,H),H) = \max_{\by_H \in{\cal Y}_{\bx}\bigcap {\cal Y}_{\kappa z_H \Ndelta}}\tilde{Q}(\bx, \by_H, H)\geq \tilde{Q}(\bx, {\pi}^*(\bx,H),H).$$ 

By symmetry we can similarly show that when  $V(\bx, H)< \tilde{V}(\bx, H),$
\begin{align*}
|V(\bx,H)-\tilde{V}(\bx,H)| = \calO\left(\frac1N\right).
\end{align*}

We now prove the induction step, and suppose that: (i) for any $\bx$ such that $\bx\in{\cal X}_{z_{h+1}\Ndelta}$,
\begin{equation*}
  \left|\tilde{V}(\bx, h+1)-V(\bx, h+1)\right| = \tON;
\end{equation*}
and (ii) for $\bx\in{\cal X}_{z_{h+1}\Ndelta}$ and $\by_h \in {\cal Y}_{\bx}\bigcap{\cal Y}_{\kappa z_{h+1}\Ndelta}$,
\begin{equation*}
  \left|\tilde{Q}(\bx, \by_h, h+1)-Q(\bx,\by_h, h+1)\right| = \tON.
\end{equation*}

We recall that by definition if $\bx \notin \simpXN$, then
\begin{equation*}
  Q(\bx,\by_h, h)=\tilde{Q}(\bx,\by_h, h)\quad\hbox{and}\quad V(\bx, h)=\tilde{V}(\bx, h).
\end{equation*}
Hence we only need to consider $\bx$ such that $\bx \in \simpXN$. We first consider the Q-function for  $\bx\in{\cal X}_{z_h\Ndelta}$ and $\by_h \in {\cal Y}_{\bx}\bigcap{\cal Y}_{\kappa z_h\Ndelta}$:
\begin{align*}
    & Q(\bx,\by_h,h) - \tilde{Q}(\bx,\by_h,h)  \\
  = & \ \expect{V(\bX, h+1)} - \expect{\tilde{V}(\tilde{\bX}, h+1)} \\
  = & \ \expect{V(\bX, h+1)} - \expect{\tilde{V}(\bX, h+1)}+\expect{\tilde{V}(\bX, h+1)}-\expect{\tilde{V}(\tilde{\bX}, h+1)},
\end{align*} 
where $\bX$ and $\tilde{\bX}$ are the states at stage $h+1$ in the $N$-system and Gaussian stochastic system, respectively, under action $\round(\by_h)$ and $\by_h$, respectively.

We first bound $ \expect{V(\bX, h+1)} - \expect{\tilde{V}(\bX, h+1)}$. We have 
\begin{align*}
    & \expect{V(\bX, h+1)} - \expect{\tilde{V}(\bX, h+1)}\\
  = & \ \expect{\left(V(\bX, h+1)-\tilde{V}(\bX, h+1)\right)1_{ \left\{\bX\in{\cal X}_{z_{h+1}\Ndelta}\right\} }}  \\
   & \qquad + \expect{\left(V(\bX, h+1)-\tilde{V}(\bX, h+1)\right)1_{ \left\{\bX\notin{\cal X}_{z_{h+1}\Ndelta}\right\} }}   \\
  \leq & \ \tON + \OlogN = \tON,
\end{align*}
where the last inequality is based on the induction assumption and the high probability bound in Item~\ref{lem-high-proba-item-1} of Lemma~\ref{lem:high-proba-bound}.

Next to bound $\expect{\tilde{V}(\bX, h+1)}-\expect{\tilde{V}(\tilde{\bX}, h+1)}$, we construct a joint distribution $(\bX', \Xdiff')$ such that: 
\begin{enumerate}[label=(\roman*)]
  \item $\bX' \simd \bX$ and $\Xdiff' \simd \Xdiff$,
  \item $\expect{\norminf{\bX' - \Xdiff'}} \leq 2 \, d_W^{(1)}(\bX', \Xdiff') = \tON$.
\end{enumerate}
This is possible due to the Wasserstein distance bound established in Lemma~\ref{lem:Wasserstein-multivariate}.

We then apply the almost locally Lipschitz property of the value function $\tilde{V}$ in Lemma~\ref{lem:local-Lipschitz}, which states that for any $\bx', \tilde{\bx}' \in{\cal X}_{z_{h+1}\Ndelta}$, it holds true that
\begin{equation*}
    \tilde{V}(\bx', h+1) - \tilde{V}(\tilde{\bx}', h+1) \leq u_{h+1} \|\bx' - \tilde{\bx}'\|_\infty + \OlogN.
\end{equation*}
Set $\bx'=\bX'$ and $\tilde{\bx}'=\Xdiff'$ and then take expectation (which are coupled in the same probability space), we have
\begin{align*}
  & \expect{\tilde{V}(\bX, h+1) \mid \bX\in{\cal X}_{z_{h+1}\Ndelta}}-\expect{\tilde{V}(\tilde{\bX}, h+1) \mid \tilde{\bX}\in{\cal X}_{z_{h+1}\Ndelta}} \\
 = & \ \expect{\tilde{V}(\bX', h+1) - \tilde{V}(\tilde{\bX}', h+1) \mid \bX',\tilde{\bX}'\in{\cal X}_{z_{h+1}\Ndelta}} \\
 \le & \ u_{h+1} \expect{\norminf{\bX' - \Xdiff'}\mid \bX',\tilde{\bX}'\in{\cal X}_{z_{h+1}\Ndelta}} + \OlogN \\
 =_{(a)} & \ \tON + \OlogN = \tON,
\end{align*} 
where equality (a) holds because given $\bx\in{\cal X}_{z_h\Ndelta}$ and $\by_h \in {\cal Y}_{\bx}\bigcap{\cal Y}_{\kappa z_h\Ndelta}$, we have
\begin{align*}
    \expect{\norminf{\bX' - \Xdiff'}} \geq & \ \expect{\norminf{\bX' - \Xdiff'}\mid \bX',\tilde{\bX}'\in{\cal X}_{z_{h+1}\Ndelta}} \proba{\bX',\tilde{\bX}'\in{\cal X}_{z_{h+1}\Ndelta}} \\
    = & \ \expect{\norminf{\bX' - \Xdiff'}\mid \bX',\tilde{\bX}'\in{\cal X}_{z_{h+1}\Ndelta}} \left(1-\OlogN \right),
\end{align*} 
where the last equality holds due to Item~\ref{lem-high-proba-item-1} and Item~\ref{lem-high-proba-item-2} of the high probability bound Lemma \ref{lem:high-proba-bound}. Hence 
\begin{align*}
    \expect{\norminf{\bX' - \Xdiff'}\mid \bX',\tilde{\bX}'\in{\cal X}_{z_{h+1}\Ndelta}}=\expect{\norminf{\bX' - \Xdiff'}} \left(1+\OlogN \right)= \tON.
\end{align*}

From the analysis above, we conclude that for any $\bx\in{\cal X}_{z_h \Ndelta}$ and $\by_h \in {\cal Y}_{\bx}\bigcap{\cal Y}_{\kappa z_h \Ndelta}$, we have
\begin{align}  \label{eq:Q-value-bound-1}
    |Q(\bx,\by_h,h) - \tilde{Q}(\bx,\by_h,h)| = \tON. 
\end{align} 

Now if $V(\bx, h)\geq \tilde{V}(\bx, h)$, note that $\pi^*(\bx,h)\in{\cal Y}_{\bx}\bigcap{\cal Y}_{\kappa z_h \Ndelta}$, hence 
\begin{align*}
  & |V(\bx,h)-\tilde{V}(\bx,h)|\\
= & \ V(\bx,h)-\tilde{V}(\bx,h)\\
= & \ Q(\bx, \pi^*(\bx,h), h) - \tilde{Q}(\bx, \tilde{\pi}^*(\bx,h),h)\\
\leq & \ {Q}(\bx, {\pi}^*(\bx,h), h)-\tilde{Q}(\bx, {\pi}^*(\bx,h),h) \\
= & \ \tON,
\end{align*} 
where the last equality follows from \eqref{eq:Q-value-bound-1}. By symmetry we can similarly show that when $V(\bx, h)< \tilde{V}(\bx, h)$ we also have
\begin{align*}
|V(\bx,h)-\tilde{V}(\bx,h)| = \tON.
\end{align*}

\subsubsection{Wasserstein distance bound between Gaussian stochastic and $N$-system transitions}
\label{sec:wasserstein-bound}

Define the 1--Wasserstein distance with respect to $\ell_1$-norm by
\begin{equation} \label{eq:1-Wasserstein-def}
   d_{W}^{(1)}(\mu,\nu) :=  \inf_{\gamma\in\Gamma(\mu,\nu)} \int_{\RR^{S}\!\times\RR^{S}} \norm{\bx - \by}_1 d\gamma(\bx,\by),
\end{equation}
where $\Gamma(\mu,\nu)$ denotes the set of all joint distributions of the laws of $\bX$ and $\bY$.

In this section, we prove the following Wasserstein distance bound between Gaussian stochastic and $N$-system transitions.

\begin{lemma}[Wasserstein distance bound]  \label{lem:Wasserstein-multivariate}
Consider any $\bx_h\in {\calX}_{z_h\Ndelta}\cap\simpXN$ and any $\by_h\in \calY_{\bx_h} \cap {\cal Y}_{\kappa z_h\Ndelta}\cap\simpYN$ and assume that $\Xdiff_h=\Xsys_h=\bx_h$ and $\Ydiff_h=\Ysys_h=\by_h$.
Then the $1$-Wassertein distance between the distributions of $\Xdiff_{h+1}$ and $\Xsys_{h+1}$, denoted as $\mu_{\Xdiff_{h+1}}$ and $\mu_{\Xsys_{h+1}}$, respectively, can be bounded as
\begin{equation}
    \wass{\mu_{\Xdiff_{h+1}}}{\mu_{\Xsys_{h+1}}} = \tON.
\end{equation}
\end{lemma}

\begin{proof}[Proof of Lemma~\ref{lem:Wasserstein-multivariate}]
Recall that given $\Ydiff_h=\Ysys_h=\by_h$, we can write
\begin{equation*}  
  \Xdiff_{h+1} =  \proj{\Delta^S}{\sum_{s,a} y_h(s,a)\pp_h(\cdot \mid s,a)+\frac{\bZ_h}{\sqrt{N}}},
\end{equation*}
where $\bZ_h \simd \calN(\bzero, \Gamma_h(\by_h^*))$, and
\begin{equation*}
\Xsys_{h+1} =\sum_{s,a}\frac{1}{N}\bVarMulti_h^{(s,a)},
\end{equation*}
where each $\bVarMulti^{(s,a)}_h=(\VarMulti^{(s,a)}_h(s'))_{s'\in \calS}\in \mathbb{N}^S$ follows a multinomial distribution $\multi(N y_h(s,a), \pp_h(\cdot \mid s,a))$ and the $\bVarMulti^{(s,a)}_h$'s for different $(s,a)$'s are independent.

We first define two auxiliary random vectors below:
\begin{align*}
    \Xdiff_{h+1}^{(1)} & := \sum_{s,a} y_h(s,a)\pp_h(\cdot \mid s,a)+\frac{\bZ_h}{\sqrt{N}},\\
    \Xdiff_{h+1}^{(2)} & := \sum_{s,a} y_h(s,a)\pp_h(\cdot \mid s,a)+\frac{\bZ_h'}{\sqrt{N}},
\end{align*}
where $\bZ_h' \simd \calN(\bzero, \Gamma_h(\by_h))$.
Then by triangle inequality,
\begin{align*}
\wass{\mu_{\Xdiff_{h+1}}}{\mu_{\Xsys_{h+1}}}
\le \wass{\mu_{\Xdiff_{h+1}}}{\mu_{\Xdiff_{h+1}^{(1)}}}
+\wass{\mu_{\Xdiff_{h+1}^{(1)}}}{\mu_{\Xdiff_{h+1}^{(2)}}}
+\wass{\mu_{\Xdiff_{h+1}^{(2)}}}{\mu_{\Xsys_{h+1}}}.
\end{align*}
We prove that each of these three terms is $\calO\left(\frac{\log N}{N}\right)$ below.

\paragraph{Upper bound on $\wass{\mu_{\Xdiff_{h+1}}}{\mu_{\Xdiff_{h+1}^{(1)}}}$.}
By the definition of the Wasserstein distance, 
\begin{align}
    \wass{\mu_{\Xdiff_{h+1}}}{\mu_{\Xdiff_{h+1}^{(1)}}}&\le \Ex{\left\|\Xdiff_{h+1}-\Xdiff_{h+1}^{(1)}\right\|_1}\\
    &=\Ex{\left\|\Xdiff_{h+1}-\Xdiff_{h+1}^{(1)}\right\|_1 1_{\left\{\Xdiff_{h+1}\neq\Xdiff_{h+1}^{(1)}\right\}}}\\
    &\le \sqrt{\Ex{\left\|\Xdiff_{h+1}-\Xdiff_{h+1}^{(1)}\right\|_1^2}} \sqrt{\Ex{1^2_{\left\{\Xdiff_{h+1}\neq\Xdiff_{h+1}^{(1)}\right\}}}}\label{eq:CS-indicator}\\
    &= \sqrt{\Ex{\left\|\Xdiff_{h+1}-\Xdiff_{h+1}^{(1)}\right\|_1^2}} \sqrt{\proba{\Xdiff_{h+1} \neq \Xdiff_{h+1}^{(1)}}},
\end{align}
where \eqref{eq:CS-indicator} follows from the Cauchy–Schwarz inequality.
Note that by Item~\ref{lem-high-proba-item-2} of Lemma~\ref{lem:high-proba-bound},
\begin{align*}
    \proba{\Xdiff_{h+1} \neq \Xdiff_{h+1}^{(1)}}= \proba{\Xdiff_{h+1}^{(1)} \not \geq \bzero} = \OlogN.
\end{align*}
Furthermore, observe that $\norm{\Xdiff_{h+1} - \Xdiff_{h+1}^{(1)}}_1^2 \le \norm{\Xdiff_{h+1}}_1^2$ because 
\begin{equation*}
   \tilde{X}_{h+1}(s) = 
  \begin{cases}
    \tilde{X}_{h+1}^1(s) & \text{if } \tilde{X}_{h+1}^1(s) \ge 0,\\
    0    & \text{if } \tilde{X}_{h+1}^1(s) < 0.
  \end{cases}
\end{equation*}
Thus
\begin{equation*}
  \Ex{\left\|\Xdiff_{h+1}-\Xdiff_{h+1}^{(1)}\right\|_1^2} \le \Ex{\left\|\Xdiff_{h+1}^{(1)}\right\|_1^2} \le S \Ex{\left\|\Xdiff_{h+1}^{(1)}\right\|_2^2}.
\end{equation*}
Note that
\begin{align*}
    \Ex{\left\|\Xdiff_{h+1}^{(1)}\right\|_2^2}&=\sum_{s'}\Ex{\left(\Xdiff_{h+1}^{(1)}(s')\right)^2}\\
    &=\sum_{s'}\left(\left(\Ex{\Xdiff_{h+1}^{(1)}(s')}\right)^2+\Var{\Xdiff_{h+1}^{(1)}(s')}\right)\\
    &=\sum_{s'}\left(\left(\sum_{s,a} y_h(s,a)\pp_h(s' \mid s,a)\right)^2+\frac{1}{N}\Var{\bZ_h(s')}\right)\\
    &=\sum_{s'}\left(\sum_{s,a} y_h(s,a)\pp_h(s' \mid s,a)\right)^2 \\
    & \qquad +\frac{1}{N}\sum_{s'}\sum_{s,a}y_h^*(s,a)\pp_h(s' \mid s,a)(1-\pp_h(s' \mid s,a))\\
    &\le \left(\sum_{s'} \sum_{s,a} y_h(s,a)\pp_h(s' \mid s,a)\right)^2+\frac{1}{N}\sum_{s'}\sum_{s,a}y_h^*(s,a)\pp_h(s' \mid s,a)\\
    &\le 1+\frac{1}{N}.
\end{align*}
Combining the bounds above yields
\begin{align*}
    \wass{\mu_{\Xdiff_{h+1}}}{\mu_{\Xdiff_{h+1}^{(1)}}}&\le \sqrt{S\left(1+\frac{1}{N}\right)}\cdot \sqrt{\calO\left(\frac{1}{N^{\log N}}\right)}=\calO\left(N^{-\frac{\log N}{2}} \right)=\calO\left(\frac{\log N}{N}\right).
\end{align*}

\paragraph{Upper bound on $\wass{\mu_{\Xdiff_{h+1}^{(1)}}}{\mu_{\Xdiff_{h+1}^{(2)}}}$.}
Note that $\bZ_h$ can be written as $\bZ_h=\sum_{s,a}y_h^*(s,a)\bN_h(s,a)$, where each $\bN_h(s,a) \simd \calN(\bzero,\Sigma_h(s,a))$ and the $\bN_h(s,a)$'s across different $(s,a)$'s are independent.
We couple $\Xdiff_{h+1}^{(1)}$ and $\Xdiff_{h+1}^{(2)}$ by letting $\bZ_h'=\sum_{s,a}y_h(s,a)\bN_h(s,a)$.
Then
\begin{align*}
    \wass{\mu_{\Xdiff_{h+1}^{(1)}}}{\mu_{\Xdiff_{h+1}^{(2)}}}
    &\le \Ex{\left\|\Xdiff_{h+1}^{(1)}-\Xdiff_{h+1}^{(2)}\right\|_1}\\
    &=\frac{1}{\sqrt{N}}\Ex{\left\|\bZ_h-\bZ_h'\right\|_1}\\
    &\le \frac{1}{\sqrt{N}} \sum_{s,a} \abs{y_h(s,a) - y^*_h(s,a)} \expect{\norm{\bN_h(s,a)}_1} \\
  & \le \frac{1}{\sqrt{N}} \sum_{s,a} \abs{y_h(s,a) - y^*_h(s,a)} \sqrt{S} \sqrt{\expect{\norm{\bN_h(s,a)}^2_2}} \\
  &=\frac{1}{\sqrt{N}} \sum_{s,a} \abs{y_h(s,a) - y^*_h(s,a)} \sqrt{S} \sqrt{\sum_{s'}\pp_h(s' \mid s,a)(1-\pp_h(s' \mid s,a))}\\
  &\le \frac{\sqrt{S}}{\sqrt{N}} \norm{\by_h - \by^*_h}_1 \\
  & \le \frac{2S\sqrt{S}}{\sqrt{N}} \norminf{\by_h - \by^*_h} \\
  &=\calO\left(\frac{\log N}{N}\right),
\end{align*}
where the last step has used the assumption that $\by_h\in{\cal Y}_{\kappa z_h\Ndelta}$.

\paragraph{Upper bound on $\wass{\mu_{\Xdiff_{h+1}^{(2)}}}{\mu_{\Xsys_{h+1}}}$.}
We first rewrite $\Xsys_{h+1}$ using the sum of $N$ independent random vectors on $\mathbb{R}^S$.
Let $\be_{s}\in \mathbb{R}^S$ be the vector whose $s$th entry is equal to $1$ and all other entries are equal to $0$.
Since each $\bVarMulti_h^{(s,a)}\sim\multi(N y_h(s,a), \pp_h(\cdot \mid s,a))$, we can write it as
\begin{align*}
    \bVarMulti_h^{(s,a)}=\sum_{i=1}^{Ny_h(s,a)} \psi^{(s,a)}_i,
\end{align*}
where the $\psi^{(s,a)}_i$'s are i.i.d.\ random vectors on $\mathbb{R}^S$ such that $\psi^{(s,a)}_i=\be_{s'}$ with probability $\pp_h(s' \mid s,a)$.
Then 
\begin{align*}
    \Xsys_{h+1}=\frac{1}{N}\sum_{s,a}\sum_{i=1}^{Ny_h(s,a)} \psi^{(s,a)}_i=\sum_{s,a}y_h(s,a)\pp_h(\cdot \mid s,a)+\frac{1}{N}\sum_{s,a}\sum_{i=1}^{Ny_h(s,a)} \left(\psi^{(s,a)}_i-\pp_h(\cdot \mid s,a)\right).
\end{align*}
Let
\begin{equation*}
    \Psi=\sum_{s,a}\sum_{i=1}^{Ny_h(s,a)} \left(\psi^{(s,a)}_i-\pp_h(\cdot \mid s,a)\right).
\end{equation*}
Note that $\Psi$ is the sum of $N$ independent random vectors $\left(\psi^{(s,a)}_i-\pp_h(\cdot \mid s,a)\right)$'s with $\left\|\psi^{(s,a)}_i-\pp_h(\cdot \mid s,a)\right\|_1\le 2$ and $\Cov{\Psi}=N\Gamma_h(\by_h)$.

Recall that 
\[\Xdiff_{h+1}^{(2)}=\sum_{s,a} y_h(s,a)\pp_h(\cdot \mid s,a)+\frac{\bZ_h'}{\sqrt{N}}=\sum_{s,a} y_h(s,a)\pp_h(\cdot \mid s,a)+\frac{1}{N}\sqrt{N}\bZ_h',\] 
where $\bZ_h' \simd \calN(\bzero, \Gamma_h(\by_h))$.
Then
\begin{align*}
    \wass{\mu_{\Xdiff_{h+1}^{(2)}}}{\mu_{\Xsys_{h+1}}}&=\frac{1}{N}\wass{\mu_{\sqrt{N}\bZ_h'}}{\mu_{\Psi}}.
\end{align*}
By \citep[Theorem 2]{valiant2010clt},
\begin{align*}
    \wass{\mu_{\sqrt{N}\bZ_h'}}{\mu_{\Psi}}\le 2S(2.7+0.93\log N).
\end{align*}
Therefore,
\[
\wass{\mu_{\Xdiff_{h+1}^{(2)}}}{\mu_{\Xsys_{h+1}}}=\calO\left(\frac{\log N}{N}\right).
\]
This completes the proof.
\end{proof}

\subsubsection{Almost Locally Lipschitz value function}

\begin{lemma}[Almost Locally Lipschitz value function] \label{lem:local-Lipschitz}
  For all $1 \le h \le H$ and $\bx, \bx'\in{\cal X}_{z_h\Ndelta}$, it holds true that
  \begin{equation*}
    |\tilde{V}(\bx, h)-\tilde{V}(\bx', h)| \leq u_h \|\bx-\bx'\|_\infty + \OlogN,
  \end{equation*}
where $u_h$ is a recursive sequence defined via $u_H = r_{\max}AS(S+1)$, and for $H-1 \ge h \ge 1$,
  \begin{equation}  \label{eq:sequence-u-h}
    u_h = r_{\max}AS(S+1) + \sqrt{S} u_{h+1} L_h (S+1),
  \end{equation}
with $L_h>0$ being the Lipschitz constant of the linear function $\phi_h(\cdot)$ defined in \eqref{ea:function-phi}.
\end{lemma}

\begin{proof}[Proof of Lemma~\ref{lem:local-Lipschitz}]
  We use backward induction. Starting with $h=H$, consider $\bx, \bx'\in{\cal X}_{H\Ndelta}.$ Without loss of generality, assume $\tilde{V}(\bx, H)>\tilde{V}(\bx',H)$: 
\begin{align*}
    \tilde{V}(\bx, H)-\tilde{V}(\bx',H)=\tilde{Q}(\bx, \tilde{\pi}^*(\bx, H), H)-\tilde{Q}(\bx', \tilde{\pi}^*(\bx', H), H) > 0.
\end{align*}
Define $\by_H = \tilde{\pi}^*(\bx, H)$, and let $\by'_H$ be the control constructed using the action mapping algorithm defined in Section \ref{sec:action-mapping}. From the definition of the value function, we have
\begin{align*}
  \tilde{V}(\bx, H)-\tilde{V}(\bx',H) \leq & \ \tilde{Q}(\bx, \by_H, H)-\tilde{Q}(\bx', \by'_H, H) \\
  = & \ \br_H \by^\top_H - \br_H \by'^\top_H \\
  \leq & \ r_{\max}AS\|\by_H - \by'_H \|_\infty \\
  \leq & \ r_{\max}AS(S+1) \|\bx-\bx'\|_\infty \\
  = & \ u_H \|\bx-\bx'\|_\infty,
\end{align*} 
where the last inequality follows from the Locally Lipschitz action mapping Lemma \ref{lem:action-mapping-construction}.

Now we prove the induction step. Consider $h<H$ and suppose that for any $\bx, \bx'\in{\cal X}_{z_{h+1}\Ndelta}$, we have 
\begin{align}  \label{eq:proof-induction-step}
    |\tilde{V}(\bx, h+1)-\tilde{V}(\bx',h+1)|\leq u_{h+1} \|\bx-\bx'\|_\infty + \OlogN.
\end{align} 

Now consider $\bx, \bx'\in{\cal X}_{z_h\Ndelta}$. Without loss of generality, assume $\tilde{V}(\bx, h)>\tilde{V}(\bx',h)$:
\begin{align*}
    \tilde{V}(\bx, h)-\tilde{V}(\bx',h)=&\tilde{Q}(\bx, \tilde{\pi}^*(\bx, h), h)-\tilde{Q}(\bx', \tilde{\pi}^*(\bx', h), h).
\end{align*} 
Define $\by_h = \tilde{\pi}^*(\bx, h)$, and let $\by'_h$ be the control constructed using the action mapping algorithm in Section \ref{sec:action-mapping}. From the definition of the value function, we have
\begin{align*}
    \tilde{V}(\bx, h)-\tilde{V}(\bx',h) &\leq  \tilde{Q}(\bx, \by_h, h)-\tilde{Q}(\bx', \by'_h, h) \\
    &=  \br_h  \by^\top_h + \expect{\tilde{V}(\bX, h+1) \mid \by_h} - \br_h  \by'^\top_h - \expect{\tilde{V}(\bX',h+1) \mid \by'_h}.
\end{align*}

Here $\Xdiff$ and $\Xdiff'$ are the states at stage $h+1$ in the two Gaussian stochastic systems, respectively, under action $\by_h$ and $\by'_h$ respectively. 
We now couple $\Xdiff$ and $\Xdiff'$ via the same Gaussian term $\bZ_h/\sqrt{N},$ i.e., 
\begin{align*}  
  \Xdiff  = & \proj{\Delta^S}{\sum_{s,a} y_h(s,a)\pp_h(\cdot \mid s,a)+\frac{\bZ_h}{\sqrt{N}}}, \\
  \Xdiff' = & \proj{\Delta^S}{\sum_{s,a} y'_h(s,a)\pp_h(\cdot \mid s,a)+\frac{\bZ_h}{\sqrt{N}}},
\end{align*} and define
\begin{equation*}
  \triangle\bx := \sum_{s,a} y_h(s,a)\pp_h(\cdot \mid s,a)-\sum_{s,a} y'_h(s,a)\pp_h(\cdot \mid s,a) = \phi_h(\by) - \phi_h(\by').
\end{equation*}
Then 
\begin{align}
    &\tilde{V}(\bx, h)-\tilde{V}(\bx',h) \\
    \leq & \quad \br_h \by^\top_h - \br_h \by'^\top_h \label{eq:aa}\\
    &\quad + \expect{ \left( \tilde{V}(\Xdiff, h+1)-\tilde{V}(\Xdiff',h+1) \right) 1_{ \left\{\Xdiff, \; \Xdiff' \in{\cal X}_{z_{h+1}\Ndelta} \right\}} \mid \tilde\bY_h=\by_h, \tilde{\bY}'_h=\by'_h} \label{eq:ab}\\
    & \quad + \expect{ \left( \tilde{V}(\Xdiff, h+1)-\tilde{V}(\Xdiff',h+1) \right) 1_{ \left\{\Xdiff \notin {\cal X}_{z_{h+1}\Ndelta} \; \mathrm{or} \; \Xdiff' \notin {\cal X}_{z_{h+1}\Ndelta} \right\} } \mid \tilde\bY_h=\by_h, \tilde{\bY}'_h=\by'_h}.\label{eq:ac}  
\end{align} 
Note that according to the Locally Lipschitz action mapping Lemma \ref{lem:action-mapping-construction}:
\begin{equation*}
  \eqref{eq:aa}\leq  \ r_{\max}AS(S+1)\|\bx-\bx'\|_\infty.
\end{equation*}

From the induction assumption \eqref{eq:proof-induction-step}, we further have 
\begin{align*}
    \eqref{eq:ab}\leq & \ \expect{ \left(u_{h+1} \|\Xdiff-\Xdiff'\|_\infty + \OlogN\right) 1_{ \left\{\Xdiff, \; \Xdiff' \in{\cal X}_{z_{h+1}\Ndelta} \right\}} \mid \tilde\bY_h=\by_h, \tilde{\bY}'_h=\by'_h}\\
    \leq_{(a)} & \ \sqrt{S} u_{h+1} \expect{ \|\triangle \bx\|_\infty  1_{ \left\{\Xdiff, \; \Xdiff' \in{\cal X}_{z_{h+1}\Ndelta} \right\}} \mid \tilde\bY_h=\by_h, \tilde{\bY}'_h=\by'_h}+ \OlogN\\
    \leq_{(b)} & \ \sqrt{S} u_{h+1} L_h \|\by_h - \by'_h \|_\infty+ \OlogN
\end{align*} where inequality (a) holds due to the coupling we defined and the fact that the projection mapping $\bx \to \proj{\Delta^S}{\bx}$ is 1-Lipschitz under the $\norm{\cdot}_2$ norm, the additional $\sqrt{S}$ factor is due to the $\norminf{\cdot}$ norm that we use here (see the deduction of Equation~\eqref{eq:middle-step-3} for a similar situation); and inequality (b) holds because $L_h > 0$ is the Lipschitz constant of the function $\phi_h(\cdot)$. Finally, 
\begin{align*}
    \eqref{eq:ac}\leq & \ Hr_{\max} \expect{1_{ \left\{\Xdiff \notin {\cal X}_{z_{h+1}\Ndelta} \; \mathrm{or} \; \Xdiff' \notin {\cal X}_{z_{h+1}\Ndelta} \right\} } \mid \tilde\bY_h=\by_h, \tilde{\bY}'_h=\by'_h}\\
    \leq & \ Hr_{\max} \proba{\Xdiff \notin {\cal X}_{z_{h+1}\Ndelta}|\tilde{\bY}_h=\by_h \in {\cal Y}_{\bx}\cap {\cal Y}_{\kappa z_h\delta_N}, \Xdiff=\bx\in{\cal X}_{z_h\delta_N}}\\
    & \ + Hr_{\max} \proba{\Xdiff' \notin {\cal X}_{z_{h+1}\Ndelta}|\tilde{\bY}'_h=\by'_h \in {\cal Y}_{\bx}\cap {\cal Y}_{\kappa z_h\delta_N}, \Xdiff'=\bx'\in{\cal X}_{z_h\delta_N}}\\
    = & \ \OlogN,
\end{align*} where the last equality follows from \eqref{eq:middle-step-1}.

Summarizing the results above, we obtain 
\begin{align*}
    \tilde{V}(\bx, h)-\tilde{V}(\bx',h)
        \leq & \ r_{\max}AS(S+1)\|\bx-\bx'\|_\infty + \sqrt{S} u_{h+1} L_h \|\by_h - \by'_h \|_\infty + \OlogN \\
    \leq_{(a)} & \ r_{\max}AS(S+1)\|\bx-\bx'\|_\infty + \sqrt{S} u_{h+1} L_h (S+1) \|\bx-\bx'\|_\infty + \OlogN \\
    =_{(b)} & \ u_h \|\bx-\bx'\|_\infty + \OlogN,
\end{align*} 
where inequality (a) follows from the Locally Lipschitz action mapping Lemma \ref{lem:action-mapping-construction}, and equality (b) holds by definition \eqref{eq:sequence-u-h} of the sequence $u_h$.
\end{proof}

\subsubsection{Locally Lipschitz action mapping}
\label{sec:action-mapping}

To complete the proof of $\tilde{V}(\bx, h)$ being locally Lipschitz, we first present an action mapping algorithm, which is locally Lipschitz, maps an action $\by_h$ for state $\bx$ to an action $\by'_h$ for $\bx'$.

The following algorithm maps an action $\by_h$ for $\bx$ to an action $\by'_h$ for $\bx'.$

{\bf Action mapping algorithm:}
\begin{itemize}
    \item {\bf Initialization:} Given $\bx,$ $\bx'$ and $\by_h$, initialize $\by'_h$ such that 
$$y'_h(s,1)=\min\{y_h(s,1),x'(s)\}\quad\hbox{and}\quad y'_h(s,0)=x'(s)-y'_h(s,1).$$ 
\item {\bf Adjustment:} Iterate each state $1 \le s \le S$ once as follows:
\begin{itemize}
    \item {\bf Case 1:} If $y_h(s,1)\leq x'(s)$ and $\by'_h(s,\cdot)\in{\cal Y}^s_{\kappa z_h \Ndelta},$ then skip. 
    \item {\bf Case 2:} If $y_h(s,1)\leq x'(s)$ and $\by'_h(s,\cdot)\notin{\cal Y}^s_{\kappa z_h\Ndelta},$ then consider two sub-cases:
    \begin{itemize}
        \item {\bf Sub-case 1:} If $y'_h(s,0)>y^*_h(s,0)+\kappa z_h\Ndelta,$ then let $\epsilon = y'_h(s,0)-y^*_h(s,0)-\kappa z_h\Ndelta>0$ and find state $\hat{s}$ such that $y'_h(\hat{s},1)>y^*_h(\hat s,1)+2z_h \Ndelta.$ 
        \item {\bf Sub-case 2:} If $y'_h(s,0)<y^*_h(s,0)-\kappa z_h\Ndelta,$ then let $\epsilon = y'_h(s,0)-y^*_h(s,0)+\kappa z_h\Ndelta<0$ and find state $\hat{s}$ such that $y'_h(\hat{s},1)<y^*_h(\hat s,1)-3z_h \Ndelta$ 
    \end{itemize}
    Adjust $\by'_h(s,\cdot)$ and $\by'_h(\hat s, \cdot)$ as follows: 
        \begin{align*}
        y'_h(s,0)\leftarrow y'_h(s,0)-\epsilon\quad \hbox{and}\quad y'_h(s,1)\leftarrow y'_h(s,1)+\epsilon\\
        y'_h(\hat s,0)\leftarrow y'_h(\hat s,0)+\epsilon\quad \hbox{and}\quad y'_h(\hat s,1)\leftarrow y'_h(\hat s,1)-\epsilon.
    \end{align*}

    \item {\bf Case 3:}  If $y_h(s,1)> x'(s),$ let $\epsilon=y_h(s,1)-x'(s)$ and find $\hat{s}$ that satisfies one of two conditions: (1) $y'_h(\hat s,0)\geq 2 z_h \Ndelta$ and $y^*_h(\hat s,0)=0$ or (2) $y'_h(\hat s, 1)\leq y^*_h(\hat s,1)+\frac{1}{2}\kappa z_h\Ndelta $ and $y_h^*(\hat s,0)>0.$ Then make the following change: 
$$y'_h(\hat{s},0)\leftarrow y'_h(\hat{s},0)-\epsilon \quad\hbox{and}\quad  y'_h(\hat{s},1)\leftarrow y'_h(\hat{s},1)+\epsilon.$$
\end{itemize}
\end{itemize}

\begin{lemma}[Locally Lipschitz action mapping]   \label{lem:action-mapping-construction}
Given $\bx, \bx'\in{\cal X}_{z_h\Ndelta}$ and $\by_h \in {\cal Y}_{\kappa z_h\Ndelta}\bigcap{\cal Y}_{\bx}$, the vector $\by'_h$ obtained by the action mapping algorithm above satisfies  $\by'_h\in{\cal Y}_{\kappa z_h\Ndelta}\bigcap{\cal Y}_{\bx'}$ and $$\|\by_h-\by'_h\|_\infty\leq (S+1)\|\bx-\bx'\|_\infty,$$ i.e. the action mapping is locally Lipschitz.      
\end{lemma}
\begin{proof}
The first step is to show that the algorithm is feasible, i.e. $\hat s$ defined in Case 2 and Case 3 always exists.  To prove this, we first observe that after initialization, 
\begin{equation}
    \alpha-2Sz_h \Ndelta\leq \sum_s y'_h(s,1)\leq \alpha \label{eq:pull}
\end{equation}
because 
\begin{align*}
  0 \leq \sum_s y_h(s,1)-\sum_sy'_h(s,1) & \ = \sum_s y_h(s,1)-\sum_s\min \{y_h(s,1), x'(s)\} \\
   & \ \leq \sum_{s:y_h(s,1)>x'(s)} x(s)-x'(s)\leq S\|\bx-\bx'\|_\infty \\
   & \ \leq 2Sz_h \Ndelta.
\end{align*}

We now analyze the feasibility of each case:
\begin{itemize}
    \item {\bf Case 2.1}: Note that  $y'_h(s,0)>y^*_h(s,0)+\kappa z_h\Ndelta$ implies 
    \begin{align*}
        y'_h(s,1) = x'(s)-y'_h(s,0) & \ \leq x^*_h(s) + z_h\Ndelta-y^*_h(s,0)-\kappa z_h \Ndelta \\
        & \ = y^*_h(s,1)-(\kappa-1)z_h\Ndelta.
    \end{align*} Now suppose $y'_h(\hat{s}, 1)\leq y_h^*(\hat s, 1)+2z_h \Ndelta$ for all $\hat s\not=s$ then 
    $$\sum_s y'_h(s,1)\leq \alpha +2Sz_h \Ndelta - (\kappa-1)z_h\Ndelta,$$ which contradicts \eqref{eq:pull} since $\kappa>1+6S$. We note that $$y'_h(\hat s, 1)>2z_h \Ndelta$$ and $\epsilon =\Delta(s,0)\leq 2z_h \Ndelta$ according to Lemma \ref{lem:delta} stated and proven below, so $y'_h(\hat s, \cdot)\geq {\bf 0}$ after the adjustment.   

    \item {\bf Case 2.2}: Note that  $y'_h(s,0)<y^*_h(s,0)-\kappa z_h\Ndelta$ implies 
    \begin{align*}
        y'_h(s,1) = x'(s)-y'_h(s,0) & \ \geq x^*_h(s)-z_h\Ndelta-y^*_h(s,0)+\kappa z_h \Ndelta \\
        & \ = y^*_h(s,1)+(\kappa-1)z_h\Ndelta.
    \end{align*} Now suppose $y'_h(\hat{s}, 1)\geq y_h^*(\hat s, 1)-3z_h \Ndelta$ for all $\hat s\not=s,$ then 
    $$\sum_s y'_h(s,1)\geq \alpha + (\kappa-1)z_h\Ndelta-3Sz_h \Ndelta,$$ which contradicts \eqref{eq:pull}. We note that $$y'_h(\hat s, 0)=x'(\hat s)-y'_h(\hat s, 1)\geq x^*_h(\hat s)-z_h \Ndelta - y_h^*(\hat s, 1)+3\Ndelta h=y_h^*(\hat s, 0)+2z_h \Ndelta$$ and $\epsilon =\Delta(s,0)\leq 2z_h \Ndelta$ according to Lemma \ref{lem:delta}, so $y'_h(\hat s, \cdot)\geq {\bf 0}$ after the adjustment. 

    \item {\bf Case 3:} Note that in this case $y^*_h(s, 0)=0$ because otherwise, 
    \begin{align*}
      x'(s)\geq x_h^*(s)-z_h \Ndelta & \ = y_h^*(s,1)+y_h^*(s,0)-z_h\Ndelta \\
      & \ \geq y_h(s,1)-\kappa z_h\Ndelta +y_h^*(s,0)-z_h\Ndelta \\
      & \ >y_h(s,1),
    \end{align*}
     which is a contradiction. Therefore given $\alpha<1,$ there exists $\hat s\not=s$ such that $y^*_h(\hat s,0)>0.$ Note that $y'_h(s,1)=x'(s)\geq x_h^*(s)-z_h\Ndelta \geq y_h^*(s,1)-z_h\Ndelta$. Now Suppose no $\hat s=s$ satisfies condition (1) or condition (2). Then for $\hat{s}$ such that $y^*_h(\hat s, 0)=0,$ 
    $$y'_h(\hat s, 1)=x'(\hat s)-y'_h(\hat s, 0)\geq x_h^*(\hat s)-z_h\Ndelta -2z_h\Ndelta =y_h^*(\hat s, 1)-3z_h\Ndelta,$$ which implies that 
    \begin{align*}
      & \ y'_h(s,1)+\sum_{\hat s\not= s}y'_h(\hat s, 1) \\
    = & \ y'_h(s,1)+\sum_{\hat s\not= s, y^*_h*(\hat s, 0)=0}y'_h(\hat s, 1)+\sum_{\hat s\not= s, y^*_h*(\hat s, 0)>0}y'_h(\hat s, 1)\\
    > & \ \alpha - z_h\Ndelta +\frac{1}{2}\kappa z_h\Ndelta-(S-2)3z_h\Ndelta = \alpha+\left(\frac{1}{2}\kappa -3S+5\right)z_h\Ndelta,
    \end{align*} which contradicts \eqref{eq:pull}. Note in this case, we increase $y'_h(\hat s, 1)$ and decrease $y'_h(\hat s, 0),$ so if $\hat s$ satisfies condition (1), then $\by'_h(\hat s, \cdot)\geq {\bf 0}$ after adjustment because $\epsilon=y_h(s,1)-x'(s)\leq x(s)-x(s')\leq 2z_h\Ndelta$. If $\hat s$ satisfies condition (2), then 
    \begin{align*}
      y'_h(\hat s,0) = x'(\hat s)-y'_h(\hat s,1) & \ \geq x^*_h(\hat s)-h\Ndelta -y^*_h(\hat s, 1)-\frac{1}{2}\kappa z_h \Ndelta \\
      & \ = y_h^*(\hat s, 0)-\left(\frac{1}{2}\kappa+1\right)z_h\Ndelta \\
      & \ =\Theta(1),
    \end{align*}
    so $\by'_h(\hat s, \cdot)\geq {\bf 0}$ after adjustment. 
\end{itemize}
Since the algorithm is always feasible, according to the design of the algorithm, we have $$\sum_s y'_h(s,1)=\alpha,$$ i.e., $\by'_h\in {\cal Y}_{\bx'}.$

The second step of the proof is to show that $$\|\by_h-\by'_h\|_\infty\leq (S+1)\|\bx-\bx'\|_\infty$$ and $$\by'_h\in{\cal Y}_{\kappa z_h\Ndelta}.$$ 
Define $$\Delta(s,a)=\max\{0,y'_h(s,a)-y^*_h(s,a)-\kappa z_h\Ndelta, \; y^*_h(s,a)-\kappa z_h\Ndelta-y'_h(s,a)\},$$ which measures the distance of $y'_h(s,a)$ towards the constrained action set ${\cal Y}_{\kappa z_h \Ndelta.}$  In Lemma \ref{lem:delta}, we will show that ${\bf \Delta}={\bf 0}$ after the adjustment step, which implies that $\by'_h\in{\cal Y}_{\kappa z_h \Ndelta},$ i.e. action $\by'_h$ is in the feasible action set. 

We note that $|y_h(s,\cdot)-y'_h(s,\cdot)|\leq |x(s)-x'(s)|$ for all $s$ because 
 \begin{itemize}
    \item If $y'_h(s,1)=y_h(s,1),$ then $y'_h(s,0)-y_h(s,0)=x'(s)-y_h(s,1)-y_h(s,0)=x'(s)-x(s).$ 
    \item If $y'_h(s,1)=x'(s),$ which occurs $x'(s)<y_h(s,1),$ then $0\leq y_h(s,1)-y'_h(s,1)\leq x(s)-x'(s)$ and $0\leq y_h(s,0)-y'_h(s,0)=x(s)-y_h(s,1)\leq x(s)-x'(s).$ 
\end{itemize}
Since after visiting state $s$, each element of $y'_h$ changes by at most either $\Delta(s,a)\leq \|\bx-\bx'\|_\infty$ or $|y_h(s,1)-x'(s)|\leq \|\bx-\bx'\|_\infty$ according to Lemma \ref{lem:delta}. Therefore, 
$$\|\by_h-\by'_h\|_\infty\leq (S+1)\|\bx-\bx'\|_\infty.$$
\end{proof}

Recall that $\Delta(s,a)=\max\{0, \; y'_h(s,a)-y^*_h(s,a)-\kappa z_h\Ndelta, \; y^*_h(s,a)-\kappa z_h\Ndelta-y'_h(s,a)\}$ and $$\|\bx-\bx'\|_\infty\leq \|\bx-\bx^*\|_\infty+\|\bx'-\bx^*\|_\infty\leq  2z_h\Ndelta.$$
\begin{lemma}
\begin{itemize}
    \item ${\Delta}(s, 1)=0$ and $\Delta(s, 0)\leq |x(s)-x'(s)|$ for all $s$ after the initiation step. 
    \item $\Delta(s,a)$ is non-increasing for all $(s,a)$ at each step of the adjustment phase. 
    \item After state $s$ is visited during the adjustment phase, $\Delta(s,0)=\Delta(s,1)=0.$
\end{itemize} \label{lem:delta}
\end{lemma}
\begin{proof}
Note that after the initialization, 
if $y'_h(s,1)=y_h(s,1),$ then $\Delta(s,1)=0$ because $\by_h \in {\cal Y}_{\kappa z_h \Ndelta},$ and 
$$y'_h(s,0)=x'(s)-y_h(s,1)=x'(s)-x(s)+y_h(s,0),$$ which implies that $\Delta(s,0)\leq |x(s)-x'(s)|$ because $\by_h \in {\cal Y}_{\kappa z_h \Ndelta}.$ 

If $y'_h(s,1)=x'(s)<y_h(s,1),$ then $$y_h(s,1)>y'_h(s,1)\geq x_h^*(s) - z_h\Ndelta \geq y^*_h(s,1) - z_h\Ndelta,$$ so $\Delta(s,1)=0.$ Furthermore, $$0=y'_h(s,0)=y_h(s,0)-y_h(s,0)=y_h(s,0)-x(s)+y_h(s,1)\geq y_h(s,0)-(x(s)-x'(s)),$$ so $\Delta(s,0)\leq |x(s)-x'(s)|$ because $\by\in {\cal Y}_{\kappa z_h \Ndelta}.$

We will now show $\Delta(s,a)$ is non-increasing in the adjustment step by analyzing each case. 
\begin{itemize}
    \item {\bf Case 2.1:} For $s,$ $\Delta(s,0)$ decreases by $\epsilon=\Delta(s,0),$ so $\Delta(s,0)=0.$ Since 
    $$y'_h(s,1)=x'(s)-y'_h(s,0)\leq x_h^*(s) + z_h\Ndelta - y_h^*(s,0)-\kappa z_h \Ndelta=y^*_h(s,1)-(\kappa-1)z_h\Ndelta,$$ we still have $\Delta(s,1)=0$ after increasing $y'_h(s,1)$ by $\Delta(s,0)\leq 2z_h\Ndelta$. For $\hat s$, we have $\Delta(\hat{s},1)=0$ after decreasing $y'_h(\hat{s},1)>y_h^*(s,1)+2z_h\Ndelta$ by $\Delta(s,0).$ Since $$y'_h(\hat s, 0)=x'(\hat s)-y'_h(\hat s,1)\leq x^*_h(s)+z_h \Ndelta - y_h^*(s,1)-2z_h\Ndelta =y^*_h(s,0)-z_h \Ndelta,$$ increasing its value by $\Delta(s,0)$ only decrease $\Delta(\hat s, 0).$ 

    \item {\bf Case 2.2:} For $s,$ $\Delta(s,0)$ decreases by $\epsilon=\Delta(s,0),$ so $\Delta(s,0)=0.$ Since 
    $$y'_h(s,1)=x'(s)-y'_h(s,0)\geq x_h^*(s)-z_h\Ndelta - y_h^*(s,0)+\kappa z_h \Ndelta=y^*_h(s,1)+(\kappa-1)z_h\Ndelta,$$ we still have $\Delta(s,1)=0$ after decreasing $y'_h(s,1)$ by $\Delta(s,0)\leq 2z_h\Ndelta$. For $\hat s$, we have $\Delta(\hat{s},1)=0$ after increasing $y'_h(\hat{s},1)<y_h^*(s,1)-3z_h\Ndelta$ by $\Delta(s,0).$ Since $$y'_h(\hat s, 0)=x'(\hat s)-y'_h(\hat s,1)\geq x^*_h(s)-z_h \Ndelta -y_h^*(s,1)+3z_h\Ndelta=y^*_h(s,0)+2z_h \Ndelta,$$ decreasing its value by $\Delta(s,0)$ only decrease $\Delta(\hat s, 0).$ 

    \item {\bf Case 3:}   For $\hat s$ that satisfies condition (1), decreasing the value of $y'_h(\hat s, 0)$ decreases $\Delta(\hat s, 0)$ because $y_h^*(\hat s, 0)=0.$ Furthermore, $$y'_h(\hat s, 1)=x'(\hat s )-y'_h(\hat s, 0)\leq x_h^*(\hat s ) + z_h\Ndelta - 2z_h\Ndelta =y_h^*(\hat s, 1) - z_h\Ndelta,$$ so $\Delta(\hat s, 1)$ remains zero after increasing its value by no more than $|x(s)-x'(s)|\leq 2z_h\Ndelta.$ 
    
    For $\hat s$ that satisfies condition (2), we have $\Delta(\hat{s},1)=0$ after increasing $y'_h(\hat{s},1)<y_h^*(s,1)+\frac{1}{2}\kappa z_h \Ndelta$ by $|y_h(s,1)-x'(s)|\leq |x(s)-x'(s)|\leq 2 z_h \Ndelta.$ Since 
    \begin{align*}
      y'_h(\hat s, 0) = x'(\hat s)-y'_h(\hat s,1) & \ \geq x^*_h(s)-z_h \Ndelta-y_h^*(s,1)-\frac{1}{2}\kappa z_h\Ndelta \\
      & \ = y^*_h(s,0)-\left(1+\frac{1}{2}\kappa\right)z_h\Ndelta,
    \end{align*}
    decreasing its value by no more than $|x(s)-x'(s)|\leq 2 z_h\Ndelta$ can only decrease $\Delta(\hat s, 0).$ 
\end{itemize}

\end{proof}

\subsection{Proof of Lemma \ref{lem:ed} (Evaluation difference)}

The proof is almost identical, but simpler, to that of Lemma \ref{lem:vd}, and is based on backward induction. The only difference is 
\begin{align*}
V^N_{\tilde{\pi}^{N,*}}(\bx,h)-\tilde{V}^N_{\tilde{\pi}^{N,*}}(\bx,h)={Q}^N_{\tilde{\pi}^{N,*}}(\bx, \tilde{\pi}^{N,*}(\bx,h), h)-\tilde{Q}^N_{\tilde{\pi}^{N,*}}(\bx, \tilde{\pi}^{N,*}(\bx,h),h),
\end{align*} 
because we are considering the same policy which is evaluated in two systems.

\subsection{High probability bounds}  \label{sec:high-prob-bound}

\begin{restatable}[High probability bound]{lemma}{HighProbaBound} \label{lem:high-proba-bound}
\begin{enumerate}
  \item  \label{lem-high-proba-item-1}  When applying a policy $\pi\in \pidelta$ on the $N$-system following state transition \eqref{eq:rv-Nsys}, its state trajectory $\bX^{\pi}_h$, $1 \le h \le H$, satisfies:
\begin{equation*}
  \proba{\|\bX^\pi_h-\bx^*_h\|_\infty \leq z_h \Ndelta, \; \forall h} = 1 - \OlogN.
\end{equation*}
  \item \label{lem-high-proba-item-2}  When applying a policy $\tilde{\pi} \in \pidelta$ on the Gaussian stochastic system following state transition \eqref{eq:X-c}, its state trajectory $\Xdiff^{\tilde{\pi}}_h$ and action $\bY_h = {\tilde{\pi}}(\tilde{\bX}^{{\tilde{\pi}}}_h)$, $1 \le h \le H$ satisfy:
  \begin{align*}
    & \ \mathbb{P} \Big( \tilde{\bX}^{{\tilde{\pi}}}_{h+1}=\sum_{s,a} Y_h(s,a)\pp_h(\cdot \mid s,a)+\frac{\bZ_h}{\sqrt{N}}  \mid  \|\Xdiff^{\tilde{\pi}}_h-\bx^*_h\|_\infty \leq z_h \Ndelta \Big) \\
    = & \ \mathbb{P} \Big( \sum_{s,a} Y_h(s,a)\pp_h(\cdot \mid s,a)+\frac{\bZ_h}{\sqrt{N}} \geq \bzero  \mid  \|\Xdiff^{\tilde{\pi}}_h-\bx^*_h\|_\infty \leq z_h \Ndelta \Big) \\
    = & \  1 - \OlogN,
  \end{align*}
  and
   \begin{align*}
    \proba{\|\Xdiff^{\tilde{\pi}}_h-\bx^*_h\|_\infty \leq z_h \Ndelta, \; \forall h} = 1 - \OlogN \geq 1-\frac{\tilde{C}}{N^{\log N}},
  \end{align*} 
  where $\tilde{C}>0$ is a constant independent of $N.$
\end{enumerate}
\end{restatable}

\paragraph{Proof of Item~\ref{lem-high-proba-item-1}.}

We recall the definition of the policy class $\pidelta$:
\begin{align*}
& \pidelta :=   \\
& \big\{ \pi: \|\pi(\bx,h)-\by^*_h\|_\infty\leq \kappa z_h \Ndelta,  \ \hbox{ if }  \|\bx-\bx^*_h\|_\infty \leq z_h \Ndelta; \\
& \qquad \hbox{follow a predefined policy},  \ \hbox{ if }  \|\bx-\bx^*_h\|_\infty > z_h \Ndelta \big\}.
\end{align*}
where
\begin{equation} \label{eq:recursive-relation}
  z_1 = 1; \; z_{h+1} = \sqrt{S} \ (\kappa L_h z_h + 1), \; 1 \le h \le H-1.
\end{equation}

In the following analysis, we fix a policy $\pi \in \pidelta$, and define the event $\calE$ when we apply $\pi$ on the $N$-system for $H$ steps:
\begin{equation}\label{eq:event-calE}
  \calE := \left\{ \forall h: \norminf{\bX^{\pi}_h - \bx^*_h} \le z_h \Ndelta \right\}. 
\end{equation}
Our goal is to prove the following high probability bound:
\begin{equation}  \label{eq:high-proba-bound}
  \p(\overline{\calE}) = \OlogN.
\end{equation}

For $1 \le h \le H-1$, to ease the notation let us write $\bx = \bX^{\pi}_h$ and $\by_h = \pi(\bx,h)$. There are $N y(s,a)$ arms in state $s$ whose action is $a$, and each of these arms makes a transition to state $s'$ with probability $\pp_h(s' \mid s,a)$, independently of each other. Hence each coordinate $s'$ of the vector $\bX^{\pi}_{h+1}$ can be written as a sum of $N$ random variables as:
\begin{equation*}
  X^{\pi}_{h+1}(s') =  \frac{1}{N} \sum_{n=1}^{N} 1_{ \left\{ U_n \, \le \, \pp_h(s' \mid s_n,a_n) \right\} },
\end{equation*}
where the $U_n$'s are a total number of $N$ i.i.d. uniform $[0,1]$'s, and $(s_n,a_n)$ is the state-action pair of the $n$-th arm from the aggregated $(\bx,\by)$, for $1 \le n \le N$. Note that
\begin{equation*}
  \expect{ X^{\pi}_{h+1}(s') \mid \by} = \phi_h(\by)(s').
\end{equation*}
Hence, let us write 
\begin{equation*}
  \bE_h := \bX^{\pi}_{h+1} - \phi_h(\by).
\end{equation*}
By using Hoeffding's inequality on $N$ independent Bernoulli random variables, together with a union bound, we deduce that for $\theta > 0$:
\begin{align} \label{eq:hoeffding-bound}
  \proba{\norminf{\bE_h} > \theta} = \proba{\max_{1 \le s' \le S} |E_h(s')| > \theta }  \le 2S \exp\left(-2N \theta^2\right).
\end{align}
Consider the event
\begin{equation}\label{eq:event_calE-prime}
  \calE' := \left\{ \forall 1 \le h \le H-1: \norminf{\bE_h} \le \Ndelta \right\}. 
\end{equation}
Now the claim in \eqref{eq:high-proba-bound} follows from two facts:
(i) $\calE' \subseteq \calE$; (ii) $\proba{\overline{\calE'}} = \OlogN$. 

Fact (i) follows from the definition of the sequence $z_h$ via a recursive argument. Indeed, conditional on $\calE'$, and suppose that $\norminf{\bx - \bx^*_h} \le z_h \Ndelta$ and $\norminf{\by_h - \by^*_h} \leq \kappa z_h \Ndelta$, then
\begin{align*}
  \norminf{\bX^{\pi}_{h+1} - \bx^*_{h+1}} = \norminf{ \phi_{h}(\by) + \bE_h - \phi_h(\by^*_h) } \le \kappa L_h z_h \Ndelta + \Ndelta \le z_{h+1} \Ndelta.
\end{align*}
Fact (ii) follows from \eqref{eq:hoeffding-bound} by taking $\theta = \Ndelta= \frac{2 \log N}{\sqrt{N}}$ and a union bound:
\begin{equation*}
  \proba{\overline{\calE'}} \le \sum_{h=1}^{H} \proba{\norminf{\be(h)} > \Ndelta} \le  2HS e^{-2N \Ndelta^2} = \frac{2HS}{N^{8 \log N}} = \OlogN.
\end{equation*}

\paragraph{Proof of Item~\ref{lem-high-proba-item-2}.}

By its construction in \eqref{eq:covariance-matrix} and \eqref{eq:Gamma-definition}, each row of $\Gamma_h(\by^*_h)$ sums to $0$, so for $\bZ_h \simd \calN(\bzero, \Gamma_h(\by^*))$, the sum of its coordinates follows a Gaussian distribution with zero mean and zero variance, i.e. is equal to $0$ almost surely. 

Now let us consider a policy ${\tilde{\pi}} \in \pidelta$ and write $\bY_h = {\tilde{\pi}}(\tilde{\bX}^{{\tilde{\pi}}}_h)$. We prove the first equality by showing that the following two events coincide:
\begin{equation} \label{eq:equivalence-of-events}
  \left\{ \Xdiff^{{\tilde{\pi}}}_{h+1} = \phi_h(\bY_h)+\frac{\bZ_h}{\sqrt{N}} \right\} = \left\{ \phi_h(\bY_h)+\frac{\bZ_h}{\sqrt{N}} \ge \bzero \right\}.
\end{equation}
Indeed, if $\phi_h(\bY_h)+\frac{\bZ_h}{\sqrt{N}} \ge \bzero$, then $\phi_h(\bY_h)+\frac{\bZ_h}{\sqrt{N}}$ already lies in the simplex $\Delta^S$, and hence $\Xdiff^{{\tilde{\pi}}}_{h+1} = \phi_h(\bY_h)+\frac{\bZ_h}{\sqrt{N}}$. Conversely, if $\Xdiff^{{\tilde{\pi}}}_{h+1} = \proj{\Delta^S}{\phi_h(\bY_h)+\frac{\bZ_h}{\sqrt{N}}} = \phi_h(\bY_h)+\frac{\bZ_h}{\sqrt{N}}$, then clearly $\phi_h(\bY_h)+\frac{\bZ_h}{\sqrt{N}} \ge \bzero$.

We next prove that for $1 \le h \le H-1$, the following two inequalities hold:
\begin{align}  \label{eq:middle-step-1}
 \proba{ \|\tilde{\bX}^{{\tilde{\pi}}}_{h+1}-\bx_{h+1}^*\|_\infty \leq z_{h+1} \Ndelta \mid \| \tilde{\bX}^{{\tilde{\pi}}}_h-\bx_h^*\|_\infty\leq z_h \Ndelta} =  \ 1 - \OlogN,
\end{align}
and with $\bY_h = {\tilde{\pi}}(\tilde{\bX}^{{\tilde{\pi}}}_h)$ which by definition of ${\tilde{\pi}}$ satisfies $\bY_h \in \calY_{\bx} \cap \calY_{\kappa z_h \Ndelta}$, that
\begin{align} \label{eq:middle-step-2}
 \proba{ \phi_h(\bY_h)+\frac{\bZ_h}{\sqrt{N}}\geq \bzero \mid \|\tilde{\bX}^{{\tilde{\pi}}}_h-\bx_h^*\|_\infty\leq z_h \Ndelta} 
 =  \ 1 - \OlogN.
\end{align}
Recall that $\bZ_h \simd \calN(\bzero, \Gamma_h)$. Let us assume subsequently that $\max_{s'} \Gamma_h(s',s')>0$. Otherwise, when $\max_{s'} \Gamma_h(s',s')=0$, the vector $\bZ_h={\bf 0}$ almost surely, so \eqref{eq:middle-step-1} and \eqref{eq:middle-step-2} hold trivially. 

For convenience, let us define the following events:
\begin{align*}
  \calE_h & \ := \left\{ \norminf{\frac{\bZ_h}{\sqrt{N}}} \le \Ndelta \right\}, \; \hbox{ for } 1 \le h \le H-1; \\
  \calB_h & \ := \left\{ \phi_h(\bY_h)+\frac{\bZ_h}{\sqrt{N}}\geq \bzero \right\}, \; \hbox{ for } 1 \le h \le H-1; \\
  \calD_h & \ := \left\{ \|\tilde{\bX}^{{\tilde{\pi}}}_h-\bx_h^*\|_\infty\leq z_h \Ndelta \right\}, \; \hbox{ for } 1 \le h \le H.
\end{align*}

The proof is based on the following three key ingredients, which will be proven later:
\begin{enumerate}
  \item \begin{equation}  \label{eq:Gaussian-tail-bound}
          \proba{ \overline{\calE_h}  \mid \calD_h } = \OlogN.
        \end{equation}
  \item \begin{equation}  \label{eq:ingredient-2}
          \proba{\calD_{h+1} \mid \calE_h \cap \calD_h}=1.
        \end{equation}
  \item \begin{equation}   \label{eq:ingredient-3}
          \proba{\calB_{h} \mid \calE_h \cap \calD_h}=1.
        \end{equation}
\end{enumerate}

Based on these results, \eqref{eq:middle-step-1} follows by noting that
\begin{align*}
  & \proba{\|\tilde{\bX}^{{\tilde{\pi}}}_{h+1}-\bx_{h+1}^*\|_\infty \leq z_{h+1} \Ndelta \mid \| \tilde{\bX}^{{\tilde{\pi}}}_h-\bx_h^*\|_\infty\leq z_h \Ndelta} \\
 = & \ \proba{\calD_{h+1} \mid \calD_h} \\
 = & \ \proba{\calD_{h+1} \mid \calE_h \cap \calD_h} \proba{\calE_h \mid \calD_h} + \proba{\calD_{h+1} \mid \calE_h^c \cap \calD_h} \proba{\calE_h^c \mid \calD_h} \\
 \ge & \ \proba{\calD_{h+1} \mid \calE_h \cap \calD_h} \proba{\calE_h \mid \calD_h} \\
 =_{(a)} & \ \proba{\calE_h \mid \calD_h} \\
=_{(b)} & \ \ 1 - \OlogN,
\end{align*}
where equality (a) follows from \eqref{eq:ingredient-2} and equality (b) follows from \eqref{eq:Gaussian-tail-bound}.

As for \eqref{eq:middle-step-2}, we have
\begin{align*}
  & \proba{\phi_h(\bY_h)+\frac{\bZ_h}{\sqrt{N}}\geq \bzero \mid \| \tilde{\bX}^{{\tilde{\pi}}}_h-\bx_h^*\|_\infty\leq z_h \Ndelta} \\
 = & \ \proba{\calB_{h} \mid \calD_h} \\
 = & \ \proba{\calB_{h} \mid \calE_h \cap \calD_h} \proba{\calE_h \mid \calD_h} + \proba{\calB_h \mid \calE_h^c \cap \calD_h} \proba{\calE_h^c \mid \calD_h} \\
 \ge & \ \proba{\calB_{h} \mid \calE_h \cap \calD_h} \proba{\calE_h \mid \calD_h} \\
 =_{(a)} & \ \proba{\calE_h \mid \calD_h} \\
=_{(b)} & \ \ 1 - \OlogN,
\end{align*}
where equality (a) follows from \eqref{eq:ingredient-3} and equality (b) follows from \eqref{eq:Gaussian-tail-bound}.

Since $\tilde{\bX}^{\tilde{\pi}}_1=\bx^*_1,$ which is the initial state, so $\proba{\calD_1}=1$. Now suppose $\proba{\calD_h}=1 - \OlogN$. Then from  \eqref{eq:middle-step-1}, 
\begin{align*}
\proba{\calD_{h+1}}\geq &\proba{\calD_{h+1} \mid \calD_{h}}\proba{\calD_h}=\left(1 - \OlogN\right)\left(1 - \OlogN\right)=1 - \OlogN.
\end{align*}
Then, using the union bound, we obtain
\begin{align*}
    & \proba{\|\Xdiff^{\tilde{\pi}}_h-\bx^*_h\|_\infty \leq z_h \Ndelta, \; \forall h} \\
  = & \ \proba{\calD_H \cap \calD_{H-1} \cdots \cap \calD_1}  \\
  \geq & \ 1- \sum_h \proba{\bar \calD_h}\\
  = & \ 1-\sum_h \OlogN\\
  = & \ 1-\OlogN. 
\end{align*}

\paragraph{Proof of Claim~\eqref{eq:Gaussian-tail-bound}.}

Since events $\calE_h$ and $\calD_h$ are independent,
\begin{align*}  
 \proba{ \overline{\calE_h}  \mid \calD_h } 
= & \ \proba{\norminf{\frac{\bZ_h}{\sqrt{N}}} > \Ndelta} \nonumber \\
  = & \ \proba{\max_{1 \le s' \le S} | Z_h(s')| > \sqrt{N} \Ndelta }  \nonumber \\
\leq & \sum_{s'} \proba{| Z_h(s')| > \sqrt{N} \Ndelta }  \nonumber \\
\leq_{(a)} & \sum_{s':\Gamma_h(s',s')\not=0} 2\exp\left(-\frac{N \Ndelta^2}{2\Gamma_h(s',s')}\right)  \nonumber \\
\le & \ 2S \exp\left(-\frac{N \Ndelta^2}{2 \max_{s'} \Gamma_h(s',s')}\right) \\
\le_{(b)} & \ \frac{2S}{N^{2 \log N}} \\
= & \ \OlogN,
\end{align*}
where inequality (a) follows from the exponential bound for the tail distribution of Gaussian random variables, and inequality (b) is because $\Ndelta= \frac{2 \log N}{\sqrt{N}}$, and 
\begin{equation*}
   \Gamma_h(s',s') = \sum_{s,a}y_h^*(s,a)\pp_h(s' \mid s,a)(1-\pp_h(s' \mid s,a)) \le \sum_{s,a}y_h^*(s,a)\pp_h(s' \mid s,a) = x^*_{h+1}(s') \le 1.
\end{equation*}

\paragraph{Proof of Claim~\eqref{eq:ingredient-2}.}

Conditioned on the event $\calE_h$ and the event $\calD_h$, since $\|\Xdiff^{{\tilde{\pi}}}_{h} - \bx^*_{h} \|_\infty \leq z_h \Ndelta$, by construction of our policy ${\tilde{\pi}}$, we have $\norminf{\bY_h - \by^*_h} \leq \kappa z_h \Ndelta$. Hence
\begin{align} 
   \norminf{\Xdiff^{{\tilde{\pi}}}_{h+1} - \bx^*_{h+1}} = & \ \norminf{ \proj{\Delta^S}{\phi_h(\bY_h) + \frac{\bZ_h}{\sqrt{N}}} - \phi_h(\by^*) } \nonumber \\
 =_{(a)} & \  \norminf{ \proj{\Delta^S}{\phi_h(\bY_h) + \frac{\bZ_h}{\sqrt{N}}} - \proj{\Delta^S}{\phi_h(\by^*)}} \nonumber \\
 \le & \ \norme{ \proj{\Delta^S}{\phi_h(\bY_h) + \frac{\bZ_h}{\sqrt{N}}} - \proj{\Delta^S}{\phi_h(\by^*)}}_2 \nonumber \\
 \le_{(b)} & \ \norme{ \phi_h(\bY_h) + \frac{\bZ_h}{\sqrt{N}} - \phi_h(\by^*) }_2 \nonumber \\
 \le & \ \sqrt{S} \norminf{ \phi_h(\bY_h) + \frac{\bZ_h}{\sqrt{N}} - \phi_h(\by^*)} \nonumber \\
 \le_{(c)} & \ \sqrt{S} \ (\kappa z_h L_h \Ndelta + \Ndelta) = z_{h+1} \Ndelta.  \label{eq:middle-step-3}
\end{align}
Here equality (a) is because $\phi_h(\by^*) \in \Delta^S$; inequality (b) is because the projection mapping $\bx \to \proj{\Delta^S}{\bx}$ is 1-Lipschitz under the $\norm{\cdot}_2$ norm; inequality (c) is by definition of the event $\calE_h$ and the sequnce $z_h$.

\paragraph{Proof of Claim~\eqref{eq:ingredient-3}.}

\textbf{Case I: $x^*_{h+1}(s')=0$.}
We note that when $$x^*_{h+1}(s')=\sum_{s,a} y^*_h(s,a)\pp_h(s'\mid s,a)=0,$$ we have $$y^*_h(s,a)\pp_h(s'\mid s,a)=0, \; \forall (s,a),$$ which means $Z_h(s') \equiv 0$ because $\var{Z_h(s')} = 0$. Hence the transition in the $s'$-th coordinate of \eqref{eq:X-c} becomes deterministic, and we have
\begin{equation}  \label{eq:deterministic-transition-case}
  \sum_{s,a} Y_h(s,a)\pp_h(s' \mid s,a) \ge 0.
\end{equation}

\textbf{Case II: $x^*_{h+1}(s')>0$.} Since we have previously proven that $\proba{\calD_{h+1} \mid \calE_h \cap \calD_h}=1$, we apply \eqref{eq:middle-step-3} to see that for all state $s'$ such that $x^*_{h+1}(s')>0$, we have
\begin{equation*}
  \abs{\tilde{X}^{{\tilde{\pi}}}_{h+1}(s') - x^*_{h+1}(s')} \le z_{h+1} \Ndelta \le z_H \delta_{N},
\end{equation*}
so that
\begin{equation}  \label{eq:stochastic-transition-case}
  \tilde{X}^{{\tilde{\pi}}}_{h+1}(s') \geq x^*_{h+1}(s')- z_H \delta_{N} > 0.
\end{equation} when $N$ is sufficient large because $z_H\delta_N=\Theta(\frac{\log N}{N}).$ Therefore, we have 
\begin{align*} 
  \tilde{X}^{{\tilde{\pi}}}_{h+1}(s')=&\max\left\{\sum_{s,a} Y_h(s,a)\pp_h(s' \mid s,a)+\frac{Z_h(s')}{\sqrt{N}}, 0\right\}\\
  =&\sum_{s,a} Y_h(s,a)\pp_h(s' \mid s,a)+\frac{Z_h(s')}{\sqrt{N}}>0.
\end{align*}

By combining \eqref{eq:deterministic-transition-case} and \eqref{eq:stochastic-transition-case} we deduce that $\proba{\calB_{h} \mid \calE_h \cap \calD_h}=1$.

\section{Proof of Theorem~\ref{thm:lower-bounds}} 
\label{append:proof-of-sqrt-N-gap}

\LowerBounds*

\begin{figure}[h]
  \centering
  \includegraphics[width=0.5\linewidth]{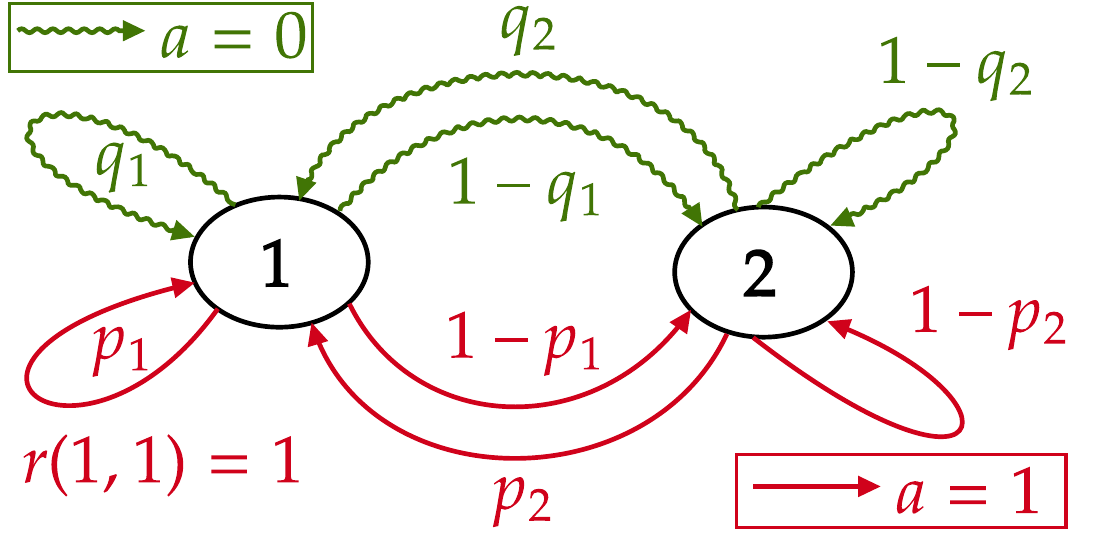}
  \caption{The one-armed MDP with two states. The only non-zero reward is under action $1$ in state $1$.}
  \label{fig:MDP}
\end{figure}

\subsubsection*{Proof of Theorem~\ref{thm:lower-bounds}.}
We prove the theorem by giving a set of RMAB instances and show that for each of them,
\begin{align*}
    \VNopt(\xini,1) - \VNLPres(\xini,1) &= \Theta(1 / \sqrt{N}),\\
    \VLP(\xini,1) - \VNopt(\xini,1) &= \Theta(1 / \sqrt{N}).
\end{align*}
We omit the $(\xini,1)$ for each value function in this proof for conciseness.

\paragraph{Constructed RMAB instances.}
This set of RMAB instances generalizes the example discussed in Section~\ref{sec:example}.
In particular, consider two-state RMABs with horizon $H=2$ and pulling budget $\alpha=0.5$.
The rewards are given by $r_1(1,1) = r_2(1,1)=1$, with all other $(h,s,a)$-tuples being $r_h(s,a)=0$. 
Let
\begin{align*}
    \pp_1(1 \mid 1,1) &= p_1, & \pp_1(1 \mid 1,0) &= q_1, \\
    \pp_1(1 \mid 2,1) &= p_2, & \pp_1(1 \mid 2,0) &= q_2, 
\end{align*}
where we assume that
\begin{equation}
    q_1+p_2-p_1-q_2 >1.
\end{equation}
One can verify that the example in Section~\ref{sec:example} satisfies this condition.
This one-armed MDP is illustrated in Figure~\ref{fig:MDP}. 
In each $N$-system, the initial condition at $h=1$ is: $N/2$ arms are in state $1$, and the other $N/2$ arms are in state $2$.

\paragraph{Optimal value function.}
Similar to the example in Section~\ref{sec:example}, the $N$-system problem can be written as
\begin{equation}
    \VNopt = \max_{0 \le Y_1(1,1) \le 0.5} Y_1(1,1) + \expect{\min \left\{0.5,X_2(1)\right\}},
\end{equation}
which we rewrite into the equivalent form below for notational convenience
\begin{equation}  \label{eq:N-sys-simple-formulation}
  \VNopt = \max_{\beta\colon 0\le\beta\le 0.5, N\beta\in\mathbb{N}} \beta + \expect{\min \left\{0.5, \frac{G(N,\beta)}{N} \right\}},
\end{equation}
where we use $\beta$ is the action $Y_1(1,1)$ and $G(N,\beta)=NX_2(1)$ given $Y_1(1,1)=\beta$.
Note that $G(N,\beta)$ can be written as
\begin{equation*}
    G(N,\beta)=U^{(1,1)}(N,\beta)+U^{(1,0)}(N,\beta)+U^{(2,1)}(N,\beta)+U^{(2,0)}(N,\beta),
\end{equation*}
where
\begin{align*}
    U^{(1,1)}(N,\beta)\sim\textrm{Binomial}(N\beta,p_1),\quad &U^{(1,0)}(N,\beta)\sim\textrm{Binomial}(N(0.5-\beta),q_1),\\
    U^{(2,1)}(N,\beta)\sim\textrm{Binomial}(N(0.5-\beta),p_2),\quad&U^{(2,0)}(N,\beta)\sim\textrm{Binomial}(N\beta,q_2),
\end{align*}
and they are independent.

\paragraph{Fluid LP.}
Based on this $N$-system problem, we can see that the fluid LP of this problem is
\begin{equation}
    \VLP=\max_{0\le\beta\le 0.5}\beta+\min\{0.5,g(\beta)\},
\end{equation}
where
\begin{equation*}
    g(\beta)=\expect{\frac{G(N,\beta)}{N}}=0.5(q_1+p_2)-(q_1+p_2-p_1-q_2)\beta,
\end{equation*}
and recall that $q_1+p_2-p_1-q_2>q_1+p_2-p_1-q_2-1>0$.
One can verify that this LP has a unique solution, and the optimal solution, denoted as $\beta^*_{\textrm{LP}}$, satisfies that $g(\beta^*_{\textrm{LP}})=0.5$.
Thus $\VLP=\beta^*_{\textrm{LP}}+0.5$.

\paragraph{Optimal value function and Q-function in the $N$-system and the Gaussian stochastic system.}
We now connect the optimal value function and Q-function in the $N$-system with those in the Gaussian stochastic system.
The Q-function of the $N$-system can be written as
\begin{equation}
    Q^N(\beta)=\beta + \expect{\min \left\{0.5, \frac{G(N,\beta)}{N} \right\}},
\end{equation}
where $0\le\beta\le 0.5$ and $N\beta$ is a nonnegative integer,
and then
\begin{equation*}
    \VNopt=\sup_{\beta\colon 0\le\beta\le 0.5, N\beta\in\mathbb{N}}Q^N(\beta).
\end{equation*}
Similarly, for the Gaussian stochastic system, the Q-function can be written as
\begin{equation}
    \tilde{Q}^N(\beta)=\beta + \expect{\min \left\{0.5, \frac{wZ}{\sqrt{N}}+g(\beta)\right\}},
\end{equation}
where $Z$ follows a standard Gaussian distribution and
\begin{equation}\label{eq:w}
  w = \sqrt{\beta^*_{\mathrm{LP}}p_1(1-p_1) + (0.5-\beta^*_{\mathrm{LP}})q_1(1-q_1) + (0.5-\beta^*_{\mathrm{LP}})p_2(1-p_2) + \beta^*_{\mathrm{LP}}q_2(1-q_2)}.  
\end{equation}
Then
\[
\VNdiffopt=\sup_{0\le\beta\le 0.5}\tilde{Q}^N(\beta).
\]

Note that prior work \citep[Theorem 1]{gast2023linear} has shown that $\VNopt-\VNLPres\le \VLP-\VNLPres=O(1/\sqrt{N})$.
So we only need to show $\VNopt-\VNLPres=\Omega(1/\sqrt{N})$.
Since the fluid LP satisfies the Uniqueness Assumption~\ref{ass:opt-lp-distance}, by Lemma~\ref{prop:sufficient-condition-of-assumption} and Lemma~\ref{lem:vd}, we know that
\begin{align*}
|Q^N(\beta)-\tilde{Q}^N(\beta)|&=\tilde{\calO}(1/N),\\
|\VNopt-\VNdiffopt|&=\tilde{\calO}(1/N).
\end{align*} 
Furthermore, since all fluid policies in $\pisub$ takes action $y_1(1,1)=\beta^*_{\mathrm{LP}}$ at time step $1$, we have
$$\VNLPres \leq {Q}^N(\beta^*_{\mathrm{LP}})= \tilde{Q}^N(\beta^*_{\mathrm{LP}})+\tilde{\calO}(1/N),$$
and thus $$\VNopt-\VNLPres\ge \VNdiffopt-\tilde{Q}^N(\beta^*_{\mathrm{LP}})+\tilde{\calO}(1/N).$$
Therefore, to prove the theorem, it suffices to prove that 
\begin{align*}
    \VNdiffopt - \tilde{Q}^N(\beta^*_{\mathrm{LP}})&= \Omega(1 / \sqrt{N}),\\
    \VLP - \VNdiffopt &= \Theta(1 / \sqrt{N}).
\end{align*}

\paragraph{Equivalent forms of the Q-functions.}
With slight abuse of notation, we now rewrite the Q-functions of the $N$-system and the Gaussian stochastic system into an equivalent form and use this form in the remainder of the proof.
We first rewrite the action $\beta$ as $\beta=\beta^*_{\textrm{LP}}+\frac{c}{\sqrt{N}}$.
Then for the $N$-system,
\begin{equation}
    Q^N(c)=\beta^*_{\textrm{LP}}+\frac{c}{\sqrt{N}}+ 0.5+\expect{\min \left\{0, \frac{G(N,\beta^*_{\textrm{LP}}+\frac{c}{\sqrt{N}})-0.5N}{N} \right\}},
\end{equation}
and thus
\begin{align*}
    \sqrt{N}\left(Q^N(c)-\beta^*_{\textrm{LP}}-0.5\right)
    &=c+\expect{\min \left\{0, \frac{G(N,\beta^*_{\textrm{LP}}+\frac{c}{\sqrt{N}})-0.5N}{\sqrt{N}} \right\}}\\
    &=c+\expect{\min\left\{0,\hat{G}(N, c)\right\}},
\end{align*}
where
\begin{equation*}
  \hat{G}(N,c) := \frac{G(N,\beta^*_{\mathrm{LP}}+\frac{c}{\sqrt{N}}) - 0.5 N}{\sqrt{N}}.
\end{equation*}
Note that by the central limit theorem, for any fixed constant $c$ independent of $N$, 
\begin{align*}
\hat{G}(N,c) &=  \frac{G(N,\beta^*_{\mathrm{LP}}+\frac{c}{\sqrt{N}}) - \expect{G(N,\beta^*_{\mathrm{LP}}+\frac{c}{\sqrt{N}})}}{\sqrt{N}}-c(q_1+p_2-p_1-q_2)\\
&\mspace{-36mu}\cvd wZ-c(q_1+p_2-p_1-q_2),
\end{align*}
where $Z$ follows a standard Gaussian distribution and $w$ is the same as the one defined in \eqref{eq:w}.

Similarly, for the Gaussian stochastic system,
\begin{align*}
    \tilde{Q}^N(c)&=\beta^*_{\textrm{LP}}+\frac{c}{\sqrt{N}}+ 0.5+\expect{\min \left\{0, \frac{wZ}{\sqrt{N}}+g\left(\beta^*_{\textrm{LP}}+\frac{c}{\sqrt{N}}\right)-0.5\right\}}\\
    &=\beta^*_{\textrm{LP}}+\frac{c}{\sqrt{N}}+ 0.5+\expect{\min \left\{0, \frac{wZ}{\sqrt{N}}-\frac{c}{\sqrt{N}}(q_1+p_2-p_1-q_2)\right\}},
\end{align*}
and thus
\begin{align*}
    \sqrt{N}\left(\tilde{Q}^N(c)-\beta^*_{\textrm{LP}} - 0.5\right)&=c+\expect{\min \left\{0, wZ-c(q_1+p_2-p_1-q_2)\right\}}.
\end{align*}

\subsection{Proof of $\VNopt - \VNLPres = \Theta(1/\sqrt{N})$ in Theorem~\ref{thm:lower-bounds}}
We just showed that it suffices to prove that $\VNdiffopt - \tilde{Q}^N(0) = \Omega(1 / \sqrt{N})$.
We first obtain the value of $\VNdiffopt$.
Note that
\begin{align*}
    \VNdiffopt &=\sup_{c\colon 0\le\beta^*_{\textrm{LP}}+\frac{c}{\sqrt{N}}\le 0.5}\tilde{Q}^N(c)
    \le \sup_{c\colon c\in\mathbb{R}}\tilde{Q}^N(c).
\end{align*}
We now solve for the $c\in\mathbb{R}$ that maximizes $\tilde{Q}^N(c)$.  
Let us use a change of variable $\tau=c(q_1+p_2-p_1-q_2)/w$ for convenience. Then
\begin{align}
    &\mspace{23mu}\sqrt{N}\left(\tilde{Q}^N(c)-\beta^*_{\textrm{LP}} - 0.5\right)\big/w\nonumber\\
    &=c/w+\expect{\min \left\{0, Z-c(q_1+p_2-p_1-q_2)/w\right\}}\nonumber\\
    &=-c(q_1+p_2-p_1-q_2-1)/w+
    \expect{\min \left\{ c(q_1+p_2-p_1-q_2)/w, Z\right\} }\nonumber\\
    &=-\tau\frac{q_1+p_2-p_1-q_2-1}{q_1+p_2-p_1-q_2}
    +\frac{1}{\sqrt{2 \pi}} \left( \int_{-\infty}^{\tau} x  e^{-\frac{x^2}{2}} dx + \int_{\tau}^{\infty} \tau e^{-\frac{x^2}{2}} dx \right) \nonumber\\
    &= -\tau \frac{q_1+p_2-p_1-q_2-1}{q_1+p_2-p_1-q_2}+\frac{1}{\sqrt{2 \pi}} \left( - e^{-\frac{\tau^2}{2}} +\tau\int_{\tau}^{\infty} e^{-\frac{x^2}{2}} dx \right).\label{eq:gap-Qtilde-VLP}
\end{align}
Taking the derivative with respect to $\tau$ gives
\begin{align}
    &\mspace{23mu}\frac{d}{d\tau}\left(-\tau \frac{q_1+p_2-p_1-q_2-1}{q_1+p_2-p_1-q_2}+\frac{1}{\sqrt{2 \pi}} \left( - e^{-\frac{\tau^2}{2}} +\tau\int_{\tau}^{\infty} e^{-\frac{x^2}{2}} dx \right) \right)\nonumber\\
    &=-\frac{q_1+p_2-p_1-q_2-1}{q_1+p_2-p_1-q_2}+\frac{1}{\sqrt{2 \pi}} \left(\tau e^{-\frac{\tau^2}{2}}-\tau e^{-\frac{\tau^2}{2}}+\int_{\tau}^{\infty} e^{-\frac{x^2}{2}} dx \right)\nonumber\\
    &=-\frac{q_1+p_2-p_1-q_2-1}{q_1+p_2-p_1-q_2}+\frac{1}{\sqrt{2 \pi}}\int_{\tau}^{\infty} e^{-\frac{x^2}{2}} dx.\label{eq:deriv-tau}
\end{align}
It is easy to see that there is a unique $\tau^*>0$ that makes the derivative equal to $0$ and maximizes $\tilde{Q}^N(c)$.
Let $c^*=w\tau^*/(q_1+p_2-p_1-q_2)$ be the corresponding maximizer.

Note that $c^*$ is a positive real number independent of $N$ and one can verify that $\beta^*_{\textrm{LP}}<0.5$ under our assumption of the parameters.
Thus, we have $0\le\beta^*_{\textrm{LP}}+\frac{c}{\sqrt{N}}\le 0.5$ for all large enough $N$.
This implies that in fact for all large enough $N$,
\begin{align*}
    \VNdiffopt &=\sup_{c\colon 0\le\beta^*_{\textrm{LP}}+\frac{c}{\sqrt{N}}\le 0.5}\tilde{Q}^N(c)
    =\tilde{Q}^N(c^*).
\end{align*}

With this, we derive an explicit form of $\VNdiffopt-\tilde{Q}^N(0)$ below.
Using calculations similar to those in \eqref{eq:gap-Qtilde-VLP}, we get
\begin{align*}
    &\sqrt{N}\left(\tilde{Q}^N(c^*)-\tilde{Q}^N(0)\right)\big/w\\
    &=-\tau^* \frac{q_1+p_2-p_1-q_2-1}{q_1+p_2-p_1-q_2}+\frac{1}{\sqrt{2\pi}}\left(1-e^{-\frac{(\tau^*)^2}{2}}+\tau^*\int_{\tau^*}^{\infty}e^{-\frac{x^2}{2}}dx\right)\\
    &=\frac{1}{\sqrt{2\pi}}\left(1-e^{-\frac{(\tau^*)^2}{2}}\right),
\end{align*}
which is a positive number independent of $N$.
Therefore, 
\begin{align*}
    \VNdiffopt - \tilde{Q}^N(0)=\Theta(1/\sqrt{N}).
\end{align*}

\paragraph{Remark on the optimal policy for the Gaussian stochastic system.}
Note that the maximizer $c^*$ gives the optimal policy for the Gaussian stochastic system, $\tilde{Y}^*_1(1,1)=\beta^*_{\textrm{LP}}+\frac{c^*}{\sqrt{N}}$. 
The maximizer $c^*$ can be obtained by solving for $\tau^*$, which is the solution to the equation that the derivative in \eqref{eq:deriv-tau} is equal to $0$.
This equation can be numerically solved.
For the example in Section~\ref{sec:example}, we plugged the specific values $p_1 = 0.2$, $p_2 = 0.7$, $q_1 = 0.9$, $q_2 = 0.25$ and got the numerical solution
$c^*=c^*_{\text{d}} = 0.3940$ (recorded with $4$ digits of precision).

\subsection{Proof of $\VLP - \VNopt = \Theta(1/\sqrt{N})$ in Theorem~\ref{thm:lower-bounds}}   \label{append:additinal-details-solving-example}

Recall that we showed that it suffices to prove that $\VLP - \VNdiffopt = \Theta(1 / \sqrt{N}).$
We have also showed that $\VLP=\beta^*_{\textrm{LP}}+0.5$ and that $\VNdiffopt=\tilde{Q}^N(c^*)$ for all large enough $N$.
Then using calculations similar to those in \eqref{eq:gap-Qtilde-VLP} gives
\begin{align*}
    &\mspace{23mu}
    \sqrt{N}(\VLP - \VNdiffopt)/w\\
    &=\sqrt{N}\left(\beta^*_{\textrm{LP}} + 0.5-\tilde{Q}^N(c)\right)\big/w\\
    &= \tau^* \frac{q_1+p_2-p_1-q_2-1}{q_1+p_2-p_1-q_2}+\frac{1}{\sqrt{2 \pi}} \left( e^{-\frac{(\tau^*)^2}{2}} -\tau^*\int_{\tau^*}^{\infty} e^{-\frac{x^2}{2}} dx \right)\\
    &=\frac{1}{\sqrt{2 \pi}} e^{-\frac{(\tau^*)^2}{2}},
\end{align*}
which is a positive constant independent of $N$.
Therefore, $\VLP - \VNdiffopt = \Theta(1 / \sqrt{N})$, which completes the proof.

\section{Proof of Theorem \ref{thm:improvement}}  \label{append:proof-of-performance-improvement}

\PerformanceImprovement*

Let us first recall the construction of the $1/\sqrt{N}$-scale SP in Section~\ref{sec:numerical-experiments}. For SP~\eqref{eq:problem-formulation-Gaussian}, we have introduced the following $1/\sqrt{N}$-scale vectors $(\bc_h, \bd_h)_{h \in \left\{ 1,2,\dots,H\right\}}$: 
\begin{align*}
{\tilde \bY}_h := & \ \by^*_h+\frac{\bc_h}{\sqrt{N}}   \\
\bd_{h+1} := & \ \proj{\sqrt{N}(\Delta^S - \bx^*_{h+1})}{\sum_{s,a} c_h(s,a)\pp_h(\cdot \mid s,a)+{\bZ_h}}  \\
\tilde{\bX}_{h+1} := & \ \bx^*_{h+1}+\frac{\bd_{h+1}}{\sqrt{N}} 
\end{align*}
Substituting this into Problem~\eqref{eq:problem-formulation-Gaussian}, we obtained
\begin{alignat}{2}  
    & \quad \underset{\pi_c\in \Npic}{\mathrm{max}} \quad \sum_{h=1}^{H}  \mathbb{E} \left[ \br_h \bc_h^\top \right] \nonumber \\
    \hbox{s.t.} & \quad \sum_{s} c_h(s, 1) = 0 \ \text{and} \ \bc_h(\cdot, 0) + \bc_h(\cdot, 1) = \bd_h, \quad 1 \le h \le H, \nonumber \\
    & \quad c_h(s, a) \geq - \sqrt{N} y_h^*(s,a), \quad 1 \le h \le H,   \nonumber  \\
    & \quad \bd_1 = \bzero, \quad \quad \bd_{h+1} = \proj{\sqrt{N}(\Delta^S - \bx^*_{h+1})}{\sum_{s,a} c_h(s,a)\pp_h(\cdot \mid s,a)+{\bZ_h}},  \quad 1 \le h \le H-1,  \nonumber
    \end{alignat} 
    where $\Npic$ is the set of policies that maps $\bd_h$ to $\bc_h$ such that if $\|\bd_h\|_\infty \leq z_h \Ndelta\sqrt{N}$, then $\|\bc_h\|_\infty \leq \kappa z_h \Ndelta\sqrt{N}$, and follows a predefined policy otherwise:
    \begin{align*}
    \Npic := \big\{ \pi_c: & \|\pi_c(\bd_h,h)\|_\infty\leq \kappa z_h \Ndelta \sqrt{N},  \ \hbox{ if }  \|\bd_h\|_\infty \leq z_h \Ndelta \sqrt{N}; \\
    & \hbox{follow a predefined policy},  \ \hbox{ if }  \|\bd_h\|_\infty > z_h \Ndelta \sqrt{N} \big\}.
    \end{align*}
Denote by
\begin{equation*}
  v^N_{\pi_c}(\bd_h, h) := \sum_{h'=h}^{H}  \mathbb{E} \left[ \br_{h'} \pi_c(\bd_{h'}, h')^\top \right] 
\end{equation*}
as the value function of policy $\pi_c$ in state $\bd_h$ and step $h$, and let $q^N_{\pi_c}(\bd_h,\bc_h,h)$ to be the corresponding Q-function. 
Note that there is a one-to-one correspondence between $\pidelta$ and $\Npic$: for policy $\pi \in  \pidelta$ and its corresponding $\pi_c \in \Npic$, their value functions satisfy
\begin{equation*}
  \tilde{V}^N_\pi(\bx, h)=\VLP(\bx^*_h,h)+\frac{v^N_{\pi_c}(\bd, h)}{\sqrt{N}},
\end{equation*}
where $\bx = \bx^*_h + \bd/\sqrt{N}$. Moreover, the actions are related via $\pi(\bx,h) = \by^*_h + \pi_c(\bd,h)/\sqrt{N}$.

Recall the definition of the policy class $\pisub$ as
\begin{align*}   
\pisub = \big\{ \pi:  \|\pi(\bx_h,h) - \by^*_h \|_\infty\leq \kappa \|\bx_h - \bx^*_h\|_\infty, \; \forall h \ge 1 \big\}. 
\end{align*}

Note that from the definition of $v^N_{\pi_c},$ we have 
\begin{align*}  
   \tilde{V}^N_{\tilde{\pi}^{N,*}}(\bx_{ini}, 1) = \VLP(\bx_{ini},1) +\frac{{v}^{{N}}_{\tilde{\pi}_c^{{N},*}}(\bzero, 1) }{\sqrt{N}},
\end{align*} and for $\pi\in\pisub$,
\begin{align*}  
    \tilde{V}^N_{\pi}(\x_{ini}, 1) = \VLP(\bx_{ini},1)+\frac{v^N_{\pi_c}(\bzero, 1)}{\sqrt{N}}. 
\end{align*}

Now given $\pi,$ we construct a modified policy $\pi'$ such that $\pi'(\bx, h)=\pi(\bx,h)$ for $\bx$ such that  $\|\bx-\bx^*_h\|_\infty \leq z_h \Ndelta$ and $\pi'(\bx, h)=\pip(\bx, h),$ where $\pip$ is the predetermined policy when defining $\pidelta$. We therefore have $\pi'\in\pidelta$ and $\pi$ and $\pi'$ coincide when their trajectories satisfy  $\|\tilde \bX^{\pi'}_h-\bx^*_h\|_\infty = \|\tilde \bX^{\pi}_h-\bx^*_h\|_\infty \leq z_h \Ndelta$ for all $h.$ Therefore, from Item~\ref{lem-high-proba-item-2} of Lemma~\ref{lem:high-proba-bound}, we have 
\begin{equation}  \label{eq:value-gap-1}
  v^N_{\pi_c}(\bzero, 1) - v^N_{\pi'_c}(\bzero, 1) \le  2 r_{\max}H\sqrt{N}\frac{\tilde{C}}{N^{\log N}}.
\end{equation}

Let ${\tilde{\pi}^{{N},*}_c}$ be an optimal policy in $\Npic$ for Problem~\eqref{eq:problem-formulation-Gaussian-new} that corresponds to the locally-SP-optimal policy $\tilde{\pi}^{N,*}$ for Problem~\eqref{eq:problem-formulation-Gaussian}. Since $\pi'(\bx_{ini})=\pi(\bx_{ini})= \by^*_1$ and by construction $\pi'\in\pidelta$, we have 
\begin{equation*}
  v^N_{\pi'_c}(\bzero, 1)\leq {q}^{N}_{{\tilde{\pi}_c^{{N},*}}}({\bf0},{\bf0}, 1),
\end{equation*}
which from \eqref{eq:value-gap-1} implies that
\begin{equation*}
  v^N_{\pi_c}(\bzero, 1)\leq {q}^{N}_{{\tilde{\pi}_c^{{N},*}}}({\bf0},{\bf0}, 1)+\frac{2 r_{\max}H\tilde{C}}{N^{\log N-0.5}}.
\end{equation*}
Therefore, we conclude that
\begin{align}  \label{eq:gap-condition-I}
   \tilde{V}^N_{\tilde{\pi}^{N,*}}(\bx_{ini}, 1)- \tilde{V}^N_{\pi}(\x_{ini}, 1) = \frac{{v}^{{N}}_{\tilde{\pi}_c^{{N},*}}(\bzero, 1) - v^N_{\pi_c}(\bzero, 1)}{\sqrt{N}}\geq \frac{\epsilon-\frac{2 r_{\max}H\tilde{C}}{N^{\log N-0.5}}}{\sqrt{N}}\geq  \frac{\epsilon }{2\sqrt{N}}, 
\end{align}  
where the last inequality holds for any $N$ such that 
\begin{equation*}
  N^{\log N-0.5}\geq \frac{4 r_{\max}H\tilde{C}}{\epsilon}.
\end{equation*}

We now invoke Lemma \ref{lem:ed}, which is proved for policy $\tilde{\pi}^{N,*}$, and the observation that the same result holds for any policy $\pi\in\pisub$, because by definition $\pi$ is a Lipschitz policy so its value function is Lipschitz. We then have 
\begin{equation}\label{eq:gap-condition-II}
  \abs{\tilde{V}^N_{\tilde{\pi}^{N,*}}(\x_{ini}, 1) - V_{\tilde{\pi}^{N,*}}^N(\xini, 1)} = \tON,
\end{equation}
as well as
\begin{equation}\label{eq:gap-condition-III}
  \abs{\tilde{V}^N_{\pi}(\x_{ini}, 1) - V_{\pi}^N(\xini, 1)} = \tON.
\end{equation}

The claim of the theorem is then established by combining Equations \eqref{eq:gap-condition-I}, \eqref{eq:gap-condition-II} and \eqref{eq:gap-condition-III}.

\section{More details of numerical experiments}   \label{append:numerics-and-practical}

This appendix is a collection of more details on various numerical experiments conducted to confirm the theoretical results of this paper. It is structured as follows: 
\begin{itemize}
    \item Appendix~\ref{append:proportion-degenerate-and-uniqueness} studies how likely an RMAB instance is degenerate or satisfies the Uniqueness Assumption~\ref{ass:opt-lp-distance}.  
    \item Appendix~\ref{append:parameters-RMAB} displays the parameters and implementation details of the RMAB examples used in Section \ref{sec:numerical-experiments}. 
\end{itemize}

All numerical experiments in this paper were conducted on a personal laptop equipped with a 13th Gen Intel(R) i9-13980HX CPU. We note that our experiments are based on solving reasonable size linear programs, for which well-established solution packages are available. We implement the solutions using the Python package \textbf{PuLP} and the \textbf{Gurobi LP solver}.

\subsection{Degeneracy and Uniqueness Assumption~\ref{ass:opt-lp-distance} in RMAB instances}   \label{append:proportion-degenerate-and-uniqueness}

We conducted a numerical study on how likely a randomly generated RMAB instance is degenerate (Definition~\ref{def:non-degenerate-condition}) and the corresponding LP~\eqref{eq:problem-formulation-LP} satisfies the Uniqueness Assumption~\ref{ass:opt-lp-distance}. 

In this experiment, an RMAB instance is generated as follows: Each parameter is sampled via i.i.d.\ $\exp(1)$ and normalized properly as in the initial condition $\xini$ and in the transition kernels $\pp_h$. We fix $\alpha=0.4$ and $H=5$. We considered two scenarios: ``fully-dense'' where all entries of a transition kernel are positive; ``half-sparse'' where half of the entries of each row of a transition kernel are $0$. We varied the number of sizes per arm $S=5,10,15,20$ and tested the degeneracy and the Uniqueness Assumption~\ref{ass:opt-lp-distance} over $10,000$ such RMAB instances, and recorded the numbers in percentage. The results are presented in Table~\ref{tab:proportion-degenerate-and-uniqueness}.

\begin{table}[H]
    \centering
    \caption{Proportion of RMAB instances satisfying degeneracy and Uniqueness Assumption~\ref{ass:opt-lp-distance}}
    \label{tab:proportion-degenerate-and-uniqueness}
    \begin{tabular}{@{\hskip 2pt}p{0.25\linewidth}@{\hskip 2pt}p{0.12\linewidth}@{\hskip 3pt}p{0.22\linewidth}@{\hskip 2pt}p{0.4\linewidth}@{\hskip 2pt}}
        \toprule
        \textbf{Transition Kernel} & \textbf{$S$} & \textbf{Degenerate} & \textbf{Satisfy Uniqueness Assumption~\ref{ass:opt-lp-distance}} \\
        \midrule
        Fully-dense & 5  & 11.2\%  & 100\% \\
        Fully-dense & 10 & 8.7\%   & 100\% \\
        Fully-dense & 15 & 6.1\%   & 100\% \\
        Fully-dense & 20 & 5.1\%   & 100\% \\
        \midrule
        Half-sparse & 5  & 51.3\%  & 100\% \\
        Half-sparse & 10 & 33.3\%  & 100\% \\
        Half-sparse & 15 & 28.1\%  & 100\% \\
        Half-sparse & 20 & 20.3\%  & 100\% \\
        \bottomrule
    \end{tabular}
\end{table}

We note that overall, degenerate RMABs are a significant proportion among all instances, especially in the half-sparse setting. In addition, all these randomly generated RMABs satisfy the Uniqueness Assumption~\ref{ass:opt-lp-distance}.

\subsection{Parameters and implementation details of the RMAB examples used in Section \ref{sec:numerical-experiments}}  \label{append:parameters-RMAB}

The RMAB examples used in Section \ref{sec:numerical-experiments} model a machine maintenance problem. 
In these RMAB examples, we set $H = 5$ and $\alpha = 0.4$. We consider arms representing machines each described by a $10$ states MDP constructed from two distinct machine types, each having $5$ states. The initial state of an arm determines its machine type. Consequently, the transition kernels are structured into two block matrices of size $5$. Such a block-structured kernel enables the modeling of multiple machine types under the assumption of homogeneous arms. In other words, the homogeneous RMAB model can be used for heterogeneous systems with a finite number of types. All numerical parameters provided below are recorded with $4$ digits of precision. Note that we can add a sufficiently large common constant to all rewards to offset them, making the values non-negative, as per our reward convention discussed in the main paper.

\begin{align*}
  & \pp(\cdot \mid \cdot, 0) = \\
& \begin{bmatrix}
0.5415 & 0.4585 & 0 & 0 & 0 & 0 & 0 & 0 & 0 & 0 \\
0.5471 & 0.2265 & 0.2265 & 0 & 0 & 0 & 0 & 0 & 0 & 0 \\
0.7067 & 0 & 0.1467 & 0.1467 & 0 & 0 & 0 & 0 & 0 & 0 \\
0.8578 & 0 & 0 & 0.0711 & 0.0711 & 0 & 0 & 0 & 0 & 0 \\
0.9214 & 0 & 0 & 0 & 0.0786 & 0 & 0 & 0 & 0 & 0 \\
0 & 0 & 0 & 0 & 0 & 0.6396 & 0.3604 & 0 & 0 & 0 \\
0 & 0 & 0 & 0 & 0 & 0.5694 & 0.2153 & 0.2153 & 0 & 0 \\
0 & 0 & 0 & 0 & 0 & 0.6453 & 0 & 0.1773 & 0.1773 & 0 \\
0 & 0 & 0 & 0 & 0 & 0.7007 & 0 & 0 & 0.1496 & 0.1496 \\
0 & 0 & 0 & 0 & 0 & 0.7097 & 0 & 0 & 0 & 0.2903
\end{bmatrix},
\end{align*}

\begin{align*}
  & \pp(\cdot \mid \cdot, 1) = \\
& \begin{bmatrix}
1 & 0 & 0 & 0 & 0 & 0 & 0 & 0 & 0 & 0 \\
0.7337 & 0.2663 & 0 & 0 & 0 & 0 & 0 & 0 & 0 & 0 \\
0.7265 & 0 & 0.2735 & 0 & 0 & 0 & 0 & 0 & 0 & 0 \\
0.6146 & 0 & 0 & 0.3854 & 0 & 0 & 0 & 0 & 0 & 0 \\
0.6054 & 0 & 0 & 0 & 0.3946 & 0 & 0 & 0 & 0 & 0 \\
0 & 0 & 0 & 0 & 0 & 1 & 0 & 0 & 0 & 0 \\
0 & 0 & 0 & 0 & 0 & 0.6037 & 0.3963 & 0 & 0 & 0 \\
0 & 0 & 0 & 0 & 0 & 0.6004 & 0 & 0.3996 & 0 & 0 \\
0 & 0 & 0 & 0 & 0 & 0.7263 & 0 & 0 & 0.2737 & 0 \\
0 & 0 & 0 & 0 & 0 & 0.6138 & 0 & 0 & 0 & 0.3862
\end{bmatrix},
\end{align*}

\[
\bx_{\text{ini}} = \begin{bmatrix}
0 & 0.5 & 0 & 0 & 0 & 0 & 0.5 & 0 & 0 & 0
\end{bmatrix},
\]

\[
\br(\cdot,0) = \begin{bmatrix}
0 & -5.4707 & -7.0669 & -8.5784 & -9.2141 & 0 & -5.6942 & -6.4534 & -7.0074 & -7.097
\end{bmatrix}.
\]

In the first set of experiments, we use 
\begin{align*}
  & \br(\cdot,1) = \\
  & \begin{bmatrix}
-1.9963 & -2.085 & -2.035 & -2.0661 & -1.9581 & -1.994 & -1.9647 & -2.2478 & -2.0468 & -2.2821
\end{bmatrix},
\end{align*}
which results in a fluid LP with a unique optimal solution.
In the second set of experiments, we use 
\[
\br(\cdot,1) = \begin{bmatrix}
-2 & -2 & -2 & -2 & -2 & -2 & -2 & -2 & -2 & -2
\end{bmatrix},
\]
which leads to a fluid LP with multiple optimal solutions.

The LP-based policy we compared with is the LP-update policy (also referred to as the LP policy with resolving) proposed in \citep{Brown2020IndexPA, gast2023linear}. These studies reported that LP-update is one of the best-performing LP-based policies for finite-horizon RMABs. As mentioned earlier, we used EDDP \citep{lan2022complexity} to obtain the SP-based policy. The predefined policy in Line~12 of Algorithm~\ref{algo:SP-based} is specified as follows: If, during the execution of the algorithm, the $N$-system state deviates significantly from the initially selected optimal LP solution — where ``significant'' is defined as any coordinate exceeding a threshold of $20$ as defined by the policy class $\pic,$  a new LP is resolved using the current initial state of the $N$-system and the remaining horizon, similar to the LP-update/resolving policy, and then the first action from the new LP is applied.

\end{document}